\documentclass[preprint,12pt]{elsarticle}
\usepackage{geometry}
\usepackage{amssymb}
\usepackage{amsmath}
\usepackage{amsthm}

\usepackage{mathrsfs}
\usepackage{tabularx}
\usepackage{graphicx}
\usepackage{subcaption}
\usepackage{caption}
\usepackage{enumitem}
\usepackage{algorithm}
\usepackage{algpseudocode}
\usepackage{array}
\usepackage{bbding}
\usepackage{longtable}
\usepackage{listings}
\usepackage{rotating,rotfloat} 
\usepackage{xcolor}
\usepackage{multirow, bigstrut}
\usepackage{tikz}
\usetikzlibrary{3d, positioning, shapes.geometric, arrows.meta, calc}
\usepackage{stmaryrd}
\usetikzlibrary{positioning}
\usepackage{arydshln}   
\usepackage{cleveref}
\usepackage{empheq}
\usepackage{algorithm}
\usepackage{algpseudocode}

\allowdisplaybreaks

\newtheorem{theorem}{Theorem}[section]
\newtheorem{lemma}{Lemma}[section]

\newtheorem{proposition}[theorem]{Proposition}

\newtheorem{remark}{Remark}
\newtheorem{assumption}{Assumption}

\newcommand{\bbm}{\begin{bmatrix}}
\newcommand{\ebm}{\end{bmatrix}}

\journal{Journal of Computational Physics}

\begin{document}

\begin{frontmatter}



\title{A Stabilized Numerical Framework for Necrotic Tumor Growth via Coupled Boundary Integral and Obstacle Solvers}


\author[1]{Yu Feng\corref{equal}} \ead{fengyu@gbu.edu.cn}

\author[2]{Shuo Ling\corref{equal}}
\ead{lingshuo1@sjtu.edu.cn}

\author[3]{Wenjun Ying\corref{equal}}
\ead{wying@sjtu.edu.cn}

\author[4]{Zhennan Zhou\corref{cor}}
\ead{zhouzhennan@westlake.edu.cn}

\cortext[equal]{These authors contributed equally to this work.}
\cortext[cor]{Corresponding author.}

\affiliation[1]{organization={Department of Mathematics, Great Bay University},
            city={Dongguan}, 
            postcode = {523000}, 
            state = {GuangDong},
            country = {P.R.China}}
\affiliation[2]{organization={School of Mathematical Sciences, Shanghai Jiao Tong University},
            addressline={Minhang},
            city={Shanghai},
            postcode={200240},
            country={P.R.China}}
\affiliation[3]{organization={School of Mathematical Sciences, MOE-LSC and Institute of Natural Sciences, Shanghai Jiao Tong University},
            addressline={Minhang},
            city={Shanghai},
            postcode={200240},
            country={P.R.China}}
\affiliation[4]{organization={Institute for Theoretical Sciences, Westlake University},
            addressline={Xihu},
            city={Hangzhou},
            postcode={310030},
            state={Zhejiang},
            country={P.R.China}}

\begin{abstract}
We present a robust computational framework for Hele-Shaw tumor growth with necrotic cores, a problem identified as the incompressible limit of the Porous Media Equation. Simulating this system presents a fundamental challenge: while the outer boundary evolves via advection, the inner necrotic interface is defined by an {obstacle problem} and {lacks an explicit advection structure}, causing standard schemes to fail. To address this, we introduce a stabilized predictor-corrector strategy that iteratively resolves the bidirectional coupling between the nutrient-pressure fields and the domain geometry, ensuring robust time-stepping for both the advection-driven outer surface and the obstacle-defined necrotic core. We establish rigorous convergence theory for the single-interface case and demonstrate the method's robustness in capturing the topological transition of necrotic core nucleation and complex geometric evolution.
\end{abstract}



\begin{keyword}
Tumor growth \sep Necrotic core \sep Hele-Shaw flow \sep Free boundary problem \sep Obstacle problem \sep Boundary integral method
\MSC 35R35 \sep 35Q92 \sep 65R20 \sep 65K15
\end{keyword}

\end{frontmatter}




\section{Introduction}
The mathematical modeling of tumor growth has evolved significantly over the past decades, serving as a powerful tool to understand the complex interplay between nutrient diffusion, cell proliferation, and tissue mechanics. Classical cell-density models, pioneered by Greenspan \cite{greenspan1972models} and mathematically formalized by Friedman et al. \cite{friedman2001symmetry,cui2001analysis,chen2003free,cui2006formation,friedman2007mathematical}, treat the tumor as an incompressible fluid flowing through a porous medium, governed by Hele-Shaw type equations. In recent years, the theoretical foundation of these models has been rigorously strengthened by establishing the link to the Porous Media Equation (PME). Seminal works by Perthame et al. \cite{perthame2014hele,perthame2014traveling} have shown that as the stiffness parameter of the constitutive law tends to infinity (the stiff limit), the PME converges to a free boundary problem of Hele-Shaw type. This connection not only validates the incompressible assumption but also provides a robust framework for analyzing the geometric motion of the tumor interface. 
Considerable progress has been made in recent years on porous-medium-based tumor growth models and their incompressible limits. Rigorous analytical studies have established the incompressible limit of porous-medium-based tumor growth models and their convergence to Hele–Shaw type free boundary problems \cite{kim2018porous,kim2018uniform,david2021free,liu2021existence,david2022convergence,guillen2022hele,he2023incompressible,david2023phenotypic,david2024incompressible,tong2025convergence}. Complementary analytical work has further investigated the regularity, stability, and instability properties of the resulting tumor interfaces \cite{jacobs2023tumor,kim2023tumor,collins2025regularity,feng2023tumor,feng2025nonsymmetric,liu2025three,he2024porous}. On the numerical side, structure-aware schemes capable of accurately capturing the incompressible limit and the associated free boundary dynamics have been developed \cite{liu2018accurate,liu2018analysis,david2022asymptotic,jiang2026efficient}, while recent efforts have also addressed parameter inference and model calibration for related models \cite{falco2023quantifying,feng2024unified,dkebiec2025lipschitz,liu2025data}.

While the single-interface dynamics of fully viable tumors are well-understood, the modeling of necrotic cores presents deeper physical and mathematical subtleties. Early variations of necrotic models sometimes suffered from physical inconsistencies, most notably the violation of positivity constraints, where the pressure could mathematically take negative values within the core—a physically unrealistic scenario for dead tissue. This issue was theoretically resolved by recent developments in the incompressible limit analysis, particularly by Kim et al. \cite{guillen2022hele} and Dou, Shen, and Zhou \cite{dou2024tumor}. Their analysis demonstrated that the correct limiting model for tumor growth with necrosis is governed by a variational inequality, specifically an {obstacle problem}. In this rigorous formulation, the pressure is constrained to be non-negative ($p \ge 0$), and the necrotic core is naturally defined as the coincident set $\{p=0\}$ enclosed by the inner free boundary.

This mathematical clarification, however, introduces a fundamental numerical challenge. Unlike the outer tumor boundary, which is governed by Darcy's law and evolves via a well-defined normal velocity (a hyperbolic advection character), the inner necrotic boundary lacks an explicit advection structure. As the free boundary of an obstacle problem, its motion is not driven by a local velocity field but is determined implicitly by the global redistribution of the pressure field to satisfy the unilateral constraint. Consequently, standard interface tracking methods (such as Level Set or Phase Field methods typically driven by transport equations) are ill-suited for directly capturing the necrotic interface without ambiguity. Explicit time-stepping schemes, which work well for the outer boundary, often induce severe numerical instabilities when applied to the inner boundary, as the implicit interface is highly sensitive to the temporal decoupling of the pressure field.

In this paper, we propose a comprehensive computational framework to address these challenges, capable of simulating the entire lifecycle of tumor growth. Our approach utilizes the Boundary Integral (BI) \cite{hsiao2021boundary, steinbach2008numerical, atkinson1997numerical, kress1989linear, sloan1992error, beale2001method, zhao2018computation} and Kernel-Free Boundary Integral (KFBI) \cite{ying2007kernel, ying2013kernel, ying2014kernel, xie2019fourth, xie2023fourth, tan2024gpu, LING2025108816} methods implemented on uniform Cartesian grids. 
The BI method is a mesh-free approach that is particularly effective for relatively simple elliptic problems, such as the Poisson equation and the modified Helmholtz equation, where the associated boundary integral equations admit concise formulations and the fundamental solutions are explicitly available. In contrast, for interface problems, the corresponding boundary integral formulations are considerably more involved and typically require repeated evaluation of boundary and volume integrals. To improve computational efficiency in this setting, we adopt the KFBI method, which is a Cartesian-grid-based approach that evaluates potential integrals by solving equivalent simple interface problems and employs fast algorithms, such as FFT-based solvers, for efficient implementation. This hybrid formulation combines the accuracy of integral equation methods with the efficiency and robustness of Cartesian-grid solvers. 
To resolve the numerical instability associated with the necrotic boundary, we further introduce a predictor--corrector strategy coupled with an Augmented Lagrangian method \cite{lions1980approximation, karkkainen2003augmented, gill2012primal, birgin2014practical} for obstacle problems \cite{glowinski2013numerical, tremolieres2011numerical}. This strategy stabilizes the evolution by iteratively coupling the pressure field update with the domain geometry, thereby effectively resolving the implicit dependence of the inner free boundary.

The main contributions of this work are threefold, balancing theoretical rigor with algorithmic innovation:
\begin{enumerate}
    \item \textbf{Rigorous error analysis for the single-interface model.} For the fully viable tumor model (without necrotic core), we establish a complete convergence theory for the proposed BI/KFBI scheme. We prove that the method achieves second-order accuracy for the potential problems and first-order convergence for the boundary velocity. This theoretical result provides a solid mathematical foundation for the accuracy of our overall framework.
    \item \textbf{Stabilized algorithm for the double-interface model.} For the more complex necrotic tumor model, we develop a robust predictor-corrector scheme that suppresses the numerical oscillations inherent in tracking the implicit inner boundary. Furthermore, we provide a systematic numerical treatment for capturing the topological transition of necrotic core nucleation, a feature often omitted in previous studies.
    \item \textbf{Numerical validation and exploration.} We demonstrate the effectiveness of our method through extensive numerical experiments. The scheme is validated against analytical solutions for radially symmetric cases, confirming the predicted convergence rates. We also investigate the shape relaxation dynamics and stability of the tumor under non-circular perturbations, providing physical insights into the growth regimes.
\end{enumerate}

The remainder of this paper is organized as follows. Section \ref{sec:model introduction} introduces the mathematical models for both the fully viable and necrotic tumor growth. Section \ref{Numerical Methods for the Models} details the numerical algorithms, including the error analysis for the single-interface case and the stabilization strategy for the necrotic case. Section \ref{Numerical Results} presents the numerical experiments, validating the method and exploring tumor dynamics. Finally, conclusions are drawn in Section \ref{Conclusion}.

\section{Model Introduction} \label{sec:model introduction}
We consider a continuum model for avascular tumor growth in which the tumor occupies a time-dependent domain $D(t)\subseteq\mathbb{R}^2$.
The model is derived as the Hele--Shaw limit of a reactive porous medium equation and couples pressure-driven cell motion with nutrient diffusion.
Tumor growth is regulated by a nutrient-dependent growth rate, while nutrient transport is modeled under an \emph{in vitro} setting with consumption inside the tumor and a prescribed boundary concentration.
Depending on the growth mechanism, the model may give rise to either a single moving boundary or additional internal interfaces.
The precise formulation of the cell density--pressure system and the nutrient model is presented in the following subsections.

\subsection{The Cell Density and Pressure Model}
\label{sec:cell_density_pressure}

To study tumor evolution driven by nutrient consumption and supply, we consider the tumor growth model introduced in \cite{dou2024tumor}, which is derived as the Hele--Shaw limit of a reactive porous medium equation (PME). Let $\rho(x,t)$ denote the density of tumor cells and $c(x,t)$ the concentration of nutrients. We define the tumor region $D(t)$ as the support of $\rho$, namely,
\begin{equation}
\label{eqn:set D}
D(t)=\left\{x\in\mathbb{R}^2 \mid \rho(x,t)>0\right\}.
\end{equation}

The pressure inside the tumor, denoted by $p(x,t)$, is related to the density through the Hele--Shaw monotone graph:
\begin{equation}
\label{eqn:HS graph}
p(\rho)=
\left\{
  \begin{array}{ll}
    0, & 0\leq \rho<1,\\[2pt]
    \in [0,\infty), & \rho=1.
  \end{array}
\right.
\end{equation}

As shown in \cite{dou2024tumor}, given a monotone increasing growth rate function $G(\cdot)$, at each time $t$ the pressure $p(x,t)$ solves an \emph{obstacle problem} of the form
\begin{equation}
\label{eq101010}
p=\arg\min_{u\in E}
\left\{
\int_{D(t)}
\left(\frac{1}{2}|\nabla u|^2 - G(c)u\right)\,dx
\right\},
\quad
E=\left\{u\in H_0^1(D(t)) \mid u\ge 0 \ \text{a.e.} \right\},
\end{equation}
where $H_0^1(D(t))$ denotes the Sobolev space of $H^1$ functions with zero trace on $\partial D(t)$. Meanwhile, the density $\rho$ satisfies the following equation in the weak sense:
\begin{equation}
\label{eqn:rho infty}
\frac{\partial \rho}{\partial t}
-\nabla\cdot(\rho\nabla p)
=\rho\,G(c).
\end{equation}

Depending on the sign of the growth rate function $G(c)$, the coupled system
\eqref{eq101010}--\eqref{eqn:rho infty} may exhibit qualitatively different behaviors.
When $G(c)$ remains nonnegative, the density stays fully saturated and the tumor evolution is governed by a single free boundary.
In contrast, if $G(c)$ changes sign, the solution may develop unsaturated regions, leading to the emergence of an additional internal free boundary. This naturally gives rise to two distinct regimes, which we describe below.

\paragraph{Patch regime} When $G(c)$ is nonnegative, the model preserves so-called \emph{patch solutions}, meaning that the density evolves as a characteristic function of the tumor region,
$\rho(x,t)=\chi_{D(t)}(x)$.
In this case, the solution of the obstacle problem \eqref{eq101010} coincides with the solution of the elliptic boundary value problem at each time $t$
\begin{subequations}
\label{eqn:possion}
\begin{alignat}{2}
-\Delta p &= G(c), &\quad &\text{in } D(t), \\
p &= 0, &\quad &\text{on } \partial D(t).
\end{alignat}
\end{subequations}
The evolution of the tumor boundary is governed by Darcy's law: the normal velocity $v(x)$ at $x\in\partial D(t)$ is given by
\begin{equation}
v(x)=-\nabla p(x)\cdot \mathbf{n}(x),
\end{equation}
where $\mathbf{n}(x)$ denotes the unit outward normal vector to $\partial D(t)$.

\paragraph{Non-patch regime} When $G(c)$ attains negative values, cell apoptosis occurs and the density may no longer retain the patch structure. In this case, an unsaturated region with $\rho<1$ may develop, which is referred to as a \emph{necrotic core}. This phenomenon typically arises in later stages of tumor growth, when nutrient depletion in the central region leads to cell death.

In this regime, the tumor region can be decomposed as
\begin{equation}
D(t)
=
\{x\in D(t)\mid p(x,t)>0\}
\cup
\{x\in D(t)\mid p(x,t)=0\}
=: S(t)\cup N(t),
\end{equation}
where $S(t)$ denotes the saturated region and $N(t)$ the necrotic core. By the Hele--Shaw graph relation \eqref{eqn:HS graph}, $\rho=1$ in $S(t)$ and $\rho<1$ in $N(t)$. Typically, $N(t)$ forms a simply connected region in the interior of $D(t)$, surrounded by $S(t)$. The boundary $\partial N(t)$ is referred to as the inner boundary, while $\partial D(t)$ is the outer boundary of the tumor. Consequently,
\[
\partial S(t)=\partial N(t)\cup \partial D(t),
\]
see Figure~\ref{fig1011} for an illustration.

From the classical theory of obstacle problems, $N(t)$ corresponds to the coincidence set where the pressure equals the obstacle value zero, while $S(t)$ is the non-coincidence set. Moreover, at each time $t$, given the tumor region $D(t)$, the pressure $p$ satisfies
\begin{subequations}
\label{eqn:PDE-OP}
\begin{alignat}{2}
-\Delta p &= G(c), &\quad &\text{in } S(t), \\
p &= 0, &\quad &\text{on } \partial S(t), \\
\nabla p &= 0, &\quad &\text{on } \partial N(t).
\end{alignat}
\end{subequations}
We emphasize that \eqref{eqn:PDE-OP} does not constitute a classical boundary value problem, since the domain $S(t)$, equivalently $N(t)$ or its boundary, are part of the unknowns. An analytical solution in the radially symmetric case is provided in \ref{appendix:exact_R0R1}, where, given the tumor radius $R_1$, the necrotic core radius $R_0$ is determined implicitly by an equation of the form $F(R_0,R_1)=0$.

Before concluding this subsection, we introduce two representative choices of the growth rate function $G(c)$ that will be used throughout this work.
\begin{remark}
\label{rmk10101}
We consider two representative forms of the growth rate function $G(c)$:
\begin{subequations}
\label{eq:G_forms}
\begin{alignat}{2}
\label{eqn:G1}
G(c) &= G_0\,c, \\
\label{eqn:G2}
G(c) &= G_0(c-\bar c).
\end{alignat}
\end{subequations}
In the first case, $G(c)$ remains positive throughout the tumor region, and thus no necrotic core forms, i.e., $N(t)=\emptyset$ and $S(t)=D(t)$ for all $t>0$. In contrast, in the second case, cell proliferation occurs only when the nutrient concentration exceeds the threshold $\bar c$, while apoptosis dominates otherwise. As a result, a necrotic core emerges once the tumor grows sufficiently large.
\end{remark}
These two growth laws give rise to fundamentally different free boundary structures and will serve as benchmark models for the numerical methods developed in the subsequent sections.

\subsection{The Nutrient Model}
\label{sec:Two nutrient models}

To model the nutrient dynamics, we consider an \emph{in vitro} setting, in which tumor cells consume nutrients inside the tumor region while the nutrient concentration is maintained at a prescribed background level $c_B$ on the outer boundary.
Accordingly, the nutrient concentration $c(x,t)$ satisfies
\begin{subequations}
\label{eqn:original_nutrient}
\begin{alignat}{2}
\tau\,\partial_t c - \Delta c + \Phi(\rho,c) &= 0,
&\qquad &\text{in } D(t), \\
c &= c_B,
&\qquad &\text{on } \partial D(t).
\end{alignat}
\end{subequations}
Here, $\tau$ denotes the characteristic time scale of nutrient dynamics, and $\Phi(\rho,c)$ represents the nutrient consumption rate depending on the local cell density.

To simplify the model, we make two standard assumptions.
First, since nutrient diffusion typically occurs on a much faster time scale than tumor growth, we neglect the time derivative by assuming $\tau \ll 1$ (see, e.g., \cite{greenspan1972models,adam2012survey}).
Second, we assume that nutrient consumption is spatially homogeneous within each region but differs between the saturated region $S(t)$ and the necrotic core $N(t)$.
Specifically, we set
\begin{equation}
\label{eqn:two_rates}
\Phi(\rho,c) =
\begin{cases}
\lambda c, & \text{when } \rho = 1, \\[2pt]
n_c \lambda c, & \text{when } \rho < 1,
\end{cases}
\end{equation}
where the constant $n_c<1$ reflects the reduced metabolic activity of cells in the necrotic core.

Under these assumptions, the nutrient model \eqref{eqn:original_nutrient} reduces to the following elliptic system:
\begin{subequations}
\label{eqn:nutrient_model}
\begin{alignat}{2}
-\Delta c + \lambda c &= 0,
&\qquad &\text{in } S(t), \\
-\Delta c + n_c \lambda c &= 0,
&\qquad &\text{in } N(t), \\
c &= c_B,
&\qquad &\text{on } \partial D(t),
\end{alignat}
\end{subequations}
with $c \in C^1(D(t))$.

We emphasize that the nutrient system \eqref{eqn:nutrient_model} is fully coupled with the pressure obstacle problem \eqref{eq101010}, since the regions $S(t)$ and $N(t)$ are themselves part of the unknowns.
In the following sections, we develop numerical methods for this coupled free boundary system under the growth laws introduced in Remark~\ref{rmk10101}.

\section{Numerical Methods for the Models} \label{Numerical Methods for the Models}
In this section, the numerical methods developed for solving the tumor growth models introduced in this study are presented. Building upon the numerical tools described in \ref{Numerical Solvers}, they are adapted and integrated to simulate the tumor growth models. The goal is to establish a robust and efficient computational framework capable of accurately capturing the dynamics of the model boundaries while ensuring numerical stability and consistency. All solvers employed in this section are described in detail in \ref{Numerical Solvers} and are designed based on Cartesian grids, which enables their seamless integration.

Moreover, in the following discussions, in numerical computations, a curve is generally represented using a set of control points, which can be used to construct a parametrized representation passing through them via cubic periodic spline interpolation. Conversely, given a parametrized curve, an appropriate choice of parameter values allows a corresponding set of control points to be determined. In practice, the evolution of the tumor boundary is realized through the evolution of the positions of these control points.

\subsection{The Tumor Model without Necrotic Core} \label{The in vitro model}

\subsubsection{Numerical Algorithm} \label{Numerical Algorithm}
In this section, the rate function $G(c)$ was chosen to be $G(c) = G_0 c$, as shown in \eqref{eqn:G1}. The nutrient model is taken to be the in vitro model \eqref{eqn:nutrient_model}. Note that $G(c)$ is nonnegative, the Euler--Lagrange equation associated with the minimization problem \eqref{eq101010} coincides with the Poisson equation combined with a zero Dirichlet boundary condition \eqref{eqn:possion}. By the maximum principle, the solution of this PDE is nonnegative within $D(t)$, and thus it automatically satisfies the constraint $u\ge0$ in the admissible set $E$ of \eqref{eq101010}. Consequently, the optimization problem can be equivalently solved through the PDE formulation. Additionally, as mentioned in Remark \ref{rmk10101}, $N(t)=\emptyset$ and $S(t)=D(t)$ for all $t>0$, thus the free boundary model can be written as follows:
\begin{equation} \label{eq0915_1}
\begin{aligned}
-\Delta p & = G_0 c, \quad && \text{in } D(t), \\
\left.p\right|_{\partial D} & = 0, \\
\left.v\right|_{\partial D} & = -\left.\nabla p \cdot \mathbf{n}\right|_{\partial D},
\end{aligned}
\end{equation}
coupled with the in vitro nutrient regime
\begin{equation} \label{eq0915_2}
\begin{aligned}
-\Delta c + \lambda c = 0, \quad && \text{in } D(t), \\
c = c_B, \quad && \text{on } \partial D(t).
\end{aligned}
\end{equation}

To efficiently solve the free boundary problem \eqref{eq0915_1}--\eqref{eq0915_2}, we employ a combination of boundary integral (BI) \cite{hsiao2021boundary, steinbach2008numerical, atkinson1997numerical, kress1989linear, sloan1992error, beale2001method, zhao2018computation} and Kernel-Free Boundary Integral (KFBI) \cite{ying2007kernel, ying2013kernel, ying2014kernel, xie2019fourth, xie2023fourth, tan2024gpu, LING2025108816} methods. These methods are particularly suitable for problems with elliptic PDEs and moving boundaries.

The solution procedure can be summarized as follows. First, the in vitro nutrient equation \eqref{eq0915_2} has a homogeneous source term and linear boundary conditions, which allows it to be solved using the standard indirect BI method. The BI formulation reduces the problem dimensionality by one, representing the solution entirely in terms of boundary densities, which simplifies the treatment of the evolving boundary and improves computational efficiency.

For the inhomogeneous Poisson equation \eqref{eq0915_1} satisfied by $p$, we adopt a two-step approach. We first compute a particular solution that satisfies the inhomogeneous source term using the KFBI method. This step involves evaluating a volume integral without explicitly constructing or integrating singular Green’s functions, leveraging fast algorithms such as the Fast Fourier Transform (FFT) for efficiency. Once the particular solution is obtained, the remaining homogeneous problem with zero source term is solved using the standard BI method. This decomposition into particular and homogeneous solutions allows us to take advantage of the strengths of both methods: KFBI efficiently handles the volume contribution from the source term, while BI accurately enforces boundary conditions on the moving interface. The key steps of the construction, including the decomposition into particular and homogeneous solutions and the boundary integral representation, are detailed in \ref{Poisson Equations}.

The numerical procedure employed to simulate the free boundary model governed by equations~\eqref{eq0915_1}--\eqref{eq0915_2} is presented in detail, including the discretization strategy and the boundary evolution algorithm.

\paragraph{Initialization}
\begin{itemize}
    \item Given model parameters $G_0>0$, $\lambda>0$, and $c_B>0$.
    \item Specify the initial tumor region $D(0)$ through a set of $N$ control points 
    $\{\mathbf{X}_k^0\}_{k=1}^N$, which represent the boundary $\partial D(0)$ in counterclockwise order. The control points are chosen to sample the boundary approximately uniformly.
    \item Choose a time step size $\Delta t>0$.
    \item Place a rectangular computational domain $\mathcal{B}$ that fully contains $D(0)$, and discretize $\mathcal{B}$ with a uniform Cartesian grid of $(I+1)\times (J+1)$ nodes, where the grid spacing is $\Delta x = (x_{\max}-x_{\min})/I$ and $\Delta y = (y_{\max}-y_{\min})/J$.
\end{itemize}

\paragraph{Time-stepping algorithm}
For each time step $n=0,1,2,\dots$, given the current set of control points 
$\{\mathbf{X}_k^n\}_{k=1}^N$ representing the boundary $\partial D(t^n)$ (with $t^n := n \Delta t$), perform the following steps:
\begin{enumerate}
    \item \textbf{Nutrient concentration:} Solve the nutrient equation \eqref{eq0915_2} in $D(t^n)$ using BI Solver in \ref{Modified Helmholtz Equations} to obtain $c^n(\mathbf{x})$.
    \item \textbf{Pressure field:} Substitute $c^n$ into the first two equations of \eqref{eq0915_1}, and solve for $p^n(\mathbf{x})$ in $D(t^n)$ using BI\&KFBI Solver in \ref{Poisson Equations}. Note that the computed $p^n$ is available only at the Cartesian grid points. By collecting several nearby grid values and reconstructing a local quadratic polynomial (e.g., via a least--squares fit), we can evaluate this polynomial at any $\mathbf{x} \in \overline{D}(t^n)$. In particular, this reconstruction provides approximations of 
    \(
    p^n(\mathbf{x}), \
    \partial_x p^n(\mathbf{x}), \
    \partial_y p^n(\mathbf{x})
    \)
    at boundary points.

    \item \textbf{Boundary update:} 
    \begin{enumerate}
        \item Compute the normal velocity 
        \[
        v_k^n = -\nabla p^n(\mathbf{X}_k^n)\cdot \mathbf{n}_k^n, \qquad k=1,\dots,N,
        \]
        and update the control points by
        \[
        \mathbf{X}_k^{n+1} = \mathbf{X}_k^n + \Delta t\, v_k^n \mathbf{n}_k^n.
        \]
        \item Construct a smooth closed parametric curve passing through $\{\mathbf{X}_k^{n+1}\}_{k=1}^N$ using periodic cubic spline interpolation.
        \item Redistribute the control points uniformly along this curve with respect to arc length to obtain a new set $\{\mathbf{X}_k^{n+1}\}_{k=1}^N$ for the next step.
    \end{enumerate}
\end{enumerate}

This procedure yields the evolution of the tumor region $D(t)$ over time.

\begin{remark} \label{rmk1020}
At the end of each time step, we perform a correction procedure to ensure that the control points are distributed as uniformly as possible. This helps propagate these control points to the next time step and obtain a nearly uniform parameterization of the tumor boundary by means of cubic periodic spline interpolation. Such a correction is necessary; otherwise, after several time steps of evolution, the control points may become overly concentrated or too sparse, leading to an inaccurate representation of the boundary.
\end{remark}

\subsubsection{Numerical Analysis} \label{Numerical Analysis}
This section demonstrates the error analysis of the numerical algorithm in Section \ref{Numerical Algorithm}. The temporal discretization adopted in the boundary evolution follows the standard forward Euler method, which is widely used for time-dependent free boundary problems due to its simplicity and ease of implementation.

\paragraph{Solvers for Elliptic Problems}
At each time step, the nutrient field $c$ is computed using a (at least) second–order Nyström discretization of the BI method. The pressure $p$ is obtained by combining a second–order KFBI method for computing a particular solution of the source term $G_0 c$ with a (at least) second–order Nyström discretization of the BI method for the remaining homogeneous Poisson problem. Detailed convergence analyses
supporting these conclusions can be found in \ref{Modified Helmholtz Equations}, Remark~\ref{rmk_convergence_BI}, and \ref{More Challenging Interface Problems of the Modified Helmholtz Type}, Remark~\ref{rmk_convergence_KFBI}.

\paragraph{Geometric and discretization assumptions}
For the purpose of the proof, assume that the exact boundary $\Gamma(t)$ is smooth for $t\in[0,T]$, and that the fields $c$ and $p$ admit sufficiently regular extensions up to the moving interface. These assumptions are standard for the numerical analysis of such systems, as we focus on the convergence of the scheme in the regime where the interface remains sufficiently regular. While the regularity of free boundaries is a complex theoretical topic, the smoothness of the source term and boundary conditions in \eqref{eq0915_1}-\eqref{eq0915_2} suggests such regularity is maintained for at least short time scales.

Let $h$ denote the spatial discretization parameter of the Cartesian grid. Note that the BI (and KFBI) methods require boundary discretization; the boundary is uniformly discretized into $M$ points according to its parametrization. The choice of $M$ is coupled to the discrete parameters $I$ and $J$ of the Cartesian grid such that $\frac{2\pi}{M}=\mathcal{O}(h)$. After this point, any quantity with superscript~$h$ refers to its corresponding numerical approximation.

\begin{lemma}[Second-order accuracy of the BI\&KFBI solver for the Poisson problem] \label{lem1}
Let $\Omega \subset \mathbb{R}^2$ be a bounded domain with sufficiently smooth boundary.
Consider the Poisson problem
\begin{equation}
\begin{aligned}
-\Delta u &= f, \quad && \text{in } \Omega,\\
u &= 0, \quad && \text{on } \partial\Omega,
\end{aligned}
\end{equation}
and let $u_h$ be the numerical solution obtained by the BI\&KFBI solver in \ref{Poisson Equations}. Then the numerical solution satisfies the uniform error estimate
\begin{equation}
\|u - u_h\|_{\infty} \le C h^2,
\end{equation}
where $C>0$ is a constant independent of $h$. Here, $u_h$ denotes the numerical solution defined on the Cartesian grid of the computational box $\mathcal{B}$ that contains the physical domain $\Omega$, while $u$ is understood as the exact solution evaluated at the same grid points. The discrete $\ell^\infty$-norm is defined by
\[
\|v\|_{\infty} := \max_{\mathbf{x}_i \in \Omega_h} |v(\mathbf{x}_i)|,
\]
where $\Omega_h$ denotes the set of Cartesian grid points lying in $\Omega$. This convention is adopted throughout the paper.
\end{lemma}

\begin{proof}[Proof sketch]
The solution is decomposed as $u = u_p + u_b$, where $u_p$ is a particular solution computed by the KFBI method and $u_b$ solves a homogeneous Laplace problem with boundary data $-u_p|_{\partial\Omega}$. The KFBI method yields a second-order accurate approximation of $u_p$ in the maximum norm. By stable polynomial interpolation, the boundary values of the particular solution are also approximated with second-order accuracy. The difference between the exact and approximate homogeneous solutions is controlled by the maximum principle for harmonic functions. Finally, the Nyström discretization of the boundary integral formulation for the homogeneous problem is at least second-order accurate. Combining these estimates yields the stated $O(h^2)$ uniform error bound. A detailed proof is provided in \ref{proof_lemma1}.
\end{proof}

\begin{lemma}[Second-order accuracy with perturbed source term]
\label{lem2}
Let $u$ be the solution of the Poisson problem
\begin{equation} \label{eqq1}
\begin{aligned}
-\Delta u &= f, \quad && \text{in } \Omega,\\
u &= 0, \quad && \text{on } \partial\Omega,
\end{aligned}
\end{equation}
and let $\widetilde{f}$ be a perturbed source term satisfying
\begin{equation} \label{eq1208_13}
\|f - \widetilde{f}\|_\infty \leq C_f h^2, \quad \text{for some}\ C_f > 0.
\end{equation}
Let $\widetilde{u}_h$ be the numerical solution of
\begin{equation} \label{eq1208_12}
\begin{aligned}
-\Delta \widetilde{u} & = \widetilde{f}, \quad && \text{in } \Omega, \\
\widetilde{u} & = 0, \quad &&\text{on } \partial \Omega,
\end{aligned}
\end{equation}
obtained by the BI\&KFBI solver in \ref{Poisson Equations}. Then the uniform error estimate
\begin{equation}
\|u - \widetilde{u}_h\|_{\infty} \le C h^2
\end{equation}
holds, where $C > 0$ is independent of $h$.
\end{lemma}

\begin{proof}[Proof sketch]
Decompose the error as
\[
u - \widetilde{u}_h = (u - \widetilde{u}) + (\widetilde{u} - \widetilde{u}_h).
\]
The first term $u - \widetilde{u}$ is bounded by the elliptic regularity of Poisson equations with zero Dirichlet boundary condition. The second term \(\widetilde{u} - \widetilde{u}_h\) is controlled by the second order accuracy of the BI\&KFBI solver, as established in Lemma \ref{lem1}. Combining the two contributions yields the desired $O(h^2)$ uniform error bound. The detailed derivation is provided in \ref{proof_lemma2}.
\end{proof}

\begin{proposition}[Gradient and velocity error] \label{prop1}
Let $\Omega\subset\mathbb{R}^2$ be a bounded domain with smooth boundary,
and let $p$ be the exact pressure field.
Assume that a numerical solver provides approximations $p_h$ at Cartesian
grid points such that
\[
\|p - p_h\|_{\infty} \le C_p h^2.
\]

At each boundary point $\mathbf{x} \in \partial \Omega$, a local quadratic polynomial approximation of $p$ is reconstructed in a least--squares sense from nearby grid values of $p_h$, and the reconstructed gradient $\nabla p_{h, \mathrm{rec}}(\mathbf{x})$ is obtained by evaluating the gradient of this polynomial at $\mathbf{x}$. Assume that the associated reconstruction stencil is uniformly nondegenerate. Then the reconstructed gradient satisfies the uniform error estimate
\begin{equation}
\|\nabla p - \nabla p_{h,\mathrm{rec}}\|_{\infty}
= \mathcal{O}(h),
\end{equation}
and consequently, the computed normal velocity
\[
v_h = -\nabla p_{h,\mathrm{rec}}\cdot\mathbf{n}
\]
satisfies
\begin{equation}
\|v - v_h\|_{\infty}
= \mathcal{O}(h),
\end{equation}
where $v = -\nabla p\cdot\mathbf{n}$.
\end{proposition}

\begin{remark}
To avoid excessive length in the statement of Proposition \ref{prop1} above, certain details that would make the statement fully rigorous have been omitted; the complete version and its proof are provided in \ref{appendix_prop1}.
\end{remark}

Before presenting the final error estimate for the boundary evolution, we first make two assumptions on the free boundary problem \eqref{eq0915_1}--\eqref{eq0915_2}.

\begin{assumption} \label{as1}
Assume that the exact boundary evolution of system
\eqref{eq0915_1}--\eqref{eq0915_2} is Lipschitz continuous with respect to boundary perturbations. That is, there exists a constant $L_v>0$ such that
for any two sufficiently close boundaries $\partial D(t)$ and $\widetilde{\partial D(t)}$ with parametric representations $\mathbf{X}$ and $\widetilde{\mathbf{X}}$, respectively,
\begin{equation} \label{eq1208_21}
\|v(\mathbf{X}) - v(\widetilde{\mathbf{X}})\|_\infty
\le L_v \|\mathbf{X} - \widetilde{\mathbf{X}}\|_\infty.
\end{equation}
Here, $v(\mathbf{X})$ denotes the normal velocity at time $t$ induced by system \eqref{eq0915_1}--\eqref{eq0915_2} on the boundary $\partial D(t)$,
which is parametrized by $\mathbf{X}$, and $v(\mathbf{X})$ can be viewed as a function from $\mathbb{S}^1$ to $\mathbb{R}$. The notation $\|\cdot\|_\infty$ denotes the supremum norm over $\mathbb{S}^1$, applied to functions, or componentwise to vector-valued functions.
\end{assumption}

\begin{assumption} \label{as2}
Assume that the exact boundary $\partial D(t)$ is sufficiently smooth, and that the unit outward normal vector depends Lipschitz continuously
on boundary perturbations. More precisely, there exists a constant $L_n > 0$ such that for any two
sufficiently close boundaries $\partial D(t)$ and
$\widetilde{\partial D(t)}$, with parametric representations $\mathbf{X}$ and $\widetilde{\mathbf{X}}$, respectively, the corresponding unit normal vectors $\mathbf{n}$ and $\widetilde{\mathbf{n}}$ satisfy
\begin{equation} \label{eq_as2}
\|\mathbf{n}(\mathbf{X}) - \mathbf{n}(\widetilde{\mathbf{X}})\|_\infty
\le L_n \|\mathbf{X} - \widetilde{\mathbf{X}}\|_\infty.
\end{equation}
\end{assumption}

We are now in a position to present the error estimate for the boundary evolution and to provide its proof by invoking the previously established lemmas and proposition.

\begin{theorem} \label{thm1}
Let $\Gamma_h^n$ denote the numerical boundary (from the algorithm in Section \ref{Numerical Algorithm}) at time $t^n = n\Delta t$, represented by the set of control points $\{\mathbf{X}^n_{k,h}\}_{k=1}^M$, and let $\Gamma^n$ denote the exact boundary, represented by $\{\mathbf{X}^n_k\}_{k=1}^M$. Define the boundary error at step $n$ as
\[
\mathbf{e}^n_k := \mathbf{X}^n_k - \mathbf{X}^n_{k,h},
\]   
and error vector 
$$\mathbf{e}^n := (\mathbf{e}_1^n, 
\mathbf{e}_2^n, ..., \mathbf{e}_M^n).$$ 
Then the error estimate for the boundary evolution holds:
\begin{equation} \label{eq1208_26}
\|\mathbf{e}^n\|_\infty = \mathcal{O}(h) + \mathcal{O}(\Delta t), \quad \text{for } t^n \in [0,T].
\end{equation}
Here and in the sequel, the norm $\|\mathbf{e}^n\|_\infty$ is defined as
\[
\|\mathbf{e}^n\|_\infty := \max_{1 \le k \le M} |\mathbf{e}^n_k|_\infty,
\]
where $|\cdot|_\infty$ denotes the supremum norm in $\mathbb{R}^2$.
\end{theorem}

\begin{proof}
The boundary is updated using the forward Euler scheme
\begin{equation*}
\mathbf{X}^{n+1}_k = \mathbf{X}^n_k + \Delta t \, v(\mathbf{X}^n_k) \mathbf{n}_k + \mathcal{O}(\Delta t^2), \qquad
\mathbf{X}^{n+1}_{k,h} = \mathbf{X}^n_{k,h} + \Delta t \, v_h(\mathbf{X}^n_{k,h}) \mathbf{n}_{k,h},
\end{equation*}
where $v$, $v_h$ denote the exact and numerical normal velocities, respectively, and $\mathbf{n}_k$, $\mathbf{n}_{k,h}$ denote the exact and numerical unit normal vectors. Subtracting the numerical update from the exact update gives
\[
\mathbf{e}^{n+1}_k = \mathbf{e}^n_k + \Delta t \, \bigl(v(\mathbf{X}^n_k)\mathbf{n}_k - v_h(\mathbf{X}^n_{k,h})\mathbf{n}_{k,h}\bigr) + \mathcal{O}(\Delta t^2).
\]
Adding and subtracting $v(\mathbf{X}^n_{k,h})\mathbf{n}_{k,h}$ yields
\begin{equation} \label{eq0203_1}
\mathbf{e}^{n+1}_k = \mathbf{e}^n_k + \Delta t \, \bigl(v(\mathbf{X}^n_k)\mathbf{n}_k - v(\mathbf{X}^n_{k,h})\mathbf{n}_{k,h}\bigr)
+ \Delta t \, \bigl(v(\mathbf{X}^n_{k,h})\mathbf{n}_{k,h} - v_h(\mathbf{X}^n_{k,h})\mathbf{n}_{k,h}\bigr) + \mathcal{O}(\Delta t^2).
\end{equation}

To estimate the right hand side of equation \eqref{eq0203_1}, we consider firstly that
$$
v(\mathbf{X}^n_k) \mathbf{n}_k 
- v(\mathbf{X}^n_{k,h}) \mathbf{n}_{k,h} 
= (v(\mathbf{X}^n_k) - v(\mathbf{X}^n_{k,h})) \mathbf{n}_{k} 
+ v(\mathbf{X}^n_{k,h}) (\mathbf{n}_{k}
- \mathbf{n}_{k,h}).
$$
Assumption \ref{as1} implies that 
$$
|v(\mathbf{X}^n_k) - v(\mathbf{X}^n_{k,h})| \leq
L_v \| \mathbf{X}^n - \mathbf{X}^n_h \|_\infty,
$$
and Assumption \ref{as2} implies that
$$
| \mathbf{n}_{k} - \mathbf{n}_{k,h} |_\infty \leq
L_n \| \mathbf{X}^n - \mathbf{X}^n_h \|_\infty.
$$
Note that in Assumptions \ref{as1} and \ref{as2}, the notation $\|\cdot\|_\infty$ denotes the supremum norm for (vector-valued) functions defined on the boundary, whereas in the above inequalities it is naturally interpreted as the discrete $\ell^\infty$ norm over the control points. Additionally, we assume that the speed $v$ of the system \eqref{eq0915_1}--\eqref{eq0915_2} within the time interval $[0, T]$ has an upper bound $B_v$. Based on the above estimations, 
$$
|v(\mathbf{X}^n_k) \mathbf{n}_k 
- v(\mathbf{X}^n_{k,h}) \mathbf{n}_{k,h}|
\leq
(L_v + B_v L_n) \| \mathbf{X}^n - \mathbf{X}^n_h \|_\infty,
$$
where $|\mathbf{n}_k|_\infty \leq 1$ is used. Then we obtain from \eqref{eq0203_1} that
\begin{equation*} \|\mathbf{e}^{n+1}\|_\infty \le (1 + L \Delta t) \|\mathbf{e}^n\|_\infty + \Delta t \, \|v - v_h\|_\infty + \mathcal{O}(\Delta t^2), \end{equation*}
where $|\mathbf{n}_{k, h}|_\infty \leq 1$ is used and $L := L_v + B_v L_n$.

Recalling the computation of $v_h$, we first note that the numerical solution of $c$ is obtained by a standard Nyström discretization of the boundary integral method and is second-order accurate. Consequently, Lemma \ref{lem2} implies that
the numerical solution of $p$ is also second-order accurate, which satisfies the requirement of Proposition \ref{prop1}. Therefore, by Proposition \ref{prop1}, we obtain
\[
\|v - v_h\|_\infty = \mathcal{O}(h).
\]
Then we have 
\begin{equation} \label{eq1208_23}
\|\mathbf{e}^{n}\|_\infty 
\le (1 + L \Delta t) \|\mathbf{e}^{n - 1}\|_\infty + C_1 \Delta t \cdot h + C_2 \Delta t^2, \quad \exists C_1, C_2 > 0.
\end{equation}
Iterating \eqref{eq1208_23} recursively, we obtain
\begin{equation*}
\|\mathbf{e}^n\|_\infty \le (1 + L \Delta t)^n \|\mathbf{e}^0\|_\infty + \sum_{j=0}^{n-1} (1 + L \Delta t)^{j} \left( C_1 \Delta t \cdot h + C_2 \Delta t^2 \right),
\end{equation*}
where $\|\mathbf{e}^0\|_\infty = 0$ if the initial boundary is represented exactly. Summing the geometric series and using the inequality $(1 + L \Delta t)^n \le e^{L t^n}$ gives
\begin{align*}
\|\mathbf{e}^n\|_\infty 
&\le \left( C_1 h + C_2 \Delta t \right) \Delta t \sum_{j=0}^{n-1} (1 + L \Delta t)^{j} \\[2mm]
&\le \frac{C_1 h + C_2 \Delta t}{L} \left( e^{L t^n} - 1 \right) \\[1mm]
&\le \frac{e^{L T} - 1}{L} \left( C_1 h + C_2 \Delta t \right),
\end{align*}
which implies the final estimate \eqref{eq1208_26}.
\end{proof}


\begin{remark} \label{rmk:redistribution_error}
In the numerical implementation, periodic cubic spline interpolation is employed to reconstruct the parametric representation of the boundary from the discrete control points, followed by a redistribution process to maintain uniform arc-length parameterization. It is important to note that since the boundaries are assumed to be sufficiently smooth, the geometric errors incurred by spline interpolation and re-parameterization are of a higher order relative to the $\mathcal{O}(h)$ velocity error. Consequently, these operations do not degrade the overall convergence rates established in Theorem~\ref{thm1}.
\end{remark}

\subsection{The Tumor Model with Necrotic Core} \label{The in vitro Model with Necrotic Core}
In this section, the rate function $G(c)$ is chosen to be $G(c) = G_0 (c - \bar{c})$, as shown in \eqref{eqn:G2}. The nutrient model is taken to be the in vitro model \eqref{eqn:nutrient_model}. For the sake of completeness of exposition, we restate some concepts already introduced in Section \ref{sec:model introduction}. We
divide the tumor domain $D(t)$ into two typical regions according to cell density: the necrotic core $N(t)$, which is
a single connected region where the density is strictly less than 1; and the saturated region $S(t)$ where the density reaches the maximum value 1. Mathematically,
\begin{align*}
S(t) &:= \{x \in D(t) \mid \rho(t,x) = 1\},\\
N(t) &:=D(t)\setminus S(t)= \{x \in D(t) \mid 0 \le \rho(t,x) < 1\}.
\end{align*}
Two types of boundaries are defined:
\begin{align*}
\Gamma_0(t) &:= \partial N(t) \quad \text{(boundary of the necrotic core)},\\
\Gamma_1(t) &:= \partial D(t) \quad \text{(tumor boundary)}.    
\end{align*}

For clarity of the domains and the inner and outer boundaries, a schematic illustration is provided in Figure~\ref{fig1011}. 
\begin{figure}[htbp]
    \centering
    \includegraphics[width=0.5\linewidth]{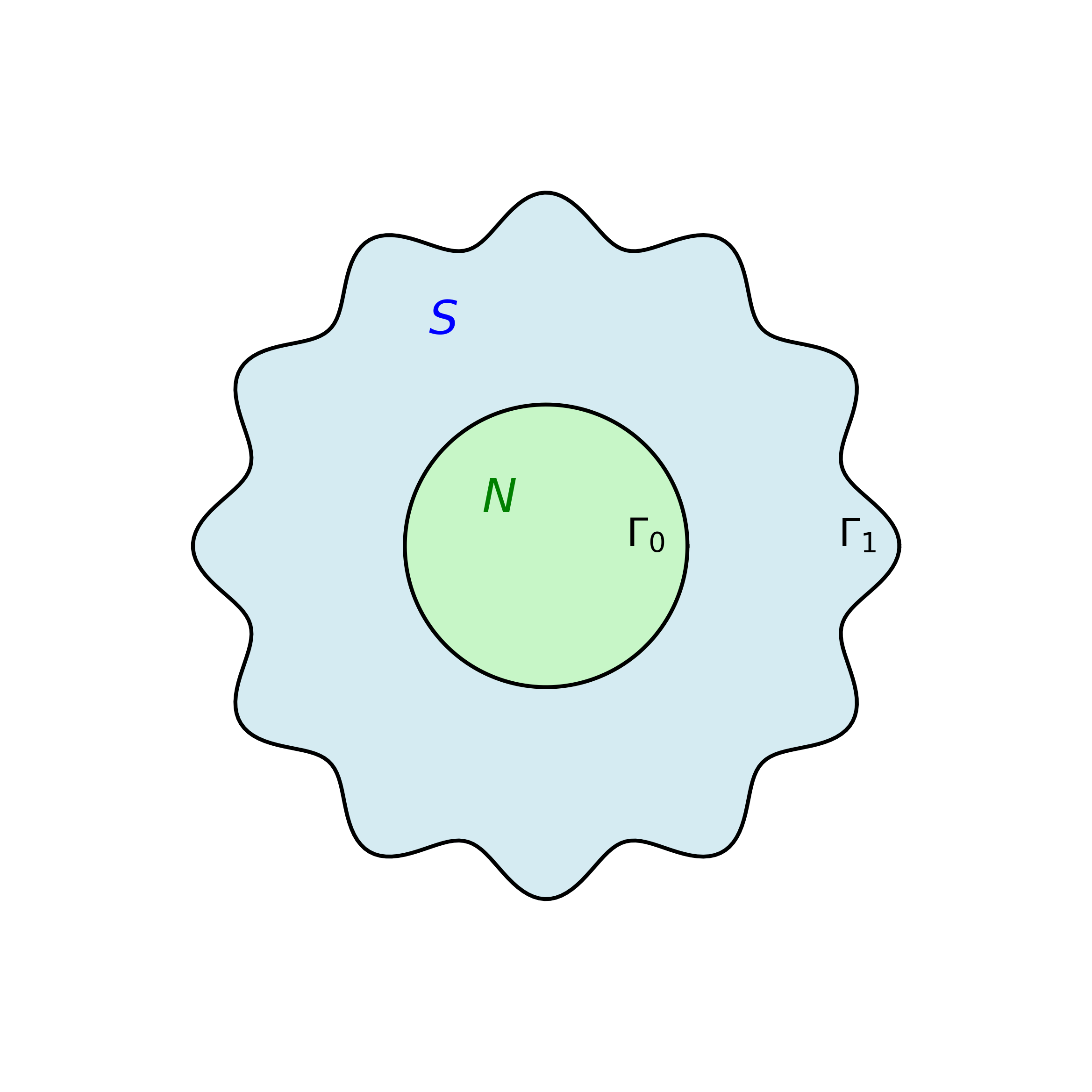}
    \caption{Schematic plot for the free boundary model.}
    \label{fig1011}
\end{figure}

The free boundary model can be written as follows:
\begin{align}
&p =\arg \min_{u \in E}\left(\int_{D(t)}\left(\frac{1}{2}|\nabla u|^2 - G_0(c-\bar{c}) u\right) dx\right), \quad  
E=\left\{u \in H_0^1(D(t)) : u \geq 0\ \text{a.e.}\right\},
\label{eq0916_1} \\
&\qquad \qquad\qquad \qquad\qquad \qquad\qquad \qquad\left.v\right|_{\partial D} = -\left.\nabla p \cdot \mathbf{n}\right|_{\partial D},
\label{eq0916_2}
\end{align}
coupled with the nutrient diffusion equations
\begin{equation} \label{eq0916_3}
\begin{aligned}
-\Delta c + \lambda n_c c &= 0, && \text{in } N(t), \\
-\Delta c + \lambda c &= 0, && \text{in } S(t), \\
[c] &= 0, && \text{on } \Gamma_0(t), \\
\left[\frac{\partial c}{\partial \mathbf{n}}\right] &= 0, && \text{on } \Gamma_0(t), \\
c(x) &= c_B, && \text{on } \Gamma_1(t).
\end{aligned}
\end{equation}

\subsubsection{Construction of Numerical Algorithm}

The simultaneous evolution of the outer and inner boundaries presents a fundamental numerical challenge due to their distinct mathematical natures. The outer boundary $\Gamma_1(t)$ is a classical sharp interface governed by Darcy's law; it possesses a well-defined normal velocity $v_n = -\nabla p \cdot \mathbf{n}$ and evolves via standard advection. In stark contrast, the inner boundary $\Gamma_0(t)$ is the free boundary of the coincident set associated with the variational inequality \eqref{eq0916_1}. Mathematically, it is defined as the interface separating the region where $p > 0$ from the region where $p = 0$. Crucially, this boundary lacks an explicit advection structure. Its motion is not driven by a local velocity field but is determined implicitly by the global redistribution of the pressure field to satisfy the unilateral constraint. Consequently, standard interface tracking methods (such as Level Set or Phase Field methods typically driven by transport equations) are ill-suited for directly capturing $\Gamma_0(t)$ without ambiguity.

To resolve the necrotic core boundary accurately at any given time instance, we rely on the precise identification of the coincident set $\{x \in D(t) \mid p(x)=0\}$. We employ the Augmented Lagrangian method, implemented via the Primal-Dual Active Set (PDAS) algorithm. As detailed in \ref{Obstacle Problems}, the PDAS algorithm leverages the semi-smooth Newton property of the problem, allowing it to identify the exact active set (and thus the discrete free boundary) in a finite number of iterations. This ensures that, for a fixed domain and source term, the location of the inner boundary is determined with high precision and without the smearing effects typical of regularization methods.

However, possessing a robust static solver is insufficient for the dynamic simulation. A critical issue arises when coupling the boundary evolution with the time-stepping scheme. Since $\Gamma_0(t)$ is inherently implicit, its position at the next time step $t^{n+1}$ depends on the pressure field $p^{n+1}$, which in turn depends on the domain geometry at $t^{n+1}$. This circular dependency creates a severe stability bottleneck. A naive approach, akin to a standard forward Euler scheme, would evolve the outer boundary explicitly and update the inner boundary based solely on the current pressure snapshot:
\begin{equation}
\Gamma_1^{n+1} = \Gamma_1^{n} + \Delta t\ v^n \mathbf{n}_{\Gamma_1}^n, \quad
\Gamma_0^{n+1} = \partial \{x \in  D(t^n) \mid p^n(x) = 0\}.
\end{equation}
While this decouples the geometric update from the field solve, it fails to capture the dynamic reshaping of the pressure profile. Our numerical experiments (see Section \ref{Numerical Results for the in vitro Model with Necrotic Core}, Example 4) reveal that this explicit strategy leads to significant lagging errors and spurious oscillations in the radius of the necrotic core, causing the simulation to break down.

To overcome this instability, a stabilization strategy is required to resolve the strong coupling between the pressure field and the domain geometry. We propose a predictor-corrector strategy that incorporates information from the future time step. Instead of relying on a single snapshot, we first generate a tentative inner boundary (prediction) to approximate the geometry at $t^{n+1}$, update the pressure field based on this prediction, and then refine the boundary position (correction). The final boundary location is obtained by averaging these estimates. This implicit coupling effectively suppresses numerical oscillations and ensures that the evolution of the coincident set remains consistent with the diffusion dynamics. The detailed workflow is illustrated in Figure \ref{fig0114}.

\begin{figure}[htbp]
\centering
\begin{tikzpicture}[>=Stealth, node distance = 1.5cm and 3cm, every node/.style={font=\small}, box/.style={draw, rounded corners, align=center, inner sep=6pt}]

\node[box] (A) {$\Gamma_1^n$ \\ $\Gamma_0^n$};
\node[box, right=of A] (B) {$\Gamma_1^n$ unchanged \\
$\Gamma_0^{n,*}=\partial\{p^{n,*}=0\}$};

\node[box, right=of B] (C) {
$\Gamma_1^{n+1} = \Gamma_1^n + \Delta t (-\nabla p^{n,**}) \mathbf{n}$ \\
$\Gamma_0^{n,*}$ unchanged};
\node[box, below=of C] (D) {
$\Gamma_1^{n+1}$ unchanged\\
$\Gamma_0^{n,***}=\partial\{p^{n,***}=0\}$};

\node[box] (E) at ($(A.center |- D.center) + (1.6, 0)$) {
$\Gamma_1^{n+1}$ unchanged \\
$\Gamma_0^{n+1} = \tfrac12 (\Gamma_0^{n,*}+\Gamma_0^{n,***})$
};
\draw[->] (A) -- node[above] {\footnotesize solve $c,p$} node[below] {\footnotesize obtain $c^{n, *},p^{n, *}$} (B);
\draw[->] (B) -- node[above] {\footnotesize solve $c,p$} node[below] {\footnotesize obtain $c^{n, **},p^{n, **}$} (C);
\draw[->] (C) -- node[left] {\footnotesize solve $c,p$} node[right] {\footnotesize to obtain $c^{n, ***},p^{n, ***}$} (D);
\draw[->] (D) -- node[above] {\footnotesize average for inner boundary} node[below] {\footnotesize obtain $\Gamma_0^{n+1}$} (E);
\end{tikzpicture}
\caption{Predictor--corrector strategy for updating the inner and outer boundaries.}
\label{fig0114}
\end{figure}
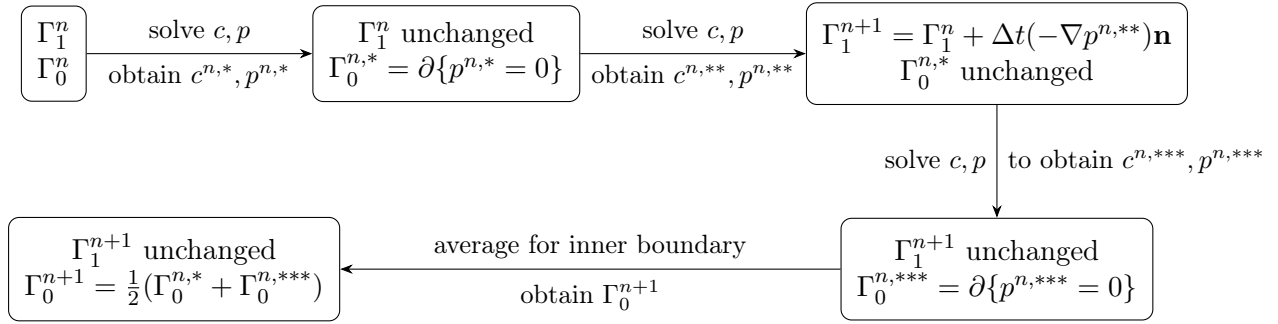

In this way, the inner boundary evolution is stabilized and remains consistent with the dynamics of the outer boundary. The nutrient diffusion equation \eqref{eq0916_3} with time-dependent boundaries $\Gamma_0(t)$ and $\Gamma_1(t)$ can be solved using the KFBI method, a type of boundary integral method that evaluates integral potentials without the need to explicitly construct singular Green’s functions. By exploiting fast algorithms such as the Fast Fourier Transform (FFT), the KFBI method efficiently computes the nutrient field on the Cartesian grid while maintaining essentially second-order spatial accuracy \cite{ying2014kernel, xie2023fourth}. This makes it particularly suitable for interface problems with geometrically complex and evolving domains, without requiring body-fitted meshes. The pressure obstacle problem \eqref{eq0916_1} is discretized on the same Cartesian grid and solved via the augmented Lagrangian active set method  \cite{karkkainen2003augmented}, known for its robustness and fast convergence. Formulated on the same grid, these two solvers are naturally coupled, yielding an efficient and flexible framework for simulating the evolution of the tumor free boundary.

\subsubsection{Implementation of Numerical Algorithm}
The numerical procedure employed to simulate the free boundary model governed by equations~\eqref{eq0916_1}--\eqref{eq0916_3} is now presented in detail, including the discretization strategy and the boundary evolution algorithm. 

\paragraph{Initialization}
\begin{itemize}
    \item Prescribe the model parameters $G_0>0$, $\lambda>0$, $c_B>0$, $0 < \bar{c} < c_B$, and $0 < n_c < 1$.
    \item Specify the initial inner boundary $\Gamma_0(0)$ of the tumor by a set of control points $\{\mathbf{X}_{0, k}^0\}_{k=1}^N$ arranged in counterclockwise order.
    \item Specify the initial outside boundary $\Gamma_1(0)$ of the tumor by a set of control points $\{\mathbf{X}_{1, k}^0\}_{k=1}^N$ arranged in counterclockwise order.
    \item Choose a time step size $\Delta t>0$.
    \item Define a rectangular computational domain $\mathcal{B}$ that fully contains $D(0)$, and discretize $\mathcal{B}$ with a uniform Cartesian grid of $(I+1)\times (J+1)$ nodes.
\end{itemize}

\paragraph{Time-stepping procedure}
Let $t^n = n\Delta t$. Suppose that the inner and outer boundaries at time $t^n$ are represented by the control points
$\{\mathbf{X}_{0,k}^n\}_{k=1}^N$ and $\{\mathbf{X}_{1,k}^n\}_{k=1}^N$, respectively.
Each time step consists of the following components.

\emph{Step 1: Initial Field Evaluation.}
Given the boundaries $\Gamma_0(t^n)$ and $\Gamma_1(t^n)$ at time $t^n$, the nutrient equation \eqref{eq0916_3} is first solved using the KFBI method in \ref{More Challenging Interface Problems of the Modified Helmholtz Type}, yielding $c^{n,*}$.  
The resulting nutrient field is then substituted into the obstacle problem \eqref{eq0916_1}, which is solved by the PDAS algorithm (Solver~\ref{alg0914} in \ref{Obstacle Problems}) to obtain $p^{n,*}$.

\emph{Step 2: Inner Boundary Predictor.}
An intermediate inner boundary is extracted from the zero-level set of the pressure,
\[
N(t^n)^* := \{ x \in D(t^n) \mid p^{n,*}(x)=0 \},
\]
from which a new set of control points $\{\mathbf{X}_{0,k}^{n,*}\}_{k=1}^N$ is obtained.

\emph{Step 3: Intermediate Field Update.}
Using the predicted inner boundary $\{\mathbf{X}_{0,k}^{n,*}\}_{k=1}^N$ together with the original outer boundary $\{\mathbf{X}_{1,k}^n\}_{k=1}^N$, 
the nutrient equation and the obstacle problem are solved again to obtain $c^{n,**}$ and $p^{n,**}$.

\emph{Step 4: Outer Boundary Advection.}
The outer boundary is advanced explicitly according to the normal velocity law.
At each control point,
\[
v_k^n = -\nabla p^{n,**}(\mathbf{X}_{1,k}^n)\cdot \mathbf{n}_k^n,
\]
and the control points are updated by
\[
\mathbf{X}_{1,k}^{n+1} = \mathbf{X}_{1,k}^n + \Delta t\, v_k^n \mathbf{n}_k^n.
\]
A smooth closed curve is then reconstructed using periodic cubic spline interpolation, and the control points are redistributed uniformly with respect to arc length.

\emph{Step 5: Inner Boundary Corrector.}
With the updated outer boundary $\{\mathbf{X}_{1,k}^{n+1}\}_{k=1}^N$ and the predicted inner boundary $\{\mathbf{X}_{0,k}^{n,*}\}_{k=1}^N$ , 
the nutrient equation and obstacle problem are solved once more to obtain $p^{n,***}$.
A corrected inner boundary $\{\mathbf{X}_{0,k}^{n,***}\}_{k=1}^N$ is extracted from the zero-level set
\[
N(t^{n+1})^{***} := \{ x \in \Omega(t^{n+1}) \mid p^{n,***}(x)=0 \}.
\]
The inner boundary is then updated via averaging:
\[
\mathbf{X}_{0,k}^{n+1}
= \frac{1}{2}\mathbf{X}_{0,k}^{n,*}
+ \frac{1}{2}\mathbf{X}_{0,k}^{n,***}, 
\qquad k=1,\dots,N.
\]

\begin{remark}[Level Set Reconstruction] \label{rmk3}
To reconstruct the inner boundary $\Gamma_0$ from the discrete pressure field, we identify grid points adjacent to the zero-level set. To ensure smoothness and filter grid noise, these points are sorted angularly and approximated using a truncated Fourier series. The truncation order is selected adaptively to minimize the maximum pointwise distance between the grid points and the reconstructed curve. This Fourier smoothing is crucial for maintaining the geometric regularity of the free boundary throughout the simulation.
\end{remark}

\begin{remark}[Handling Necrotic Core Nucleation]
The algorithm described above assumes the necrotic core $N(t)$ is already formed and resolvable. The transient phase where the core first emerges (a topological change from $N(t)=\emptyset$ to $N(t)\neq\emptyset$) requires a specialized heuristic treatment to detect the onset of necrosis and initialize the inner boundary. We detail this practical strategy in Section \ref{Numerical Observation of Necrotic Core Emergence} in the context of the relevant numerical experiment.
\end{remark}

This coupled iterative procedure yields a robust evolution of the tumor boundaries. By updating the inner boundary via a predictor-corrector averaging strategy rather than a single snapshot, the scheme effectively suppresses numerical oscillations that arise from the implicit level-set definition. The critical importance of this stabilizing strategy is demonstrated in Section \ref{Numerical Results for the in vitro Model with Necrotic Core}, particularly in Example 4, where it is shown to prevent the radius instabilities observed with standard forward Euler updates.

\section{Numerical Results} \label{Numerical Results}
This section presents numerical experiments designed to validate the proposed numerical methods and to explore the qualitative behavior of the corresponding free boundary models. 
We first verify the accuracy and reliability of the numerical scheme by considering radially symmetric configurations, for which the evolution of the tumor boundary and, when present, the necrotic core boundary can be derived analytically. 
The numerical results are compared against these analytical solutions, providing quantitative evidence for the correctness and effectiveness of the proposed methods.

After validating the scheme in the radially symmetric setting, we further investigate more general geometries by considering perturbed circular initial boundaries. 
These experiments are used to examine the stability of the numerical method and to observe the evolution of nontrivial tumor shapes beyond the idealized symmetric case.

In addition to simulations without an initial necrotic core and with a prescribed initial necrotic core, we also highlight a numerical experiment that captures the entire process of necrotic core formation starting from a fully viable tumor. 
The emergence and subsequent evolution of the necrotic region are resolved dynamically by the numerical method; see Subsection~\ref{Numerical Observation of Necrotic Core Emergence} for detailed results.

\subsection{Numerical Results for the Tumor Model without Necrotic Core}
\label{Numerical Results for the in vitro Model without Necrotic Core}

In this section, we present several numerical experiments for the tumor growth model without a necrotic core, described in Section~\ref{The in vitro model}. 
The purpose of these experiments is threefold: 
(i) to validate the accuracy of the proposed numerical method by comparison with analytical solutions when available; 
(ii) to investigate the ability of the method to capture geometric evolution of tumor boundaries under different initial configurations; and 
(iii) to examine the influence of model parameters on the shape relaxation and long-time behavior of the tumor boundary.

\medskip
\noindent \textbf{Example 1 (Radially symmetric validation test).} 
This example is designed to validate the accuracy and convergence of the numerical method by comparing the numerical solution with an exact radially symmetric solution. 
We consider the model~\eqref{eq0915_1}–\eqref{eq0915_2} with an initially circular tumor boundary $\partial D(0)$ of radius $0.5$. 
The parameters are chosen as $c_B = 10$, $\lambda = 0.5$, and $G_0 = 0.1$.

For this configuration, it has been shown in~\cite{feng2023tumor} that the tumor boundary remains circular throughout the evolution, and its radius $R(t)$ satisfies the ordinary differential equation
\begin{equation}
\label{eq1020_1}
\dot{R} = \frac{G_0 c_B I_1(\sqrt{\lambda} R)}{\sqrt{\lambda} I_0(\sqrt{\lambda} R)},
\end{equation}
where $I_0$ and $I_1$ denote the modified Bessel functions of the first kind of orders $0$ and $1$, respectively. 
This provides a reliable benchmark for assessing the numerical accuracy of the proposed method.

The computational domain is taken as $\mathcal{B} = [-1.5, 1.5] \times [-1.5, 1.5]$, and the number of boundary control points is set to $N = 32$. 
Figure~\ref{fig:radial_evolution} illustrates the time evolution of the numerical tumor boundary, while quantitative comparisons of the numerical and exact radii at $T=0.2$ are reported in Table~\ref{tab1020_1}.

\begin{figure}[htbp]
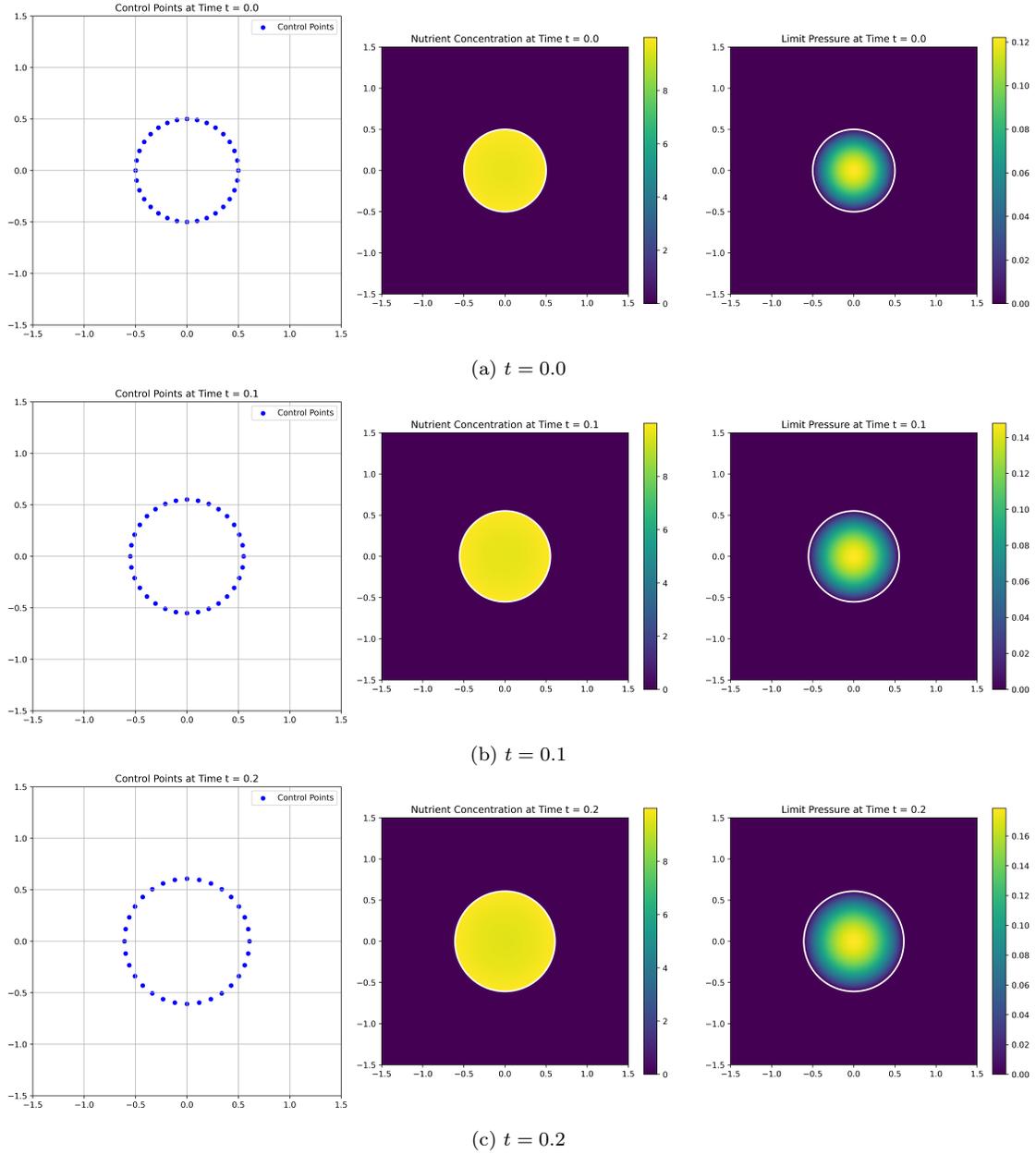

    \centering
    \begin{subfigure}{0.9\textwidth}
        \centering
        \includegraphics[width=\textwidth]{fig1020_1.png}
        \caption{$t=0.0$}
    \end{subfigure}
    \begin{subfigure}{0.9\textwidth}
        \centering
        \includegraphics[width=\textwidth]{fig1020_3.png}
        \caption{$t=0.1$}
    \end{subfigure}
    \begin{subfigure}{0.9\textwidth}
        \centering
        \includegraphics[width=\textwidth]{fig1020_5.png}
        \caption{$t=0.2$}
    \end{subfigure}
    \caption{Time evolution of the numerical solution for the radially symmetric test case. The computations are performed with $I=J=256$ and $\Delta t=0.01$.}
    \label{fig:radial_evolution}
\end{figure}

\begin{table}[ht]
\centering
\begin{tabular}{c|c|c|c|c|c} 
\hline
$I$ & $J$ & $\Delta t$ & $n_T$ & $R_{\text{numerical}}$ & error \\
\hline
64 & 64 & 0.04 & 5 & 0.60673039 & $1.696 \times 10^{-3}$ \\
128 & 128 & 0.02 & 10 & 0.60756377 & $8.629 \times 10^{-4}$ \\
256 & 256 & 0.01 & 20 & 0.60798085 & $4.438 \times 10^{-4}$ \\
512 & 512 & 0.005 & 40 & 0.60819969 & $2.270 \times 10^{-4}$ \\
\hline
\end{tabular}
\caption{Numerical results and radius errors at $T = 0.2$. $I$ and $J$ denote the discretization parameters of the Cartesian grid, with grid spacings $h_x = \frac{x_\text{max}-x_\text{min}}{I}$ and $h_y = \frac{y_\text{max}-y_\text{min}}{J}$, $n_T$ is the number of time steps to reach $T=0.2$. $R_{\text{exact}}$ is obtained by solving the ODE \eqref{eq1020_1} using Python's \texttt{solve\_ivp} function with \texttt{method=RK45'} from \texttt{scipy.integrate}, giving $R_{\text{exact}}(T)=0.60842668$. $R_{\text{numerical}}$ is computed by first performing periodic cubic spline interpolation of the boundary control points from the numerical solution, then parameterizing the boundary, computing the arc length, and dividing by $2\pi$. In the last column, error $= |R_{\text{exact}} - R_{\text{numerical}}|$.}
\label{tab1020_1}
\end{table}

From Figure~\ref{fig:radial_evolution}, the numerical tumor boundary remains circular and expands monotonically in time, in agreement with the analytical prediction. 
Table~\ref{tab1020_1} shows that the numerical radius converges to the exact solution as the spatial and temporal resolutions are refined, demonstrating both the accuracy and stability of the proposed numerical method.

\medskip
\noindent \textbf{Example 2 (Relaxation of non-circular initial shapes).} 
This example aims to investigate the capability of the numerical method to capture the shape relaxation of tumor boundaries starting from a non-circular geometry, as well as the influence of the parameter $\lambda$ on this process. 
The initial tumor boundary is chosen as an ellipse,
\[
\partial D(0) = \left\{(x, y): \frac{x^2}{a^2} + \frac{y^2}{b^2} = 1 \right\},
\]
with semi-axes $a = 2.3$ and $b = 1.1$. 
The parameters are set to $c_B = 10$ and $G_0 = 1$. 
The computational domain is $\mathcal{B} = [-10, 10] \times [-10, 10]$, with $I = J = 256$, $\Delta t = 0.01$, and $N = 32$ control points.

When $\lambda = 1$, Figure~\ref{fig1020_6} shows that the initially elliptical tumor boundary gradually evolves toward a circular shape. In contrast, for a larger value $\lambda = 50$, the boundary evolution becomes noticeably more rigid, and the relaxation toward a circular shape is significantly slowed down, as shown in Figure~\ref{fig1020_7}. These results indicate that larger values of $\lambda$ inhibit shape deformation and reduce the flexibility of the tumor boundary.

\begin{figure}[htbp]
    \centering
    \begin{minipage}[t]{0.48\textwidth}
        \centering
        \includegraphics[width=\textwidth]{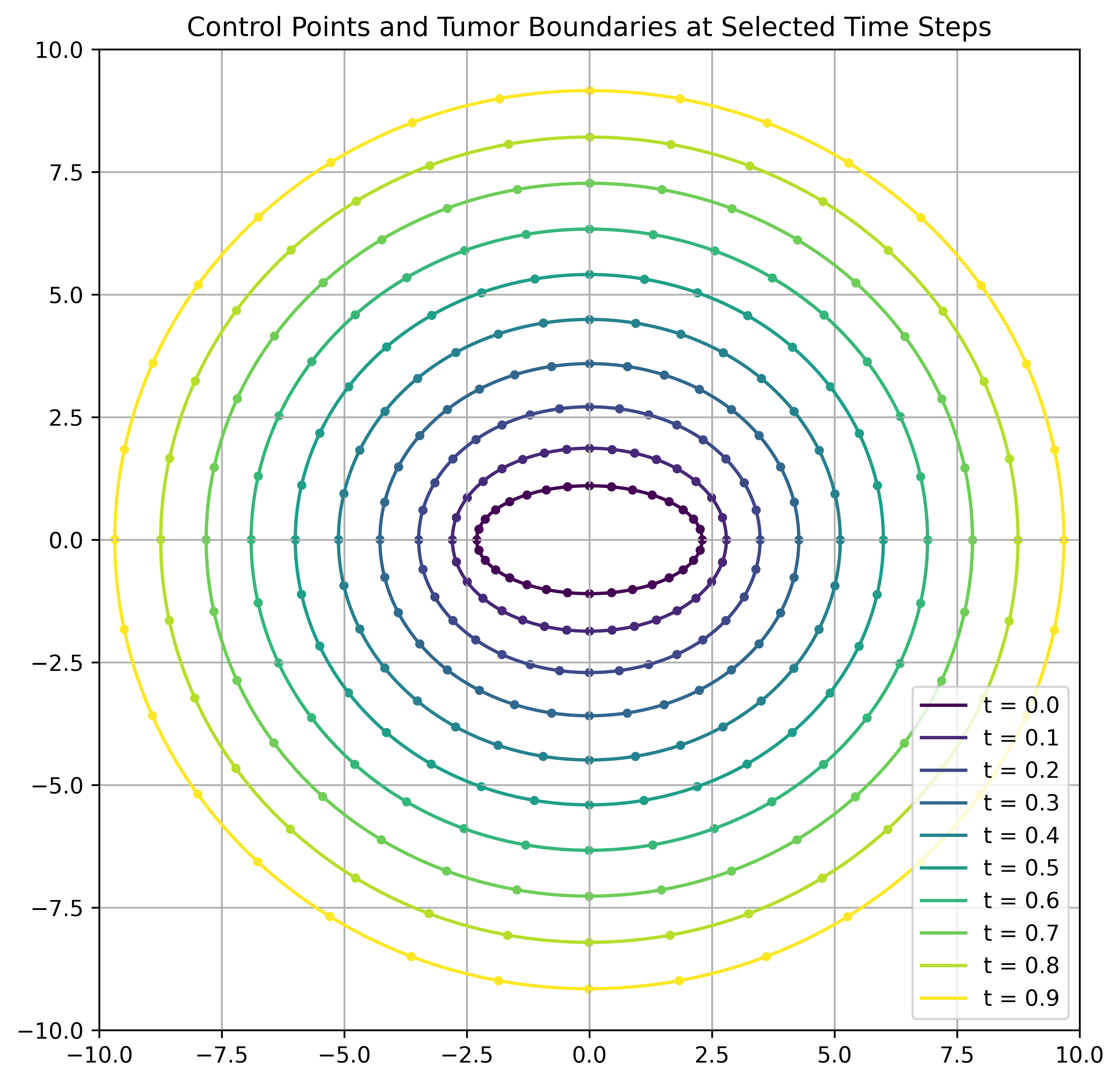}
        \caption{Evolution of the tumor boundary with $\lambda=1$.}
        \label{fig1020_6}
    \end{minipage}
    \hfill
    \begin{minipage}[t]{0.48\textwidth}
        \centering
        \includegraphics[width=\textwidth]{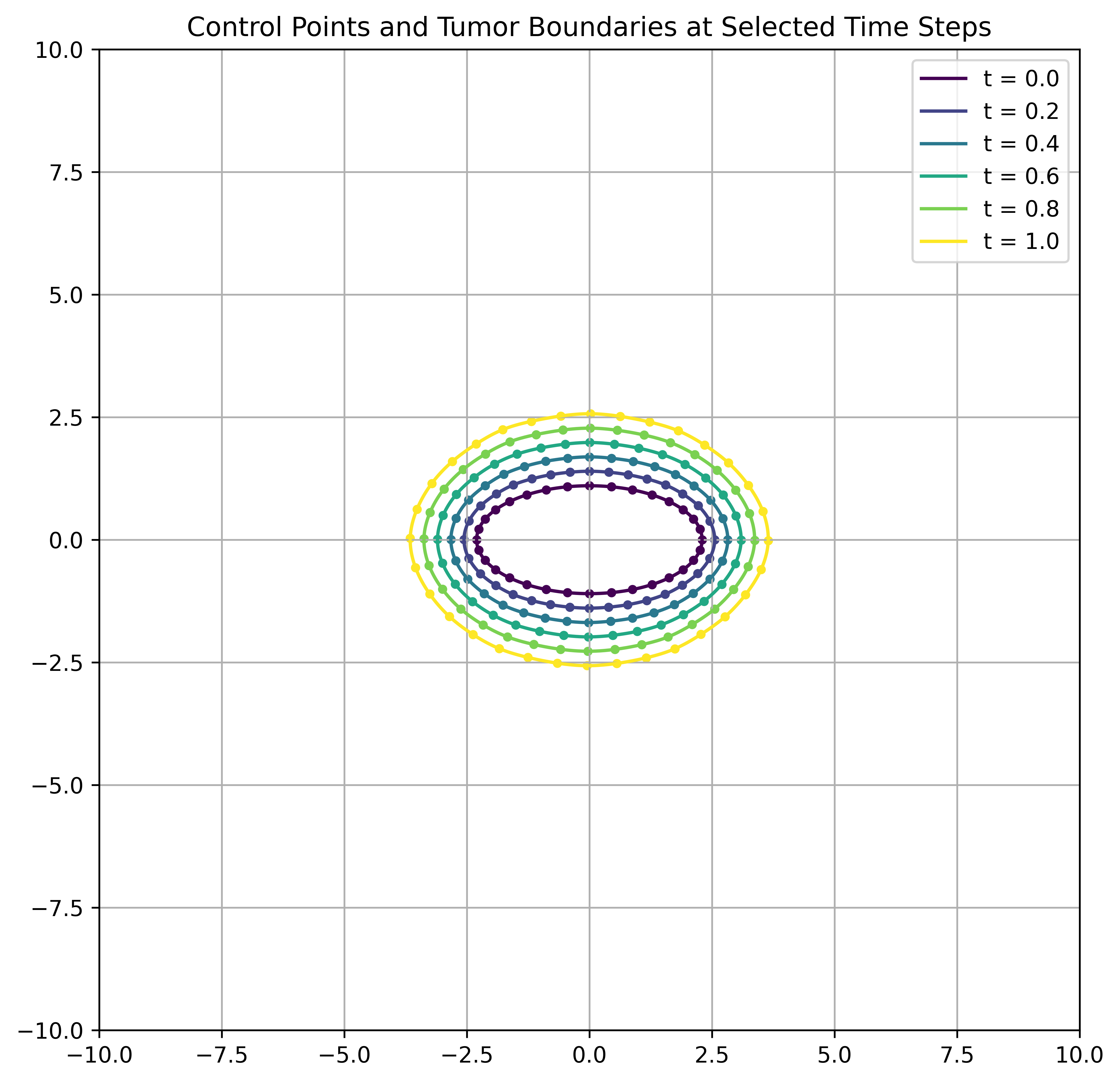}
        \caption{Evolution of the tumor boundary with $\lambda=50$.}
        \label{fig1020_7}
    \end{minipage}
\end{figure}

\medskip
\noindent \textbf{Example 3 (Stability under boundary perturbations).} 
This example examines the stability of the circular tumor shape under small perturbations of different modes. 
The initial boundary is taken as a perturbed circle,
\[
\partial D(0) = \{(x, y): x = r(\theta)\cos\theta,\ y = r(\theta)\sin\theta,\ r(\theta) = 0.8 + 0.02\cos(l\theta)\},
\]
where $l$ denotes the perturbation mode. 
The parameters are chosen as $c_B = 10$, $\lambda = 10$, and $G_0 = 0.2$. 
The computational domain is $\mathcal{B} = [-1.5, 1.5] \times [-1.5, 1.5]$, with $I = J = 256$, $\Delta t = 0.01$, and $N = 128$ control points.

Figures~\ref{fig1022_1}–\ref{fig1022_4} present the evolution of the tumor boundary for perturbation modes $l = 6, 8, 10$, and $12$, respectively. 
In all cases, the initially perturbed boundary gradually smooths out and converges to a larger circular shape as time evolves. 
This indicates that the circular tumor configuration is dynamically stable under small perturbations, and the numerical results are consistent with the analytical and numerical observations reported in~\cite{feng2023tumor}.

\begin{figure}[htbp]
    \centering
    \begin{minipage}[t]{0.48\textwidth}
        \centering
        \includegraphics[width=\textwidth]{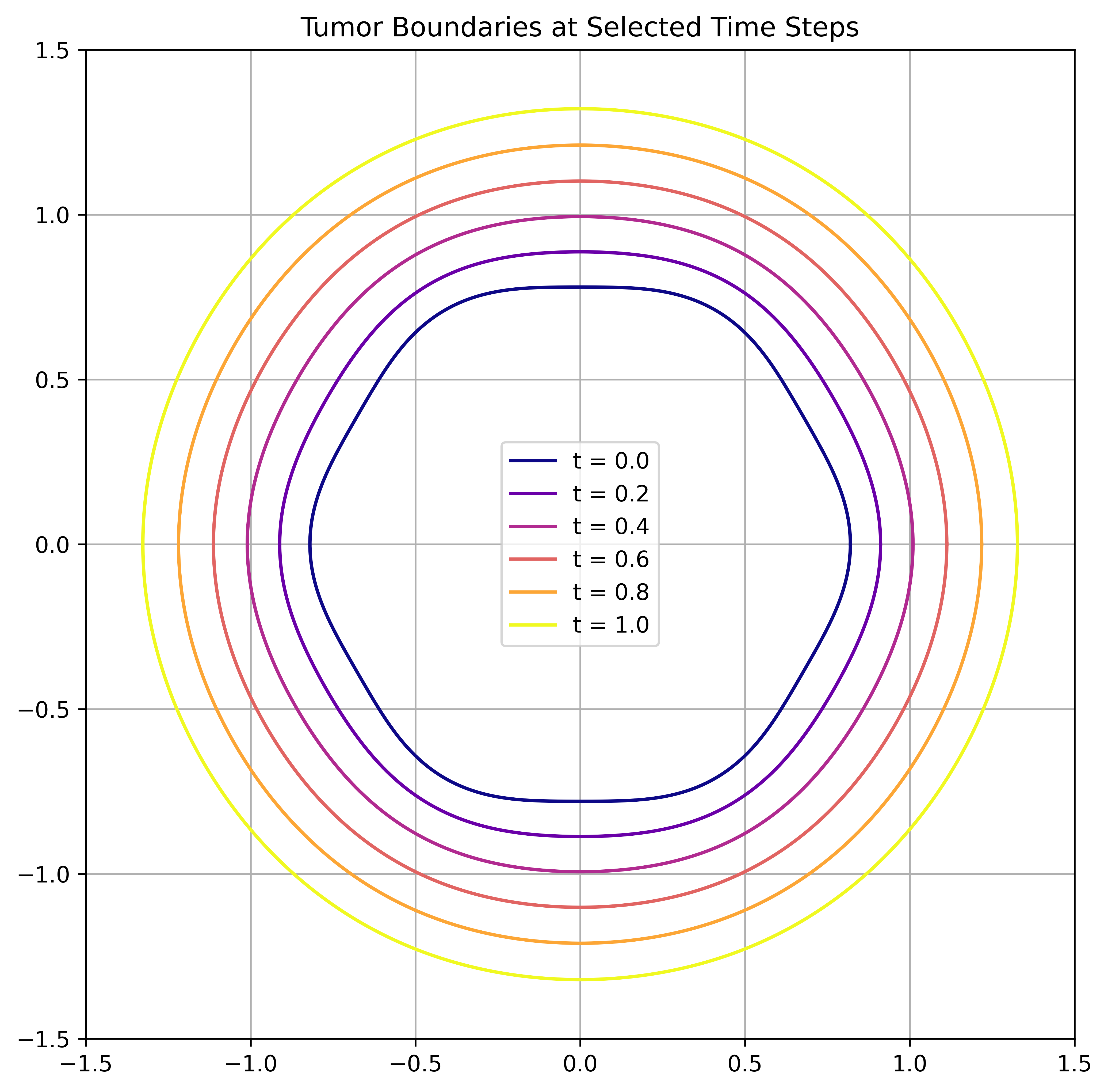}
        \caption{Evolution of the tumor boundary with $l=6$.}
        \label{fig1022_1}
    \end{minipage}
    \hfill
    \begin{minipage}[t]{0.48\textwidth}
        \centering
        \includegraphics[width=\textwidth]{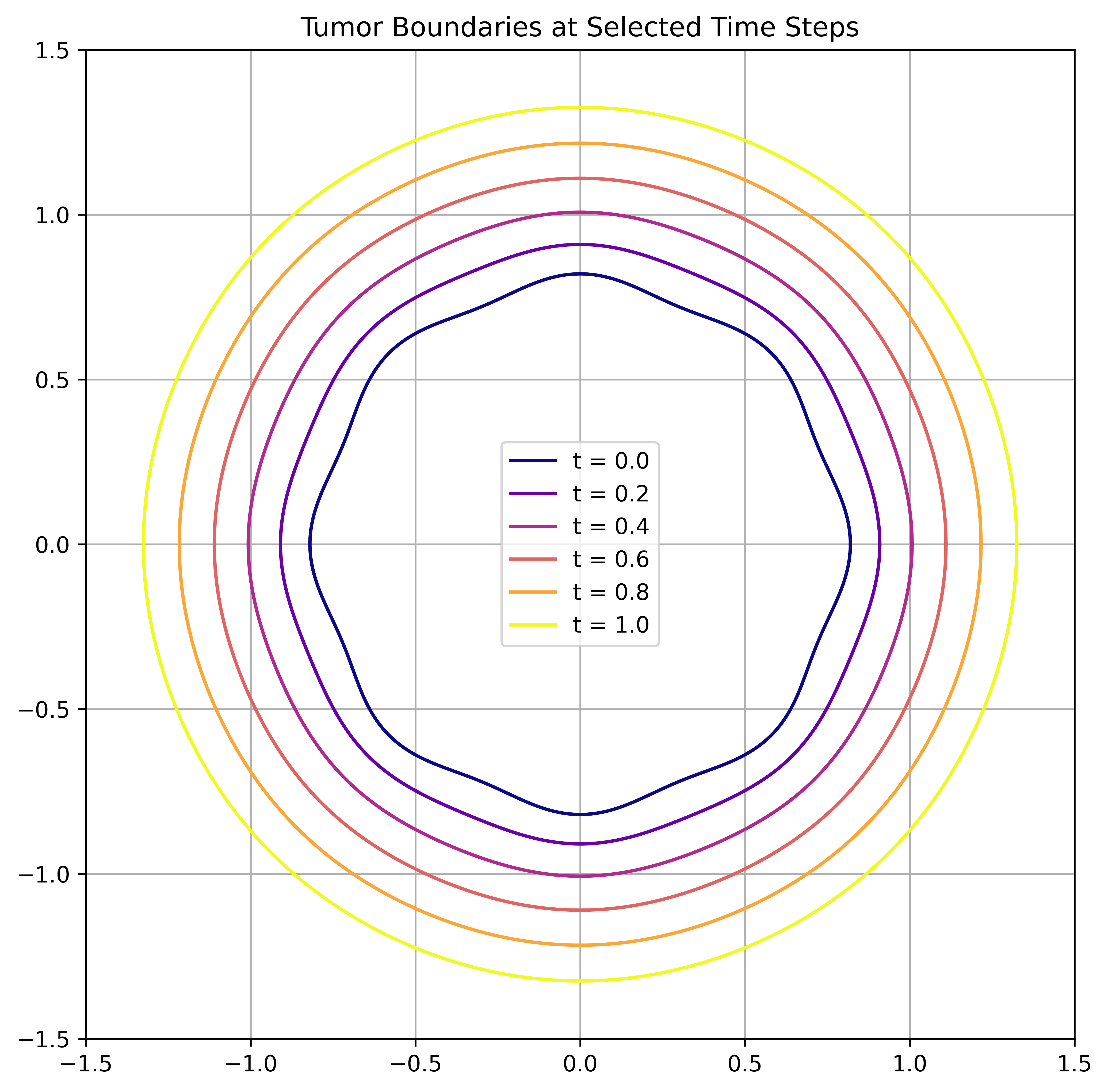}
        \caption{Evolution of the tumor boundary with $l=8$.}
        \label{fig1022_2}
    \end{minipage}
\end{figure}
\begin{figure}[htbp]
    \centering
    \begin{minipage}[t]{0.48\textwidth}
        \centering
        \includegraphics[width=\textwidth]{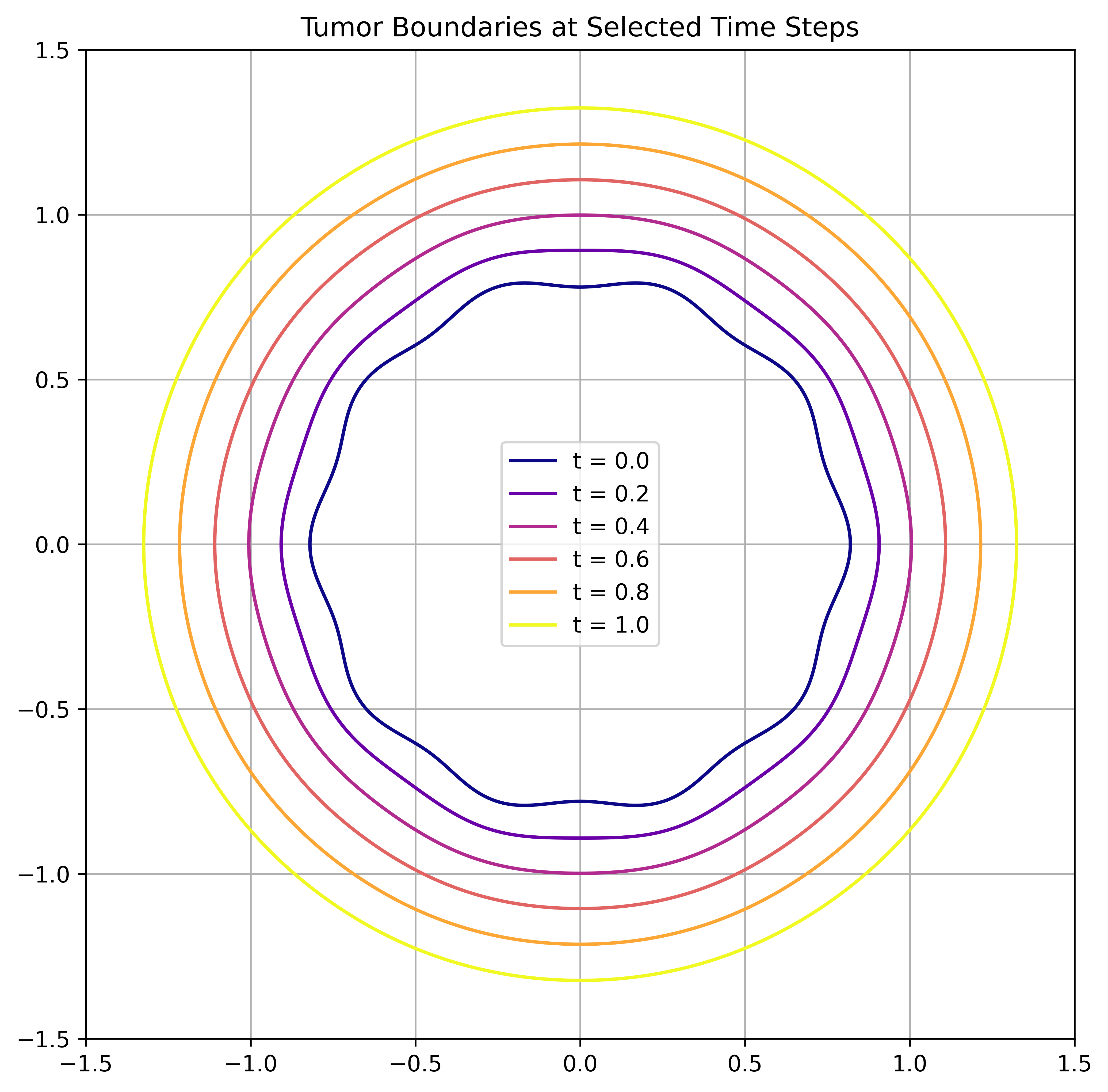}
        \caption{Evolution of the tumor boundary with $l=10$.}
        \label{fig1022_3}
    \end{minipage}
    \hfill
    \begin{minipage}[t]{0.48\textwidth}
        \centering
        \includegraphics[width=\textwidth]{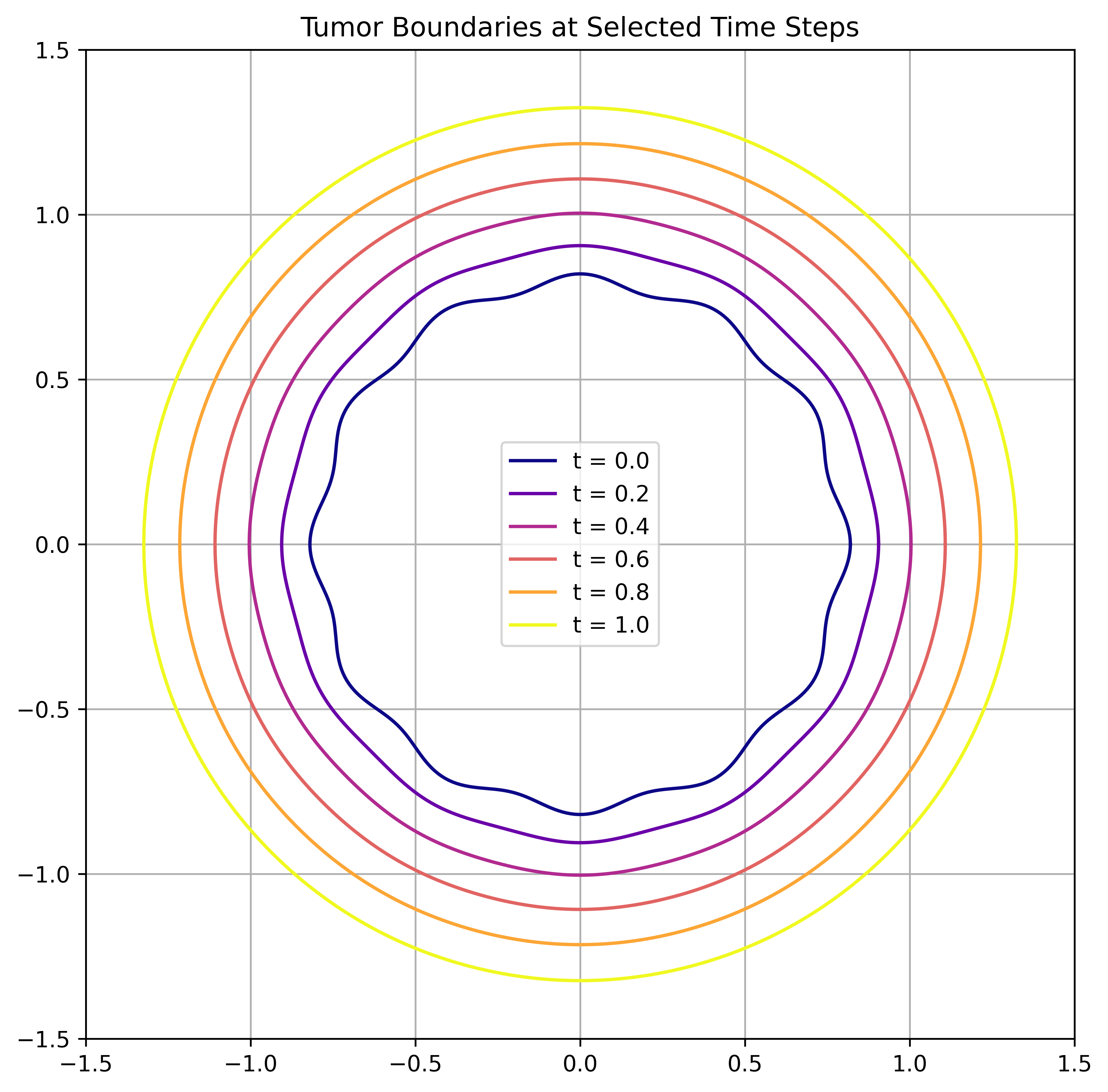}
        \caption{Evolution of the tumor boundary with $l=12$.}
        \label{fig1022_4}
    \end{minipage}
\end{figure}

\subsection{Numerical Results for the Tumor Model with Necrotic Core}
\label{Numerical Results for the in vitro Model with Necrotic Core}
In this section, we investigate the numerical performance of the proposed method for the tumor growth model with a necrotic core, introduced in Section~\ref{The in vitro Model with Necrotic Core}. 
Compared with the model without necrosis, the presence of two moving interfaces significantly increases the numerical complexity due to the strong coupling between the pressure field and the inner and outer tumor boundaries. 
The main goals of the following experiments are: 
(i) to validate the numerical accuracy of the proposed scheme using radially symmetric configurations with known analytical descriptions; and 
(ii) to demonstrate the necessity and effectiveness of the improved coupling strategy for stabilizing the interface evolution.

\medskip
\noindent \textbf{Example 4 (Radially symmetric benchmark with necrotic core).} 
This example is designed as a benchmark test to assess the accuracy and stability of the numerical method in the presence of a necrotic core. 
We consider the model~\eqref{eq0916_1}–\eqref{eq0916_3} with initially circular inner and outer tumor boundaries. 
As shown in \ref{sec: After formation of the necrotic core}, the radial symmetry is preserved throughout the evolution, allowing the interface dynamics to be characterized exactly by a coupled system of ordinary differential equations. Let $R_0(t)$ and $R_1(t)$ denote the radii of the inner boundary $\Gamma_0$ and outer boundary $\Gamma_1$, respectively, with initial values $R_0^{(0)}$ and $R_1^{(0)}$. 
Then the evolution of the outer radius $R_1(t)$ is governed by
\begin{equation}
\label{eq1024_1}
\frac{d R_1(t)}{dt} = G_0 \left[\frac{1}{\sqrt{\lambda}} \big(b_0(t) I_1(\sqrt{\lambda} R_1(t)) - b_1(t) K_1(\sqrt{\lambda} R_1(t))\big) - \frac{R_1(t)}{2}\bar{c} - \frac{A(t)}{R_1(t)} \right],
\end{equation}
which is coupled with the transcendental relation
\begin{equation}
\label{eq1024_2}
\frac{(R_1(t))^2}{4}\bar{c} + A(t)\ln R_1(t) + B(t)
- \frac{1}{\lambda} \big(b_0(t) I_0(\sqrt{\lambda} R_1(t)) + b_1(t) K_0(\sqrt{\lambda} R_1(t))\big) = 0.
\end{equation}
where the coefficients $b_0(t)$, $b_1(t)$, $A(t)$, and $B(t)$ depend on $R_0(t)$ and $R_1(t)$ as follows: \begin{equation} \label{eq1024_3} b_0(t) = \frac{c_B}{I_0(\sqrt{\lambda} R_1(t)) + K_0(\sqrt{\lambda} R_1(t)) \frac{I_0(\sqrt{\lambda n_c} R_0(t)) I_1(\sqrt{\lambda} R_0(t)) - \sqrt{n_c}\, I_1(\sqrt{\lambda n_c} R_0(t)) I_0(\sqrt{\lambda} R_0(t))}{ I_0(\sqrt{\lambda n_c} R_0(t)) K_1(\sqrt{\lambda} R_0(t)) + \sqrt{n_c}\, I_1(\sqrt{\lambda n_c} R_0(t)) K_0(\sqrt{\lambda} R_0(t))}}, \end{equation} \begin{equation} \label{eq1024_4} b_1(t) = b_0(t) \frac{I_0(\sqrt{\lambda n_c} R_0(t)) I_1(\sqrt{\lambda} R_0(t)) - \sqrt{n_c}\, I_1(\sqrt{\lambda n_c} R_0(t)) I_0(\sqrt{\lambda} R_0(t))}{ I_0(\sqrt{\lambda n_c} R_0(t)) K_1(\sqrt{\lambda} R_0(t)) + \sqrt{n_c}\, I_1(\sqrt{\lambda n_c} R_0(t)) K_0(\sqrt{\lambda} R_0(t))}, \end{equation} \begin{equation} \label{eq1024_5} A(t) = \frac{R_0(t)}{\sqrt{\lambda}} \big(b_0(t) I_1(\sqrt{\lambda} R_0(t)) - b_1(t) K_1(\sqrt{\lambda} R_0(t))\big) - \frac{(R_0(t))^2}{2}\bar{c}, \end{equation} \begin{equation} \label{eq1024_6} B(t) = \frac{1}{\lambda} \big(b_0(t) I_0(\sqrt{\lambda} R_0(t)) + b_1(t) K_0(\sqrt{\lambda} R_0(t))\big) - \frac{(R_0(t))^2}{2}\bar{c} - A(t)\ln R_0(t). \end{equation}

The parameters are chosen as
$c_B = 10$, $\lambda = 1$, $G_0 = 1$, $n_c = 10^{-3}$, and $\bar{c} = \tfrac{1}{2}c_B$.
The initial outer radius is $R_1^{(0)} = 2.5$, and the corresponding initial inner radius $R_0^{(0)} = 0.5751983378$ is obtained from~\eqref{eq1024_2}. 
The computational domain is $\mathcal{B} = [-5, 5] \times [-5, 5]$, and the numbers of control points on both inner and outer boundaries are set to $64$.

Before presenting the results obtained by the proposed algorithm, let us firstly examine the results produced by a forward Euler scheme similar to that described in Section \ref{The in vitro model}. Specifically, at each time step, we separately solve the interface PDE for $c$ and the obstacle problem for $p$. The inner boundary is then updated according to the zero-level set of $p$, while the outer boundary evolves based on the normal derivative of $p$ along the outer interface, using the forward Euler method. Although this approach appears natural, it fails to yield satisfactory numerical results. In particular, while both the inner and outer boundaries preserve their circular shapes, the evolution of their radii exhibits significant instability, as illustrated in Figures~\ref{fig1024_1}–\ref{fig1024_4}. This instability persists under different choices of the discrete parameters $I$, $J$, and $\Delta t$.

\begin{figure}[htbp]
    \centering
    \begin{subfigure}[t]{0.48\textwidth}
        \centering
        \includegraphics[width=\textwidth]{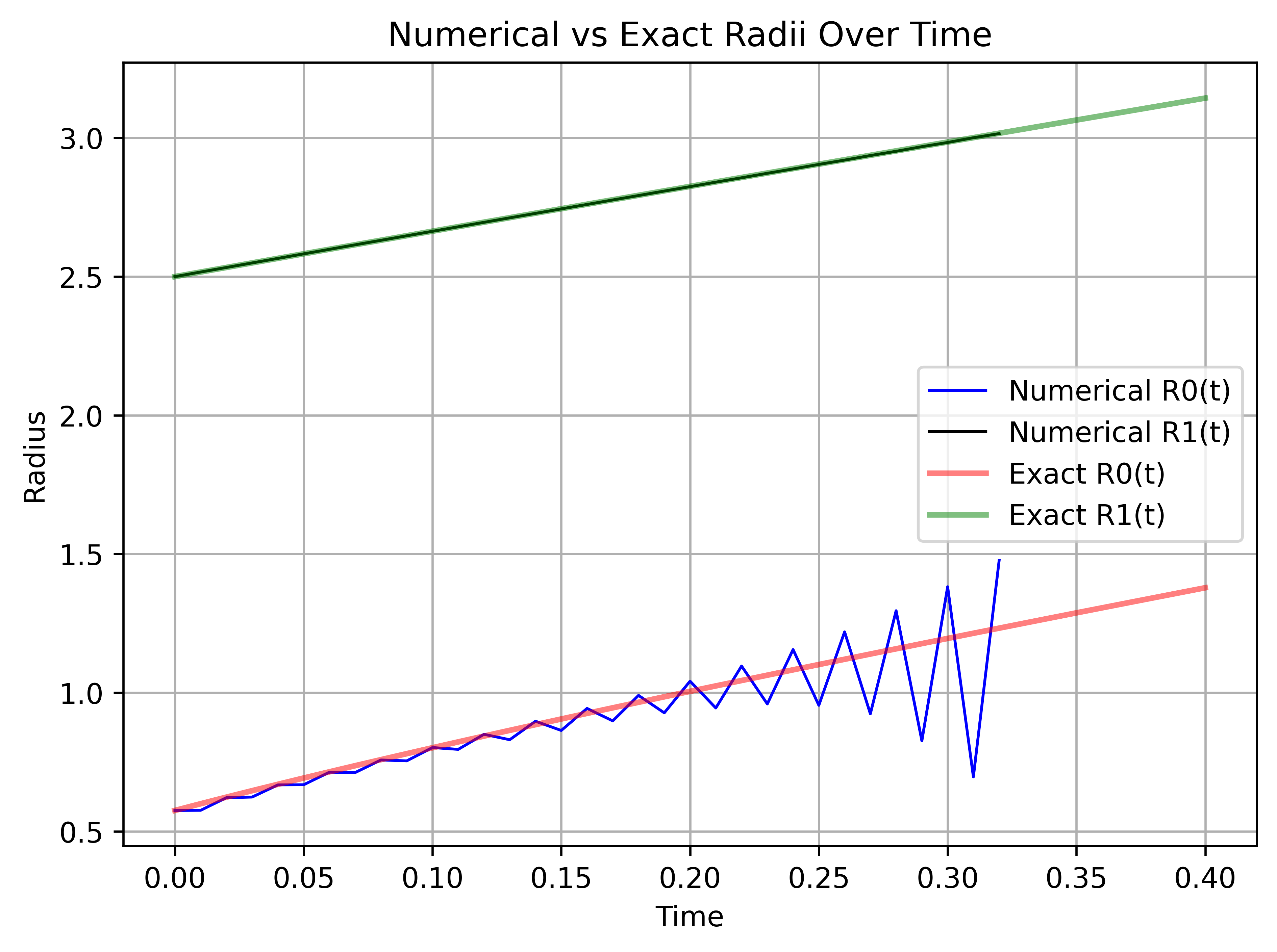}
        \caption{$I = J = 256$, $\Delta t = 0.01$}
        \label{fig1024_1}
    \end{subfigure}
    \hfill
    \begin{subfigure}[t]{0.48\textwidth}
        \centering
        \includegraphics[width=\textwidth]{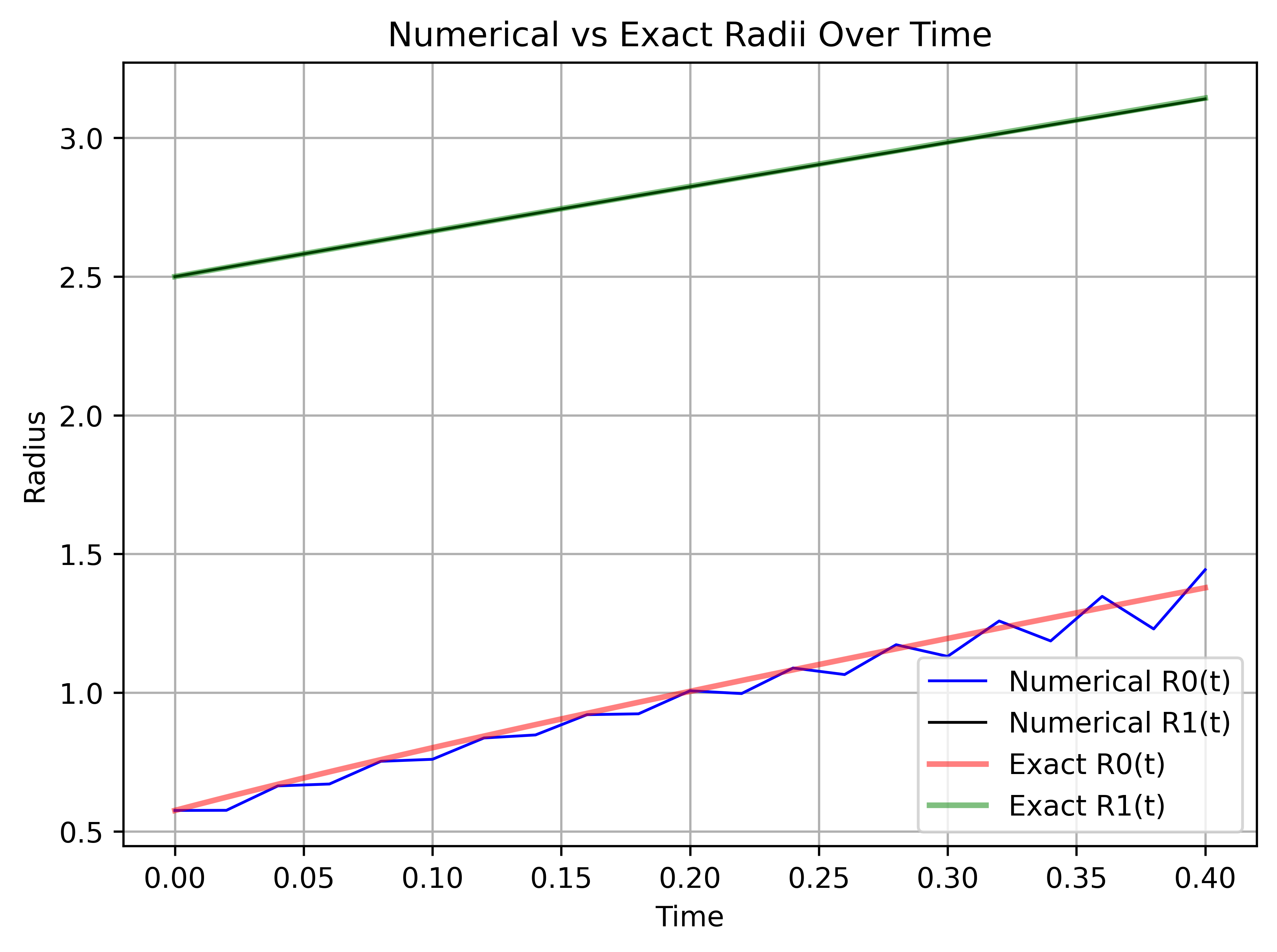}
        \caption{$I = J = 256$, $\Delta t = 0.02$}
        \label{fig1024_2}
    \end{subfigure}
    \\[1em]
    \begin{subfigure}[t]{0.48\textwidth}
        \centering
        \includegraphics[width=\textwidth]{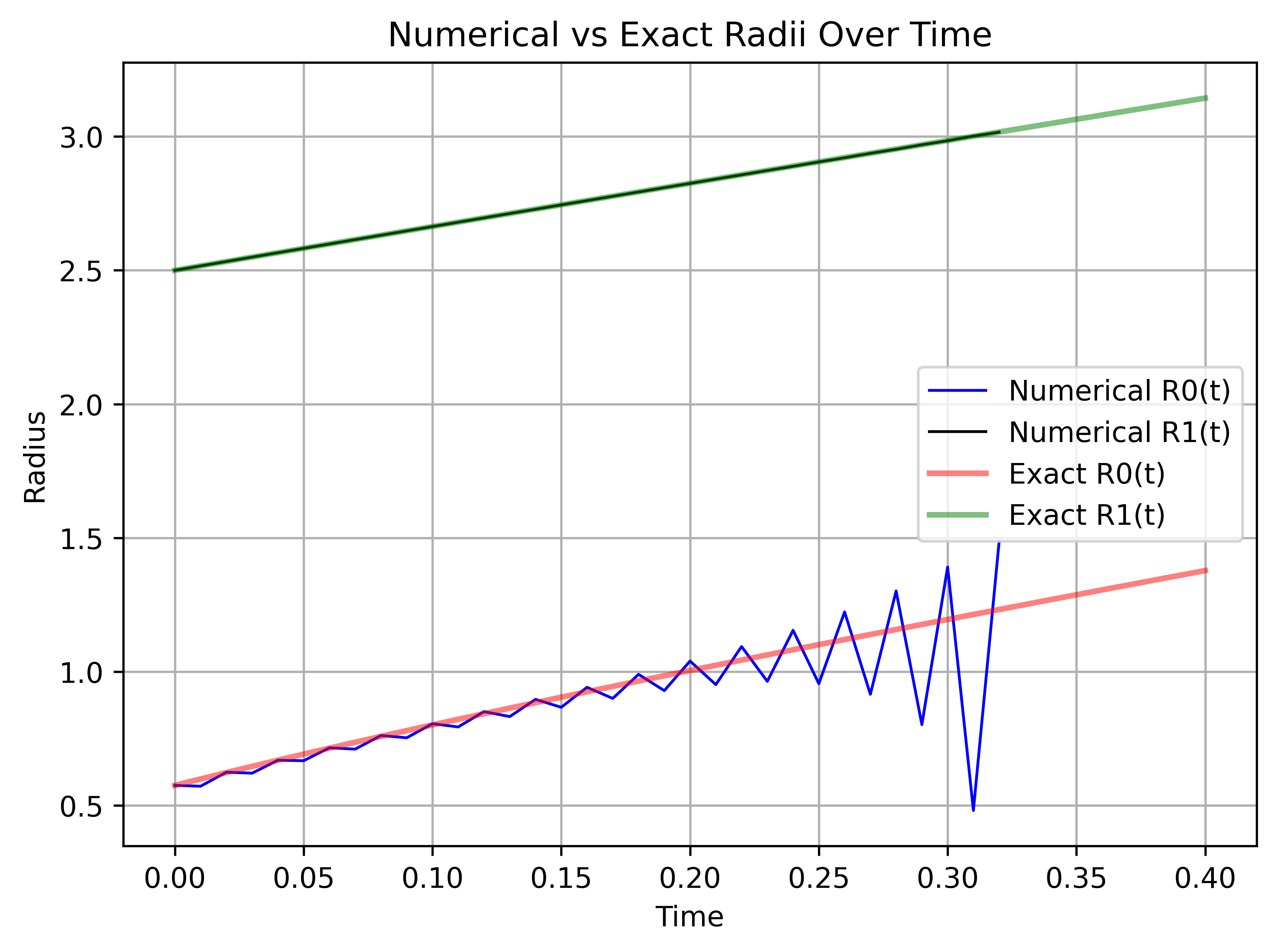}
        \caption{$I = J = 512$, $\Delta t = 0.01$}
        \label{fig1024_3}
    \end{subfigure}
    \hfill
    \begin{subfigure}[t]{0.48\textwidth}
        \centering
        \includegraphics[width=\textwidth]{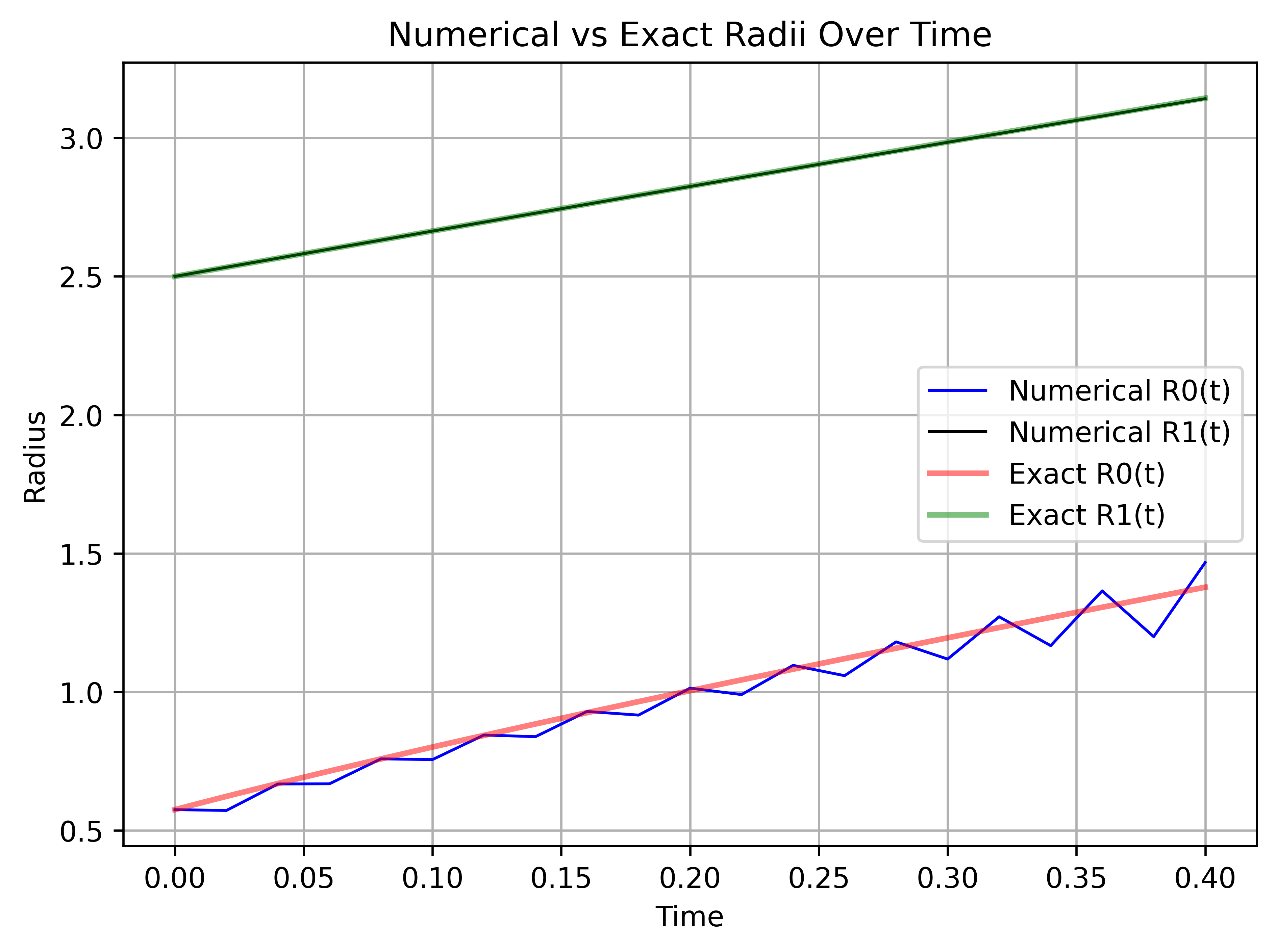}
        \caption{$I = J = 512$, $\Delta t = 0.02$}
        \label{fig1024_4}
    \end{subfigure}
    \caption{Comparison of numerical and exact tumor boundary radii obtained using a forward Euler interface update. The premature termination of some simulations indicates numerical instability caused by improper coupling between the pressure field and the moving interfaces.}
    \label{fig1024}
\end{figure}

In contrast, when the improved numerical scheme proposed in Section~\ref{The in vitro Model with Necrotic Core} is employed, the numerical behavior is significantly improved. 
Figures~\ref{fig1029_1}–\ref{fig1029_4} show excellent agreement between the numerical and exact radii of both interfaces for a wide range of spatial and temporal discretizations. 
The corresponding quantitative errors are summarized in Table~\ref{tab1029}, demonstrating clear convergence as the grid is refined. Moreover, the evolution of both the inner and outer boundaries becomes smooth and free of spurious oscillations, even for long-time simulations. 
Figure~\ref{fig1026} further illustrates that the proposed coupling strategy maintains the stability of the necrotic core and accurately captures the coupled dynamics of the two interfaces.

\begin{figure}[htbp]
    \centering
    \begin{subfigure}[t]{0.48\textwidth}
        \centering
        \includegraphics[width=\textwidth]{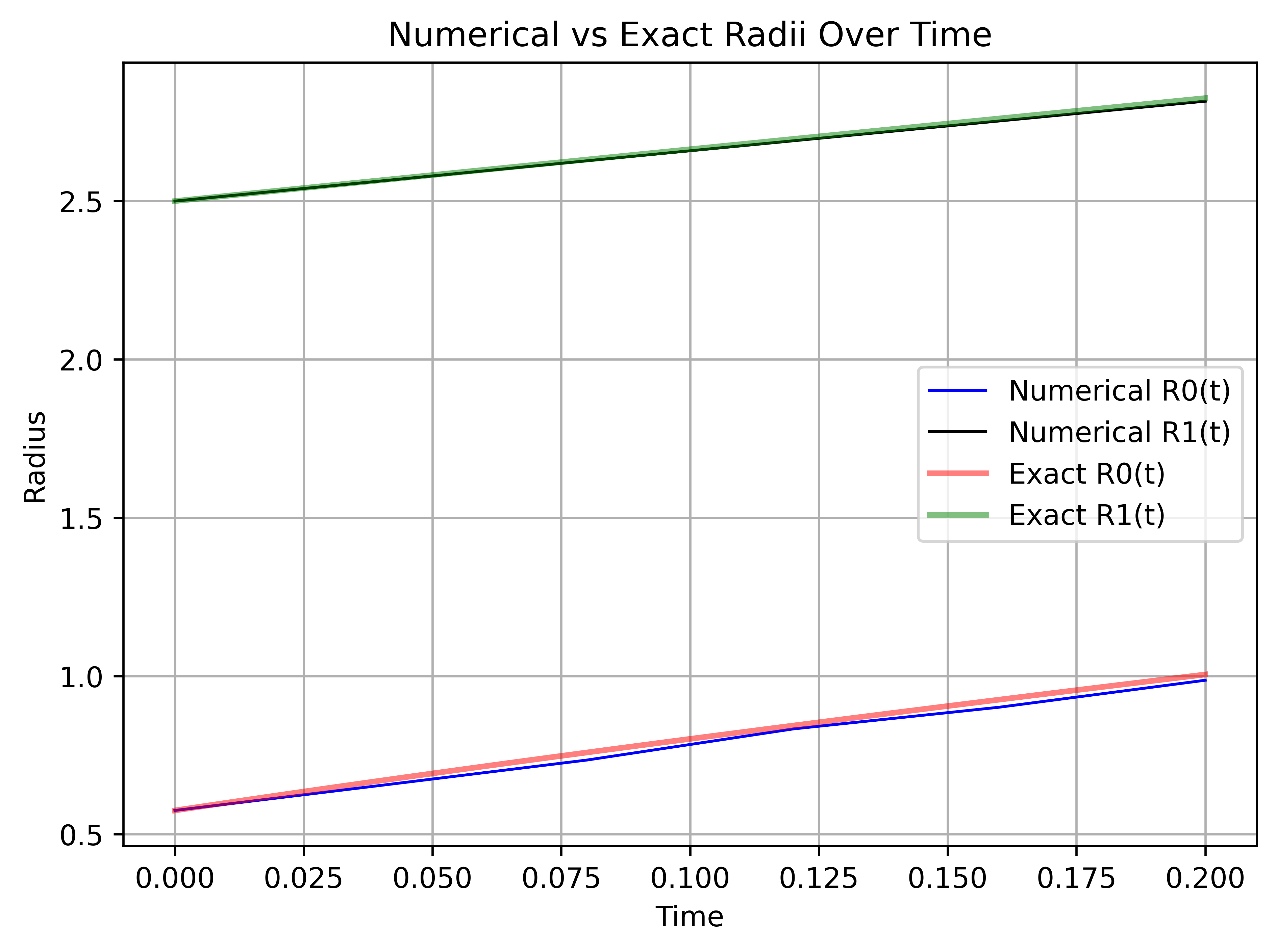}
        \caption{$I = J = 64$, $\Delta t = 0.04$}
        \label{fig1029_1}
    \end{subfigure}
    \hfill
    \begin{subfigure}[t]{0.48\textwidth}
        \centering
        \includegraphics[width=\textwidth]{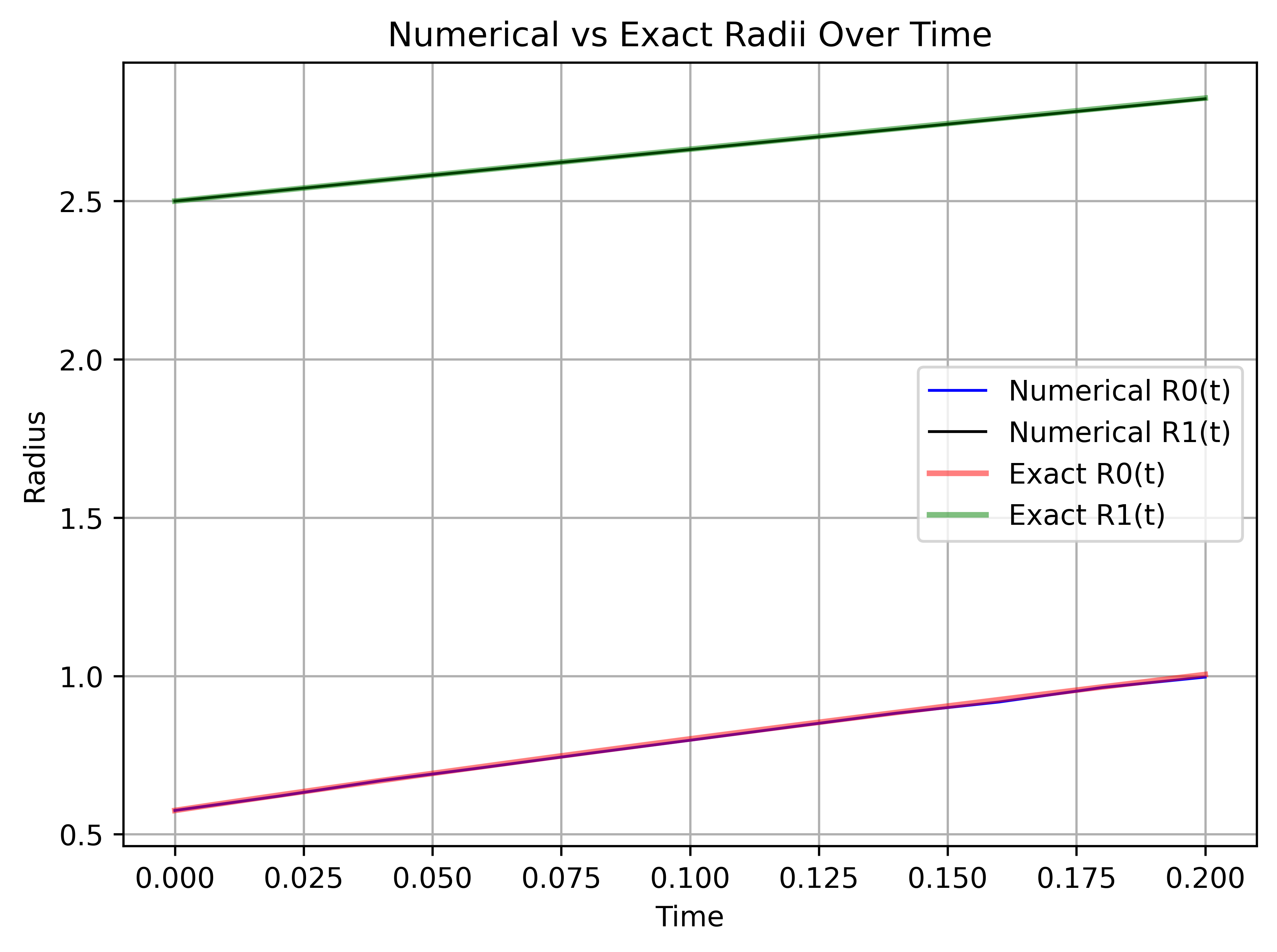}
        \caption{$I = J = 128$, $\Delta t = 0.02$}
        \label{fig1029_2}
    \end{subfigure}
    \\[1em]
    \begin{subfigure}[t]{0.48\textwidth}
        \centering
        \includegraphics[width=\textwidth]{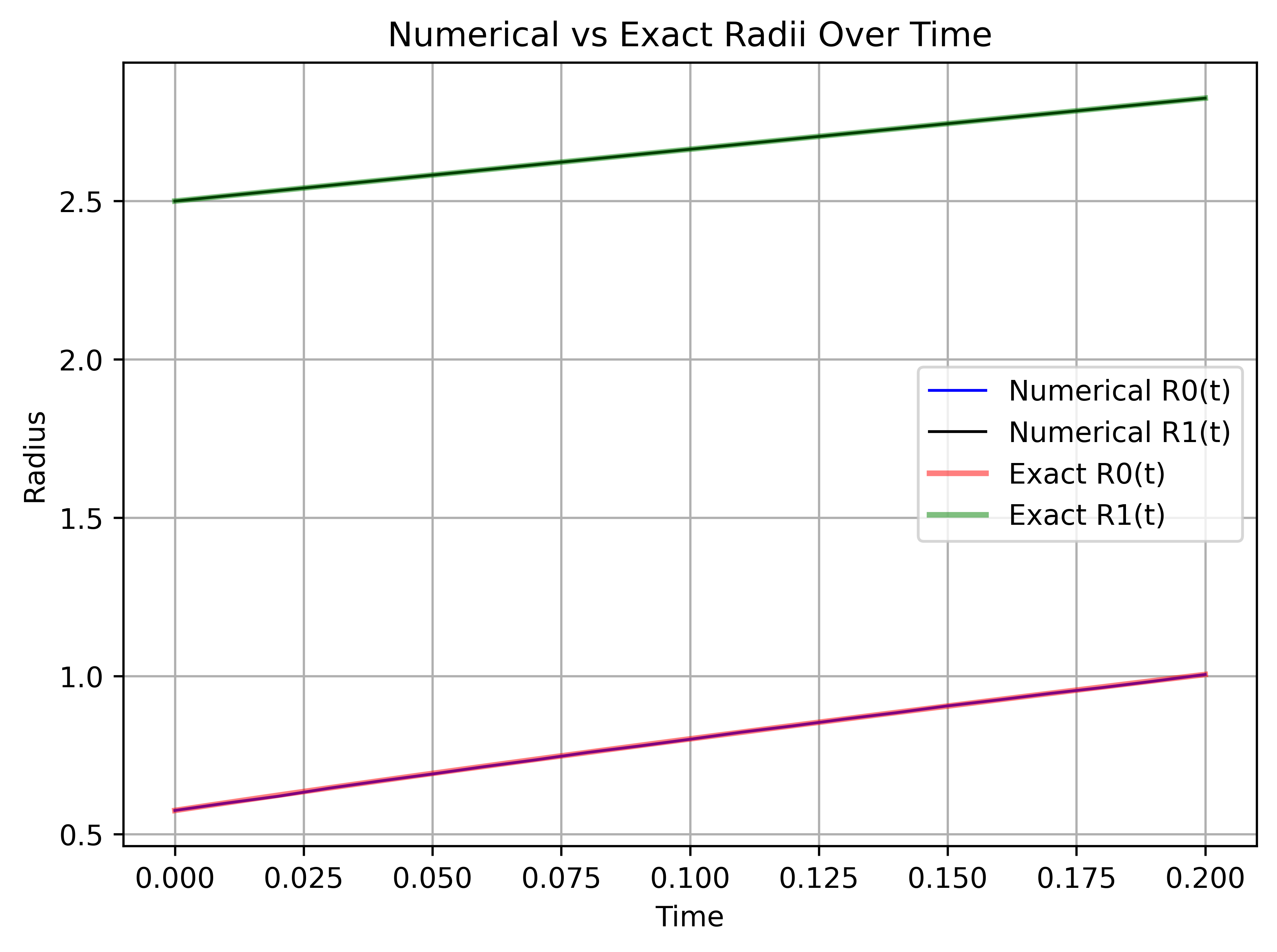}
        \caption{$I = J = 256$, $\Delta t = 0.01$}
        \label{fig1029_3}
    \end{subfigure}
    \hfill
    \begin{subfigure}[t]{0.48\textwidth}
        \centering
        \includegraphics[width=\textwidth]{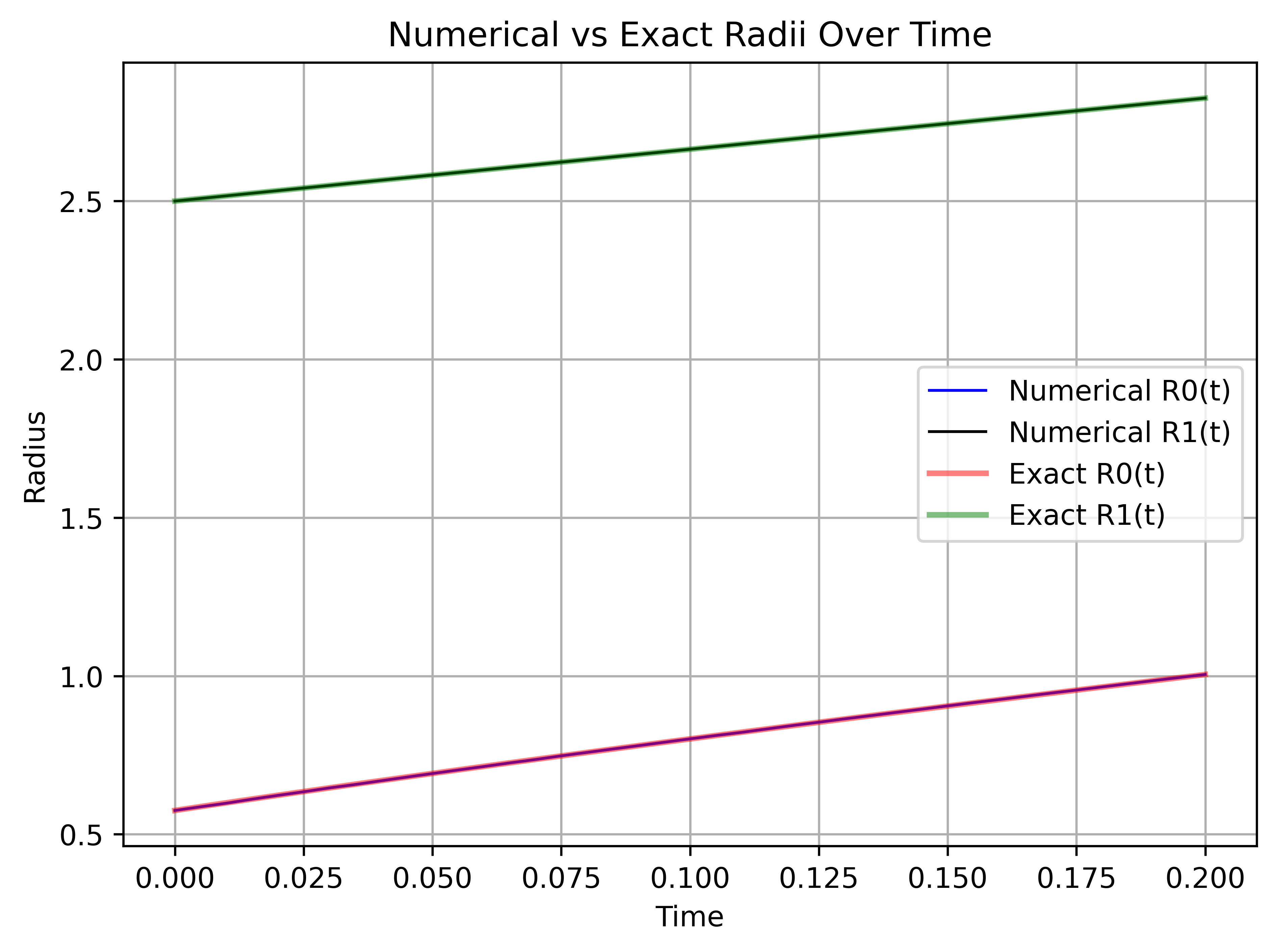}
        \caption{$I = J = 512$, $\Delta t = 0.005$}
        \label{fig1029_4}
    \end{subfigure}
    \caption{Comparison of numerical and exact tumor boundary radii obtained using the numerical method proposed in Section \ref{The in vitro Model with Necrotic Core}.}
    \label{fig1029}
\end{figure}

\begin{table}[ht]
\centering
\begin{tabular}{c|c|c|c|c|c} 
\hline
$I$ & $J$ & $\Delta t$ & $n_T$ & error for $R_0(t)$ & error for $R_1(t)$ \\
\hline
64 & 64 & 0.04 & 5 & $1.908 \times 10^{-2}$ & $6.552 \times 10^{-3}$ \\
128 & 128 & 0.02 & 10 & $4.394 \times 10^{-3}$ & $1.302 \times 10^{-3}$ \\
256 & 256 & 0.01 & 20 & $1.195 \times 10^{-3}$ & $1.443 \times 10^{-4}$ \\
512 & 512 & 0.005 & 40 & $4.172 \times 10^{-4}$ & $4.110 \times 10^{-5}$ \\
\hline
\end{tabular}
\caption{Numerical errors of the inner and outer boundary radii over the time interval $[0, T]$. The errors are computed as discrete $L^2$ norms in time, comparing the numerical results with the exact solution obtained from the ODE system~\eqref{eq1024_1}–\eqref{eq1024_6}.}
\label{tab1029}
\end{table}

\begin{figure}[htbp]
    \centering
    \begin{subfigure}[t]{0.48\textwidth}
        \centering
        \includegraphics[width=\textwidth]{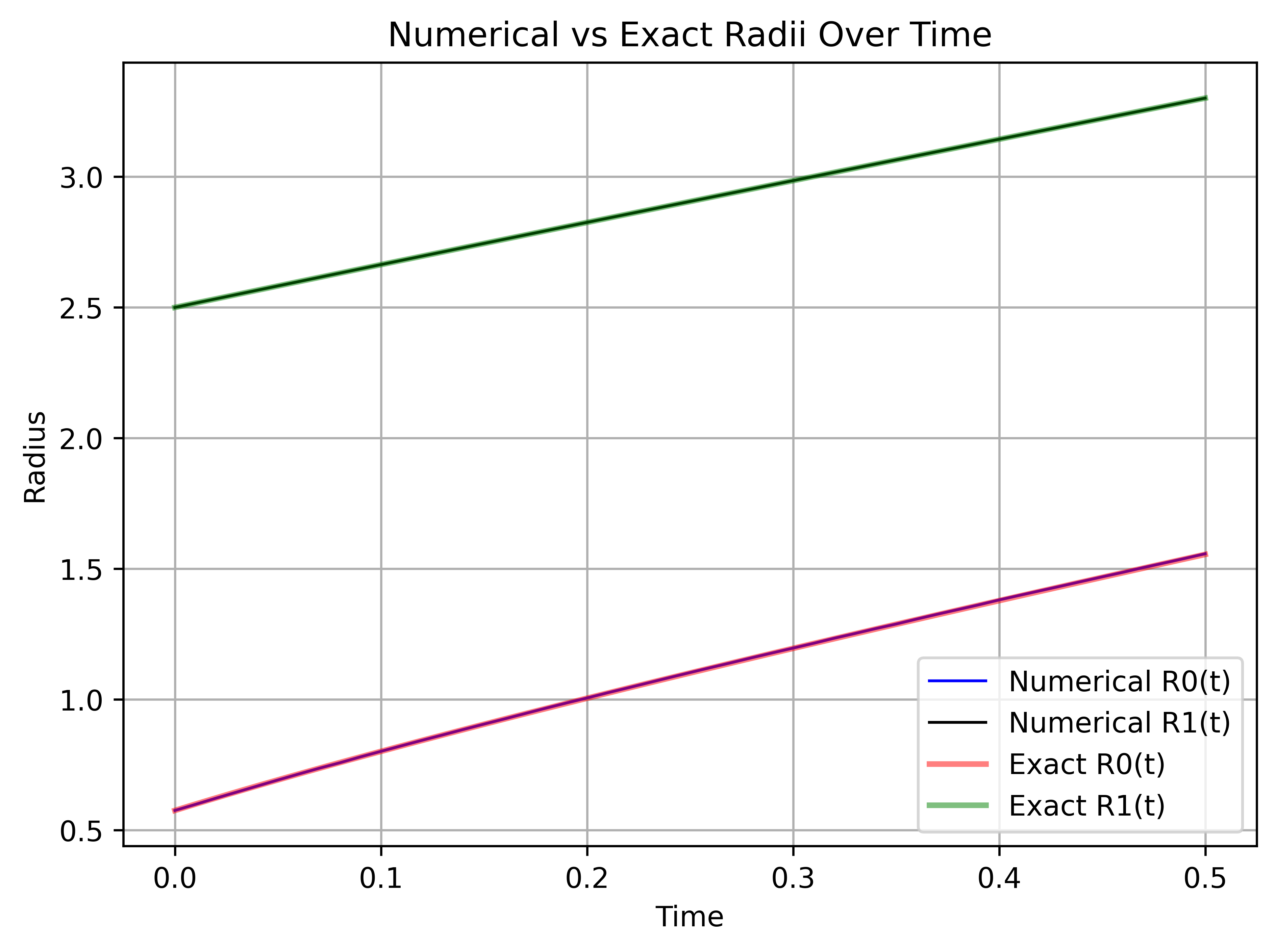}
        \caption{}
        \label{fig1026_3}
    \end{subfigure}
    \hfill
    \begin{subfigure}[t]{0.48\textwidth}
        \centering
        \includegraphics[width=\textwidth]{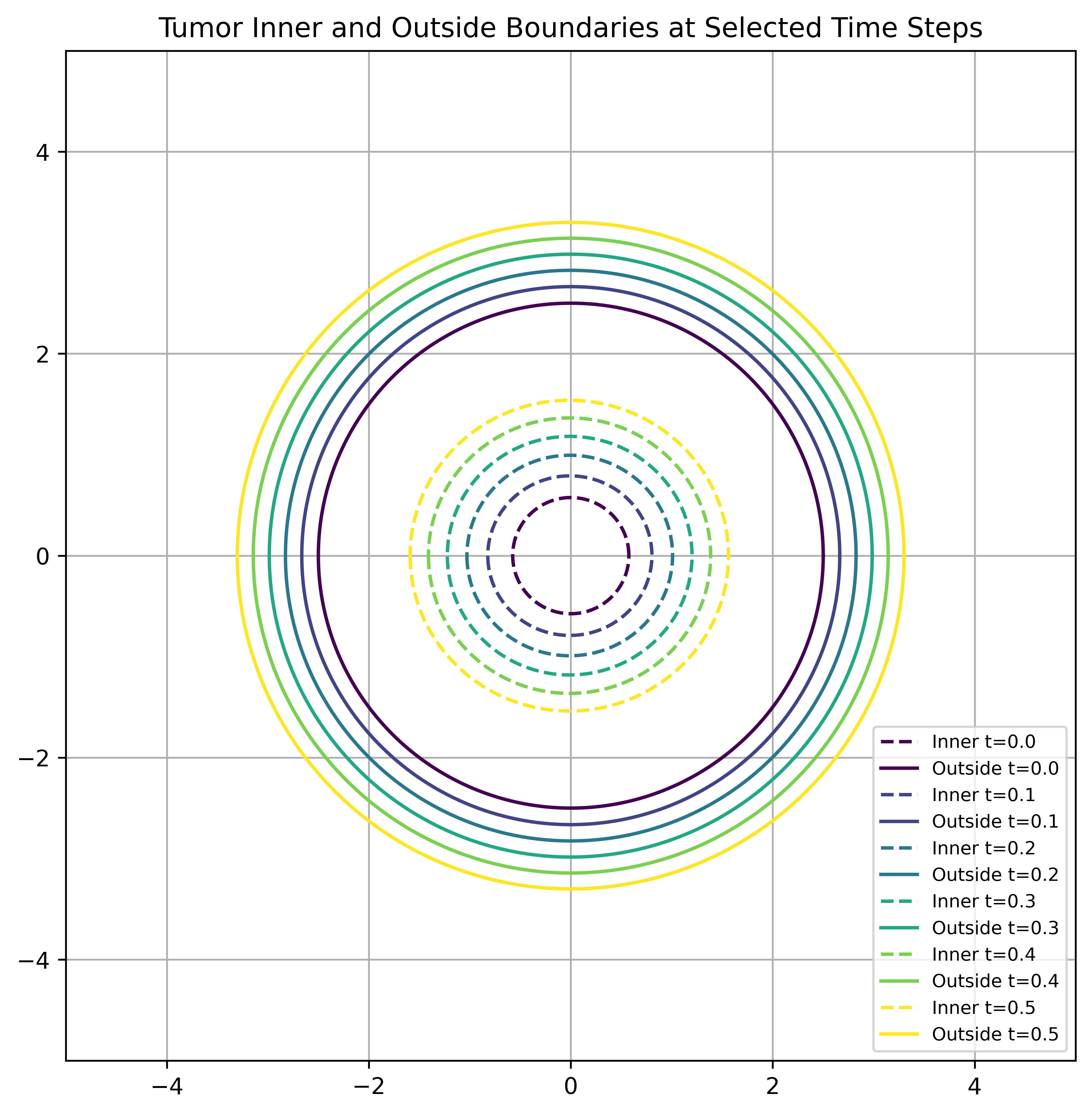}
        \caption{}
        \label{fig1026_4}
    \end{subfigure}
    \caption{Evolution of the inner and outer tumor boundaries and comparison with the exact radii for $I = J = 512$ and $\Delta t = 0.01$.}
    \label{fig1026}
\end{figure}

\subsection{Numerical Observation of Necrotic Core Emergence}
\label{Numerical Observation of Necrotic Core Emergence}
In this section, we investigate the tumor growth dynamics starting from a non-necrotic initial configuration, $N(0)=\emptyset$. The objective is to numerically capture the entire process by which a necrotic core emerges, develops, and eventually enters a stable growth regime. 

This scenario presents a numerical challenge: the governing equations change qualitatively, and a new moving interface nucleates during the simulation. To address this topological transition, we adopt an adaptive, stage-wise simulation strategy:

\medskip 
\noindent \textbf{Stage 1: Fully Viable Phase ($N(t)=\emptyset$).} 
Initially, the interface problem degenerates into the standard elliptic equation \eqref{eq0915_2}. In this phase, we verify the positivity of the pressure solution at each time step. We solve the Poisson equation using the BI \& KFBI Solver in \ref{Poisson Equations}. If the solution remains strictly positive, the tumor remains fully viable and we proceed with the standard forward Euler evolution. If the solution becomes non-positive at any grid point, it indicates the onset of necrosis; we then discard the current step and switch to the obstacle solver (Solver~\ref{alg0914}) to locate the nucleating core, transitioning to Stage 2.

\medskip 
\noindent \textbf{Stage 2: Nucleation Phase ($N(t)\neq\emptyset$ but unresolved).} 
Immediately following emergence, the necrotic core is physically non-empty but may be too small to be resolved by the boundary tracking algorithm (specifically, too few grid points fall within the core to reliably reconstruct a spline). During this transient regime, we solve the obstacle problem to update the pressure but evolve the outer boundary without explicitly tracking the inner interface. We effectively treat the core as a point sink until it grows sufficiently large to be numerically resolved. To ensure accuracy during this sensitive phase, the time step $\Delta t$ is reduced by a factor of 5.

\medskip 
\noindent \textbf{Stage 3: Developed Phase ($N(t)\neq\emptyset$ and resolvable).} 
Once the necrotic region is sufficiently large—defined in our experiments as containing more than 10 grid points that allow for a Fourier fit for the inner boundary with error less than 0.1—we switch to the full two-interface coupled scheme described in Section~\ref{The in vitro Model with Necrotic Core}.

\medskip
\noindent \textbf{Example 5 (Radially symmetric emergence of a necrotic core).} 
This example serves as a benchmark test to validate the proposed numerical strategy against an exact analytical description. 
We consider the model~\eqref{eq0916_1}–\eqref{eq0916_3} with an initially circular outer boundary of radius $R_1^{(0)}$ and no necrotic core, that is, $R_0^{(0)}=0$. 
As shown in \ref{appendix:exact_R0R1}, radial symmetry is preserved throughout the evolution, allowing for an exact characterization of the boundary dynamics.

Before the appearance of the necrotic core, the evolution of the outer boundary is governed by
\begin{equation}
\label{eq1126_1}
\frac{d R_1}{d t}=\frac{G_0}{\lambda} \frac{c_B \sqrt{\lambda}}{I_0(\sqrt{\lambda} R_1)} I_1(\sqrt{\lambda} R_1)-\frac{R_1}{2} G_0 \bar{c}.
\end{equation}
This process continues until the necrotic core emerges. The emergence time $t^{**}$ is determined by the time required for $R_1$ to grow to $R^{**}$, where $R^{**}$ is defined by the equation
\begin{equation}
\label{eq1126_2}
-\frac{G_0}{\lambda}\,\frac{c_B}{I_0(\sqrt{\lambda}\,R^{**})} + \frac{G_0 c_B}{\lambda} - \frac{(R^{**})^2}{4}\, G_0\,\bar{c} = 0.
\end{equation}
The initial radius $R_1^{(0)}$ is chosen such that $R_1^{(0)}<R^{**}$, ensuring that the numerical simulation captures the entire transition from a non--necrotic state to a necrotic one. 
After the necrotic core appears, the evolution of the inner and outer boundaries follows the coupled system~\eqref{eq1024_1}–\eqref{eq1024_6}.

The parameters are set as $c_B = 10$, $\lambda = 1$, $G_0 = 1$, $n_c = 10^{-3}$, $\bar{c} = \tfrac{1}{2}c_B$, and $R_1^{(0)} = 2.2$. 
The computational domain is $\mathcal{B} = [-5, 5] \times [-5, 5]$, with $64$ control points on both boundaries. 
For $I = J = 512$ and $\Delta t = 0.01$, the numerical results are shown in Figure~\ref{fig1126_12}. 
The results demonstrate that the proposed numerical strategy smoothly captures the emergence of the necrotic core and maintains excellent agreement with the exact solution throughout the evolution.

\begin{figure}[htbp]
    \centering
    \begin{subfigure}[t]{0.48\textwidth}
        \centering
        \includegraphics[width=\textwidth]{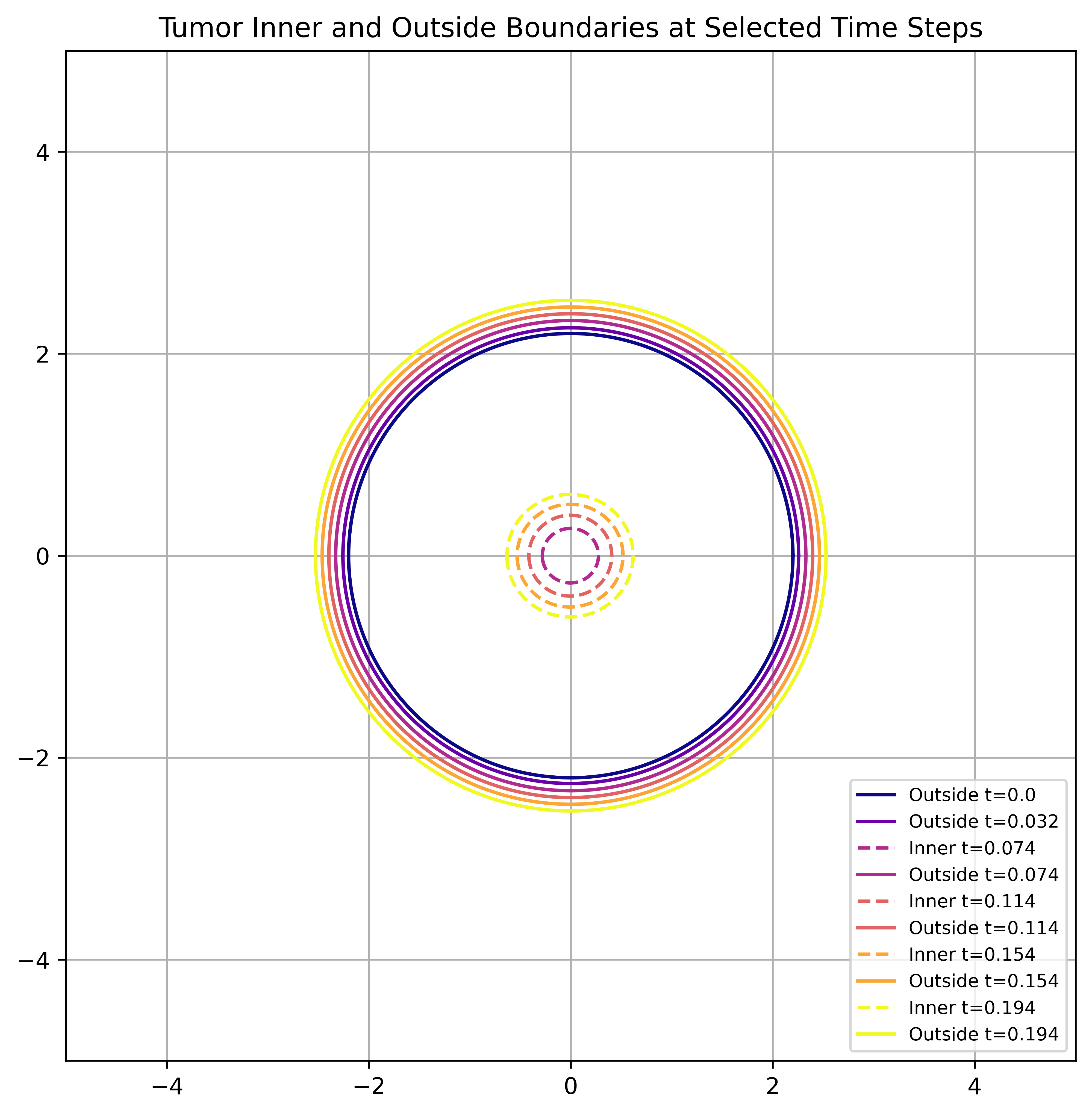}
        \caption{}
        \label{fig1126_1}
    \end{subfigure}
    \hfill
    \begin{subfigure}[t]{0.48\textwidth}
        \centering
        \includegraphics[width=\textwidth]{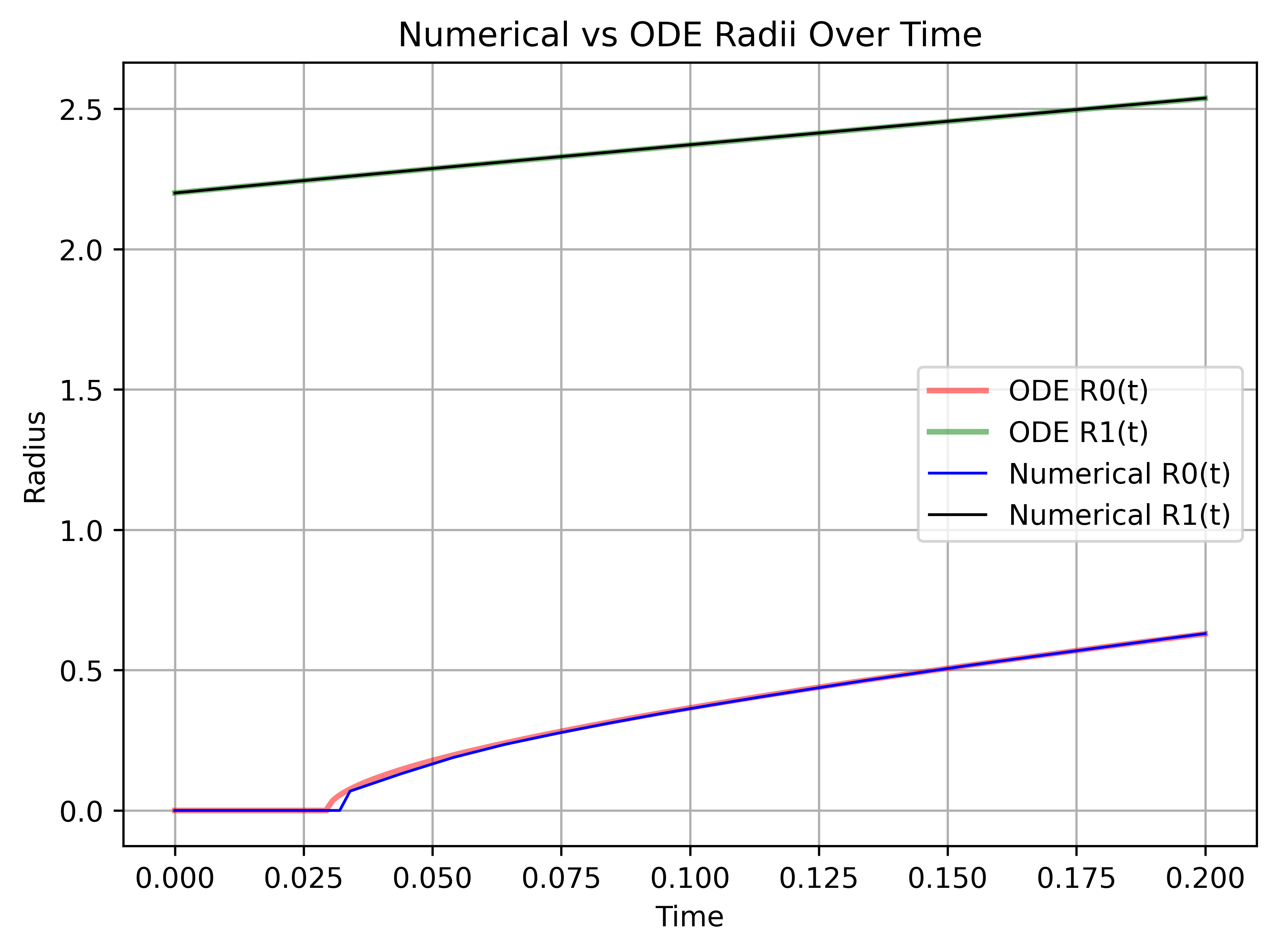}
        \caption{}
        \label{fig1126_2}
    \end{subfigure}
    \caption{Evolution of the inner and outside tumor boundaries and the comparison of the numerical and exact tumor boundary radii. For the numerical solutions, the radii of the inner and outside boundaries are computed for the control points in the same manner as described in Table \ref{tab1020_1}. The exact radius are obtained by solving the ODEs using the same ODE solver outlined in Table \ref{tab1020_1}.}
    \label{fig1126_12}
\end{figure}

\medskip
\noindent \textbf{Example 6 (Robustness under non-circular perturbations).} 
This example examines the robustness of the proposed method when the initial tumor boundary is non-circular. The initial outer boundary is given by
\[
\partial D(0) = \{(x, y): x = r(\theta)\cos\theta,\ y = r(\theta)\sin\theta,\ r(\theta) = 2.2 + 0.02\cos(4\theta)\}.
\]
The model parameters are set as $c_B = 0.5$, $\lambda = 1.4$, $G_0 = 10$, $n_c = 10^{-3}, \bar{c} = \tfrac{1}{2}c_B$. The computational domain $\mathcal{B} = [x_{\text{min}}, x_{\text{max}}] \times [y_{\text{min}}, y_{\text{max}}] = [-5, 5] \times [-5, 5]$ and the number of control points for inner and outside boundaries are both equal to $64$. For $I = J = 256, \Delta t = 0.02$, we obtain the numerical results as shown in figure \ref{fig1126_3}. From the figure, we observe that under our parameter settings, although the initial outer boundary is no longer a perfect circle, the necrotic core still evolves while maintaining a circular shape.

\medskip
\noindent \textbf{Example 7 (Comparison with the non-necrotic model).} This example compares the influence of two rate functions on the evolution of tumor boundaries. In this example, the initial tumor boundary is an ellipse with semi-axes $2.3$ and $1.2$, and $N(0)=\emptyset$. The model parameters are set as $c_B = 10$, $\lambda = 1$, $G_0 = 1$, $n_c = 10^{-3}, \bar{c} = \tfrac{1}{2}c_B$. The computational domain $\mathcal{B} = [x_{\text{min}}, x_{\text{max}}] \times [y_{\text{min}}, y_{\text{max}}] = [-5, 5] \times [-5, 5]$ and the number of control points for inner and outside boundaries are both equal to $64$. For $I = J = 256, \Delta t = 0.02$, we obtain the numerical results as shown in figure \ref{fig1126_4}. Note that the parameter choices of $c_B$, $\lambda$, and $G_0$ here are exactly the same as those used in \textbf{Example 2} for the simulation shown in Figure~\ref{fig1020_6}. However, due to the different choice of the rate function $G(c)$, the outcomes shown in Figures~\ref{fig1020_6} and~\ref{fig1126_4} are completely different. When $G(c)$ is chosen as $G_0 (c - \bar{c})$, the model develops a necrotic core, and the tumor boundary cannot easily return to a circular shape in a short time, forming a sharp contrast with the behavior observed in Figure~\ref{fig1020_6}.

\begin{figure}[htbp]
    \centering
    \begin{subfigure}[t]{0.48\textwidth}
        \centering
        \includegraphics[width=\textwidth]{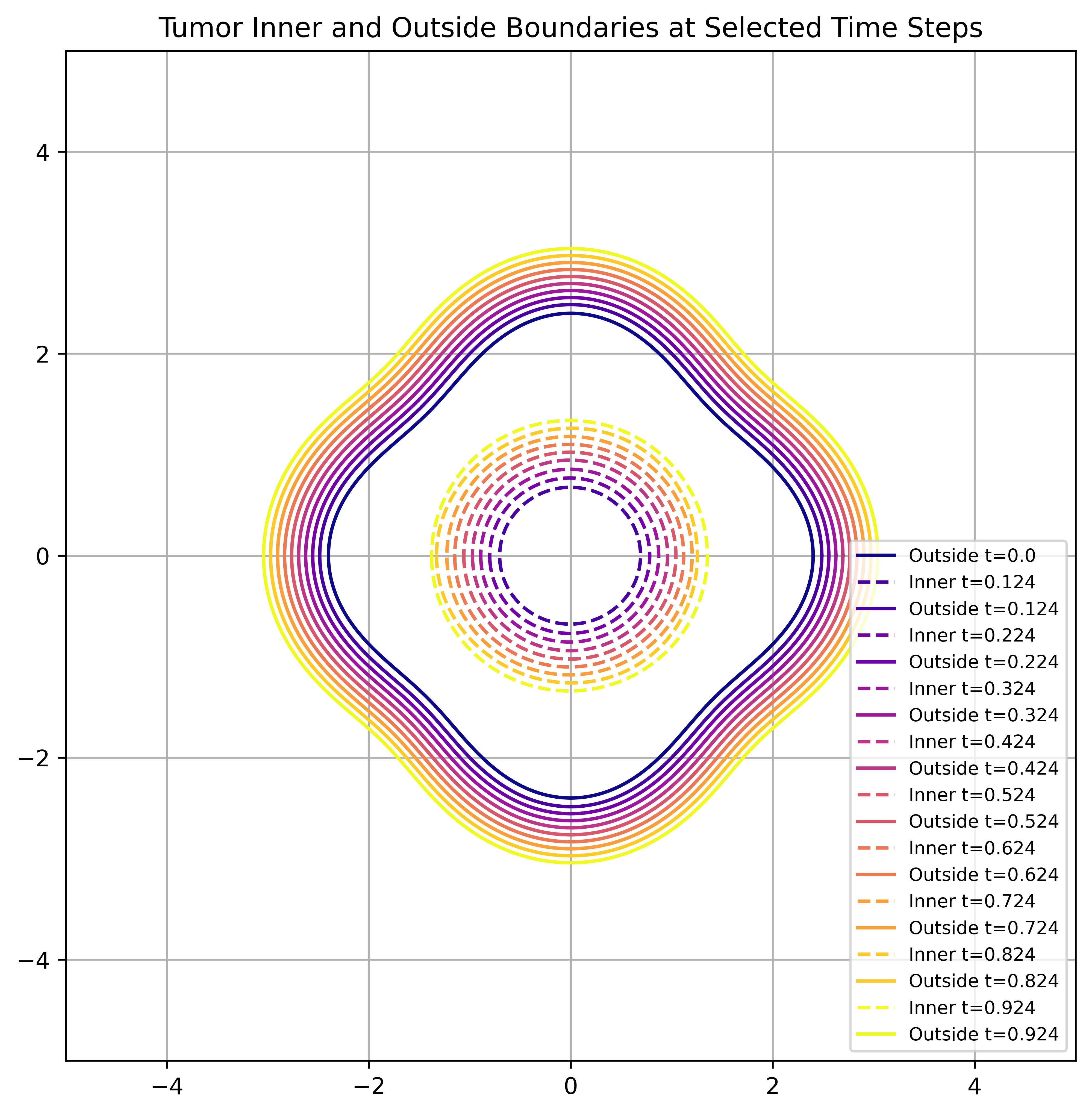}
        \caption{Evolution of boundaries in Example~6.}
        \label{fig1126_3}
    \end{subfigure}
    \hfill
    \begin{subfigure}[t]{0.48\textwidth}
        \centering
        \includegraphics[width=\textwidth]{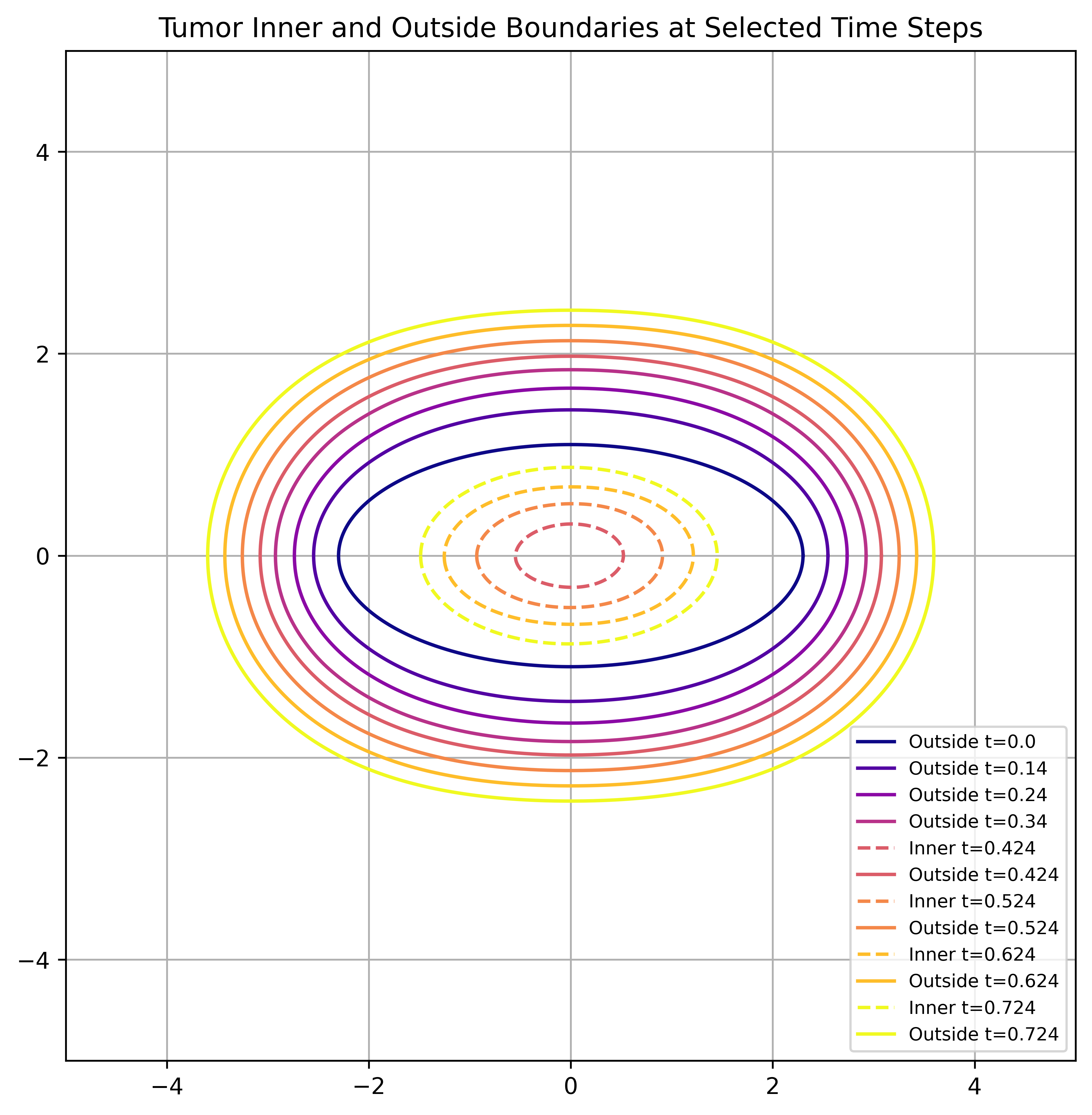}
        \caption{Evolution of boundaries in Example~7.}
        \label{fig1126_4}
    \end{subfigure}
    \caption{Numerical results for the evolution of the inner and outer tumor boundaries starting from non-necrotic initial configurations.}
    \label{fig1126_34}
\end{figure}

\section{Conclusion} \label{Conclusion}

In this work, we have developed a comprehensive numerical framework for simulating tumor growth dynamics governed by Hele-Shaw type equations. Utilizing the Boundary Integral (BI), Kernel-Free Boundary Integral (KFBI) and augmented Lagrangian active set methods as efficient underlying solvers on Cartesian grids, we constructed a robust iterative algorithm capable of simulating the complex interplay between nutrient diffusion, pressure fields, and domain evolution.

The core innovation of this study lies in the algorithmic strategy for the simultaneous evolution of two distinct types of free boundaries. The tumor entails an outer boundary driven explicitly by Darcy's velocity law, and an inner necrotic boundary defined implicitly as the free boundary of an obstacle problem. To accurately capture the latter, we introduced a predictor-corrector strategy specifically for the inner necrotic boundary. By iteratively updating the potential fields based on predicted locations of the inner interface and subsequently correcting its position, this strategy effectively stabilizes the numerical evolution, preventing the instabilities typically associated with tracking implicit free boundaries.

We have provided a rigorous error analysis for the model without necrotic core, establishing second-order accuracy for the potential solvers and first-order convergence for the boundary evolution. For the more complex model with a necrotic core, the accuracy and stability of the coupled scheme are corroborated by benchmark simulations against analytical solutions for radially symmetric configurations. Furthermore, the algorithm demonstrates exceptional robustness in handling topological transitions, seamlessly capturing the nucleation process where the necrotic core emerges from a fully viable tumor.

Looking forward, this computational framework paves the way for deeper investigations into tumor morphology. A primary short-term goal is to incorporate the more complex \textit{in vivo} nutrient model, where heterogeneous nutrient supply may trigger richer geometric dynamics. We plan to use this solver to explore and validate the mechanisms underlying boundary instabilities, such as the onset of fingering patterns. Additionally, given that the underlying KFBI method naturally avoids surface meshing, extending this algorithm to three-dimensional simulations remains a promising direction for future research.

\section*{Acknowledgements}
YF is supported by the National Key R\&D Program of China, Project Number 2021YFA1001200, the NSFC Youth program, Grant Number 12501669, the NSFC  International Collaboration Fund for Creative Research Teams, Grant Number W2541005.
WY is supported by the National Natural Science Foundation of China in the Division of Mathematical Sciences (Project No. 12471342) and the fundamental research funds for the central universities. 
ZZ is supported by the National Key R\&D Program of China, Project Number 2021YFA1001200, Zhejiang Provincial Natural Science Foundation of China, Project Number QKWL25A0501, and Fundamental and Interdisciplinary Disciplines Breakthrough Plan of the Ministry of Education of China, Project Number JYB2025XDXM502.

\newpage

\appendix

\section{Numerical Solvers} \label{Numerical Solvers}
This section presents how to apply the boundary integral (BI) method and the kernel-free boundary integral (KFBI) method to solve the Poisson equation, the modified Helmholtz equation, as well as one kind of interface problem of modified Helmholtz type. In addition, we demonstrate how to employ the Primal–Dual Active Set Algorithm to solve a unilateral obstacle problem.

The numerical methods summarized in this Section are designed to efficiently handle different classes of elliptic problems while sharing a common implementation framework. 
The boundary integral (BI) method reduces boundary value problems to boundary-only formulations, leading to a significant reduction in dimensionality and allowing for accurate treatment of complex geometries. 
Its main limitations lie in the reliance on fundamental solutions, the dense system matrices, and the need for special quadrature techniques for nearly singular integrals. 
The kernel-free boundary integral (KFBI) method alleviates these difficulties by reformulating boundary and volume potentials as equivalent simple interface problems on Cartesian grids, resulting in linear systems with standard discretization matrices that can be efficiently solved using fast algorithms such as FFT-based solvers, while avoiding explicit evaluation of Green’s functions. 
For obstacle problems, the Primal–Dual Active Set Algorithm provides an efficient and robust approach by transforming the constrained minimization problem into a sequence of linear systems associated with active and inactive sets, and typically exhibits fast, finite-step convergence. Notably, all three methods can be implemented effectively on Cartesian grids, which facilitates a unified numerical framework and enables the use of efficient solvers for free boundary problems.

\subsection{Modified Helmholtz Equations} \label{Modified Helmholtz Equations}
In this section, we introduce how to apply the boundary integral (BI) method \cite{hsiao2021boundary, atkinson1997numerical, kress1989linear} to solve the modified Helmholtz equations.

Let $\Omega \subset \mathbb{R}^2$ be a bounded domain with smooth boundary. Let $g(\mathbf{x})$ be a smooth function defined on the domain boundary $\Gamma := \partial \Omega$, $\kappa > 0$ be a coefficient. Consider the following boundary value problem
\begin{empheq}[left=\empheqlbrace]{align}
\Delta u - \kappa u = 0 \quad & \text {in}\ \Omega, \label{eq0911_1}\\
u = g \quad & \text {on}\ \partial \Omega. \label{eq0911_2}
\end{empheq}

Let $G_0(\mathbf{x}, \mathbf{y})$ be the fundamental solution of the elliptic operator $(\Delta - \kappa)$ on $\mathbb{R}^2$ that satisfies
\begin{equation} \label{eq0911_3}
\Delta_y G_0(\mathbf{x}, \mathbf{y})-\kappa G_0(\mathbf{x}, \mathbf{y}) =\delta(\mathbf{x}-\mathbf{y}), \quad \mathbf{y} \in \mathbb{R}^2,
\end{equation}
for each $\mathbf{x} \in \mathbb{R}^2$, where $\delta$ is the Dirac delta function in $\mathbb{R}^2$. For clarity, we denote the integral operators $W$, known as the double layer potentials, as
\begin{equation} \label{eq0911_4}
W[\varphi](\mathbf{x}) = \int_{\Gamma} \frac{\partial G_0(\mathbf{x}, \mathbf{y})}{\partial \mathbf{n}_{\mathbf{y}}} \, \varphi(\mathbf{y}) \, ds_{\mathbf{y}}, 
\end{equation}
where $\varphi$ is a density function defined on $\Gamma$.

The indirect boundary integral equation of problem \eqref{eq0911_1}-\eqref{eq0911_2} is given by:
\begin{equation} \label{eq0911_12}
\frac{1}{2} \varphi + W[\varphi] = g \quad \text{on}\ \Gamma.    
\end{equation}
And the solution $u$ can be obtained from $\varphi$ that
\begin{equation} \label{eq0911_13}
u(\mathbf{x}) = W[\varphi](\mathbf{x}),\quad \forall \mathbf{x} \in \Omega.    
\end{equation}

The boundary integral equation can be discretized by a Nyström method with uniformly distributed boundary nodes and composite quadrature rules \cite{atkinson1997numerical}. 
This leads to a linear system of the form
\begin{equation}
\frac12 \varphi_h + \mathbf{A}_h \varphi_h = \mathbf{g}_h,
\end{equation}
which is solved by an iterative solver. The solution inside the domain is then reconstructed via the double-layer potential. For any $\mathbf{x} \in \Omega$, 
\begin{equation} \label{eq0911_20}
u_h(\mathbf{x})=\sum_{j=0}^{M-1} h_j \frac{\partial G_0\left(\mathbf{x}_j ; \mathbf{x}\right)}{\partial \mathbf{n}_{\mathbf{x}_j}}\left|\mathbf{x}^{\prime}\left(t_j\right)\right| \varphi_j.
\end{equation}
However, when the target point $\mathbf{x}$ is not on but close to the boundary, the boundary integral \eqref{eq0911_13} is a nearly singular integral. A method proposed in \cite{beale2001method}, which combines singularity subtraction with a smoothing technique, can be applied here and achieves third-order accuracy. Specifically, if an accurate $\varphi$ is obtained, the error satisfies
\[
|u_h(\mathbf{x}) - u(\mathbf{x})| = \mathcal{O}(h^3),
\]
where $u_h(\mathbf{x})$ is computed using this method, and $u(\mathbf{x})$ is the exact solution given by the integral \eqref{eq0911_13}.

\begin{remark} \label{rmk_convergence_BI}
The rate at which $||\varphi - \varphi_h||_\infty$ tends to zero can be determined from results on the convergence of the quadrature rule used for the integral in boundary integral equations \cite[see Chapter~4.1.1]{atkinson1997numerical}. In general, the composite trapezoidal rule attains second-order accuracy. When applied to periodic functions, its convergence rate may be considerably higher; nevertheless, it is reasonable to adopt the conservative estimate of second order:
\[
\|\varphi-\varphi_h\|_\infty \leq \mathcal{C}_1 h^2,\quad \text{for some }\mathcal{C}_1 > 0.
\]
After substituting $\varphi_h$ into the reconstruction formula \eqref{eq0911_20}, one typically obtains the same order of accuracy for the PDE solution at points $\mathbf{x}$ located away from the boundary,
\[
|u(\mathbf{x})-u_h(\mathbf{x})| \leq \mathcal{C} h^2,\quad \text{for some }\mathcal{C} > 0.
\]
It can be expected that
\[
\|u-u_h\|_{L^\infty(K)} = O(h^2)\quad\text{for any compact }K\subset\Omega,
\]
provided that singularity subtraction or other appropriate treatments (as stated at the end of the preceding paragraph) are applied in the quadrature whenever necessary.
\end{remark}

\subsection{Poisson Equations} \label{Poisson Equations}
In this section, we introduce how to apply the boundary integral (BI) method \cite{hsiao2021boundary, atkinson1997numerical, kress1989linear} together with kernel-free boundary integral (KFBI) method \cite{ying2007kernel} to solve the Poisson equations.

Let $\Omega \subset \mathbb{R}^2$ be a bounded domain with smooth boundary. Let $g(\mathbf{x})$ be a smooth function defined on the domain boundary $\Gamma := \partial \Omega$, $f(\mathbf{x})$ be a source term in $\Omega$. For later convenience, let $\mathcal{B}$ be a larger regular domain, which completely contains the domain $\Omega$. Consider the following boundary value problem
\begin{empheq}[left=\empheqlbrace]{align}
\Delta u = f \quad & \text {in}\ \Omega, \label{eq0912_1}\\
u = g \quad & \text {on}\ \partial \Omega. \label{eq0912_2}
\end{empheq}

Let $G_0(\mathbf{x}, \mathbf{y})$ denote the fundamental solution of the elliptic operator $\Delta$ in $\mathbb{R}^2$. It is well known that in two dimensions,
\[
G_0(\mathbf{x}, \mathbf{y}) = \frac{1}{2\pi} \ln|\mathbf{x}-\mathbf{y}|.
\]

Let $G(\mathbf{x}, \mathbf{y})$ be the Green’s function of the elliptic operator $\Delta$ in the rectangle $\mathcal{B}$, which satisfies
\begin{equation} \label{eq0912_4}
\begin{aligned}
\Delta_y G(\mathbf{x}, \mathbf{y}) &= \delta(\mathbf{x}-\mathbf{y}), && \mathbf{y} \in \mathcal{B}, \\
G(\mathbf{x}, \mathbf{y}) &= 0, && \mathbf{y} \in \partial\mathcal{B},
\end{aligned}
\end{equation}
for each fixed $\mathbf{x} \in \mathcal{B}$. Define a function $u_p$ on $\Omega$ by  
\begin{equation} \label{eq0912_5}
u_p(\mathbf{x}) = \int_{\Omega} G(\mathbf{x}, \mathbf{y}) f(\mathbf{y}) \, d\mathbf{y}. 
\end{equation}
It is straightforward to verify, by direct computation and the definition of the Green’s function $G$, that $u_p$ is a particular solution of equation \eqref{eq0912_1}. 
Consequently, problem \eqref{eq0912_1}-\eqref{eq0912_2} can be reduced to the following homogeneous problem:
\begin{empheq}[left=\empheqlbrace]{align}
\Delta u = 0 \quad & \text{in } \Omega, \label{eq0912_6}\\[3pt]
u = g - u_p|_{\partial \Omega} \quad & \text{on } \partial \Omega. \label{eq0912_7}
\end{empheq}
The problem \eqref{eq0912_6}-\eqref{eq0912_7} can be solved by exactly the same procedure as described in Section \ref{Modified Helmholtz Equations}, except that the explicit form of $G_0$ is different. Denote the solution of \eqref{eq0912_6}-\eqref{eq0912_7} by $\bar{u}$. The the solution of \eqref{eq0912_1}-\eqref{eq0912_2} is given by $u = u_p + \bar{u}$. 

Regarding the computation of the particular solution $u_p$, i.e., the volume integral \eqref{eq0912_5}, we adopt the same strategy as in the KFBI method \cite{ying2007kernel}. Specifically, to evaluate the volume integral $u_p$, one can consider the following equivalent simple interface problems
\begin{empheq}[left=\empheqlbrace]{align} 
\Delta u_p = \mathbf{F} & := \begin{cases}f & \text { in } \Omega \\ 0 & \text { in } \mathcal{B} \backslash \Omega \end{cases}, \label{eq110910_27} \\ 
[u_p] = 0 \quad & \text {on}\ \partial \Omega, \label{eq110910_28} \\ 
\left[\frac{\partial u_p}{\partial\mathbf{n}}\right] = 0 \quad & \text {on}\ \partial \Omega, \label{eq110910_29} \\ 
u_p = 0 \quad & \text {on}\ \partial \mathcal{B}.\label{eq110910_30} 
\end{empheq}
In the KFBI method, the solution of the simple interface problem \eqref{eq110910_27}--\eqref{eq110910_30} is mainly obtained by solving the discrete system arising from the corrected finite difference scheme using fast Fourier transformation (FFT), and the method achieves second-order accuracy. Detailed descriptions can be found in Section~\ref{5 The Kernel-Free Boundary Integral Method}.

\subsection{Interface Problems of the Modified Helmholtz Type} \label{More Challenging Interface Problems of the Modified Helmholtz Type}
In this section, we introduce how to apply the kernel-free boundary integral (KFBI) method to solve the interface problems of the modified Helmholtz type. 

\subsubsection{Interface Problems} \label{5 Interface Problems}
Let $\mathbf{x} = (x_1, x_2) \in \mathbb{R}^2$ denote a point in the two-dimensional space and $\Omega \subset \mathbb{R}^2$ be a simply connected bounded domain with smooth boundary $\Gamma_1 := \partial \Omega$. Suppose that $\Gamma_0 \subset \Omega$ is a smooth interface which separates $\Omega$ into two disjoint subdomains $\Lambda$ and $\mathbf{N}$. A schematic illustration is provided in Figure \ref{fig0908_1}. For later convenience, let $\mathcal{B}$ be a larger regular domain, which completely contains the domain $\mathbf{N}$, as shown in Figure \ref{fig0908_2}.

\begin{figure}[htbp]
    \centering
    \begin{subfigure}{0.48\textwidth}
        \centering
        \includegraphics[width=\linewidth]{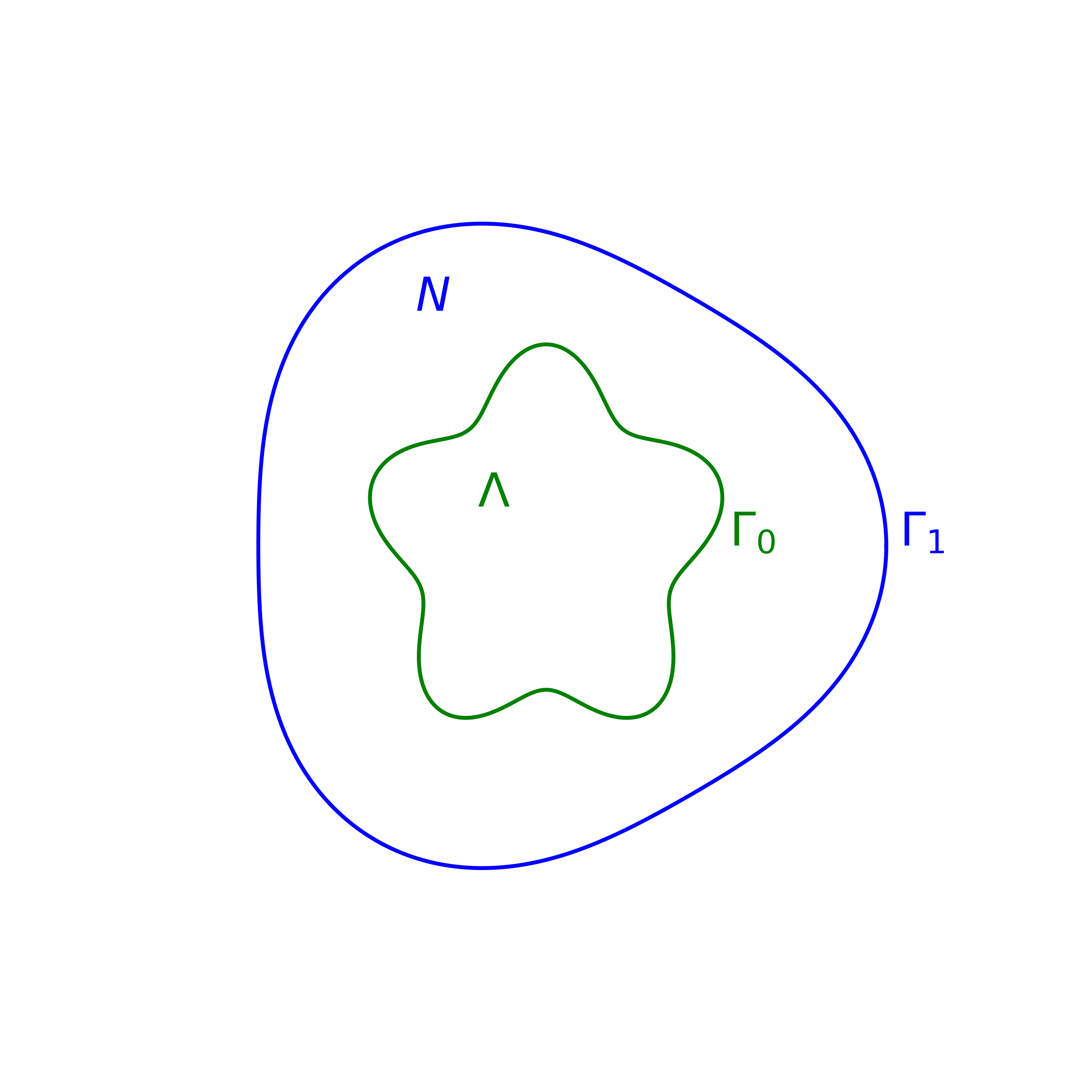}
        \caption{Two subdomains $\Lambda$ and $\mathbf{N}$.}
        \label{fig0908_1}
    \end{subfigure}
    \hfill
    \begin{subfigure}{0.48\textwidth}
        \centering
        \includegraphics[width=\linewidth]{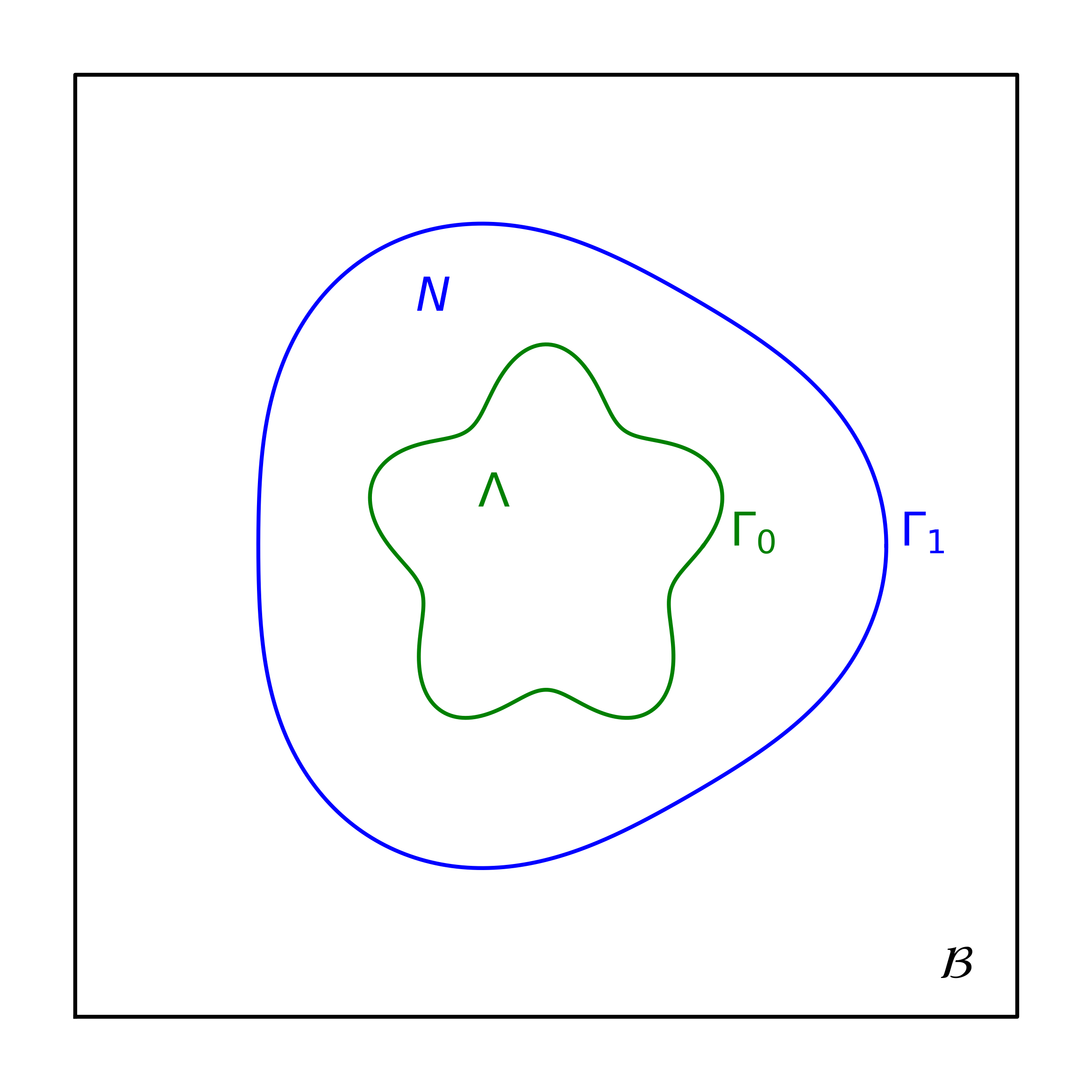}
        \caption{The extended domain $\mathcal{B}$ enclosing $\mathbf{N}$.}
        \label{fig0908_2}
    \end{subfigure}
    \caption{Schematic diagrams for domains of interface problems.}
    \label{fig0908_3}
\end{figure}

Consider the following interface problem
\begin{empheq}[left=\empheqlbrace]{align}
\Delta u - \kappa_i u = f_i \quad & \text {in}\ \Lambda, \label{eq0908_1}\\
\Delta u - \kappa_e u = f_e \quad & \text {in}\ \mathbf{N}, \label{eq0908_2}\\
[u] = g \quad & \text {on}\ \Gamma_0, \label{eq0908_3}\\
\left[\frac{\partial u}{\partial\mathbf{n}}\right] = j \quad & \text {on}\ \Gamma_0, \label{eq0908_4}\\
u = h \quad & \text {on}\ \Gamma_1. \label{eq0908_5}
\end{empheq}
Here, $\kappa_i, \kappa_e > 0$ are prescribed coefficients; $f_i, f_e, g, j, h$ denote sufficiently smooth given data; and $\mathbf{n}$ is the unit normal vector on $\Gamma_0$, oriented from $\Lambda$ toward $\mathbf{N}$.

\subsubsection{The Boundary Integral Equations} \label{5 The Boundary Integral Formulations}
Let $G_i(\mathbf{x}, \mathbf{y})$ be the Green’s function of the elliptic operator $(\Delta - \kappa_i)$ on the rectangle $\mathcal{B}$ that satisfies
\begin{equation} \label{eq0908_6}
\begin{aligned}
\Delta_y G_i(\mathbf{x}, \mathbf{y})-\kappa_i G_i(\mathbf{x}, \mathbf{y}) & =\delta(\mathbf{x}-\mathbf{y}), & & \mathbf{y} \in \mathcal{B}, \\
G_i(\mathbf{x}, \mathbf{y}) & =0, & & \mathbf{y} \in \partial \mathcal{B},
\end{aligned}
\end{equation}
for each $\mathbf{x} \in \mathcal{B}$. Let $G_e(\mathbf{x}, \mathbf{y})$ be the Green’s function of the elliptic operator $(\Delta - \kappa_e)$ on the rectangle $\mathcal{B}$ that satisfies
\begin{equation} \label{eq0908_7}
\begin{aligned}
\Delta_y G_e(\mathbf{x}, \mathbf{y})-\kappa_e G_e(\mathbf{x}, \mathbf{y}) & =\delta(\mathbf{x}-\mathbf{y}), & & \mathbf{y} \in \mathcal{B}, \\
G_e(\mathbf{x}, \mathbf{y}) & =0, & & \mathbf{y} \in \partial \mathcal{B},
\end{aligned}
\end{equation}
for each $\mathbf{x} \in \mathcal{B}$. For clarity, we denote the integral operators 
$W^0_i$, $W^0_e$, $W^1_i$, and $W^1_e$, known as the double layer potentials, as
\begin{equation} \label{eq0908_8}
W_a^k[\varphi](\mathbf{x}) = \int_{\Gamma_k} \frac{\partial G_a(\mathbf{x}, \mathbf{y})}{\partial \mathbf{n}_{\mathbf{y}}} \, \varphi(\mathbf{y}) \, ds_{\mathbf{y}}, 
\quad a = i, e, \quad k = 0, 1,
\end{equation}
where $\varphi$ is a density function defined on $\Gamma_0$ or $\Gamma_1$. Similarly, the integral operators 
$V^0_i$, $V^0_e$, $V^1_i$, and $V^1_e$, referred to as the single layer potentials, are defined by
\begin{equation} \label{eq0908_9}
V_a^k[\psi](\mathbf{x}) = \int_{\Gamma_k} G_a(\mathbf{x}, \mathbf{y}) \, \psi(\mathbf{y}) \, ds_{\mathbf{y}}, 
\quad a = i, e, \quad k = 0, 1,
\end{equation}
where $\psi$ is a density function on $\Gamma_0$ or $\Gamma_1$. And the integral operators $\mathcal{Y}_i$, $\mathcal{Y}_e$, referred to as the Yukawa potentials (or volume integral), are defined by
\begin{equation} \label{eq0908_10}
\mathcal{Y}_i[\rho_i](\mathbf{x}) = \int_{\Lambda} G_i(\mathbf{x}, \mathbf{y}) \, \rho_i(\mathbf{y}) \, d \mathbf{y}, 
\end{equation}
\begin{equation} \label{eq0908_11}
\mathcal{Y}_e[\rho_e](\mathbf{x}) = \int_{\mathbf{N}} G_e(\mathbf{x}, \mathbf{y}) \, \rho_e(\mathbf{y}) \, d \mathbf{y}, 
\end{equation}
where $\rho_i$ is a density function on $\Lambda$ and $\rho_e$ is a density function on $\mathbf{N}$.

Following standard potential theory \cite{hsiao2021boundary, steinbach2008numerical}, the interface problem can be reduced to a system of boundary integral equations 
\begin{equation} \label{eq0908_23_}
\begin{aligned}
& \varphi - (W_i^0[\varphi])^+ + (W_e^0[\varphi])^+ - (- V_i^0[\psi])^+ + (-V_e^0[\psi])^+ - (-V_e^1[\phi])^+ \\
= & (W_e^0[g])^+ + (- V_i^0[j])^+ + (W_e^1[h])^+ + (\mathcal{Y}_i[f_i])^+ + (\mathcal{Y}_e[f_e])^+,
\end{aligned}
\end{equation}
\begin{equation} \label{eq0908_24_}
\begin{aligned}
& \psi - (\frac{\partial}{\partial \mathbf{n}}(W_i^0[\varphi]))^+ + (\frac{\partial}{\partial \mathbf{n}}(W_e^0[\varphi]))^+ - (\frac{\partial}{\partial \mathbf{n}}(-V_i^0[\psi]))^+ 
+ (\frac{\partial}{\partial \mathbf{n}}(-V_e^0[\psi]))^+ - (\frac{\partial}{\partial \mathbf{n}}(-V_e^1[\phi]))^+ \\
= & -j + (\frac{\partial}{\partial \mathbf{n}}(W_e^0[g]))^+ + (\frac{\partial}{\partial \mathbf{n}}(-V_i^0[j]))^+ + (\frac{\partial}{\partial \mathbf{n}}(W_e^1[h]))^+ + (\frac{\partial}{\partial \mathbf{n}}(\mathcal{Y}_i[f_i]))^+ + (\frac{\partial}{\partial \mathbf{n}}(\mathcal{Y}_e[f_e]))^+,
\end{aligned}
\end{equation}
\begin{equation} \label{eq0908_25_}
- (W_e^0 [\varphi])^+ - (- V_e^0 [\psi])^+ + (-V_e^1 [\phi])^+ = h - (W_e^0[g])^+ -(W_e^1[h])^+ - (\mathcal{Y}_e[f_e])^+.
\end{equation}
Note that in equations \eqref{eq0908_23_}--\eqref{eq0908_24_}, 
the superscript `$+$' of a function $\xi$ denotes
\[
\xi^+(\mathbf{x}) = \lim_{\substack{\mathbf{y}\in\Lambda \\ \mathbf{y}\to\mathbf{x}}} \xi(\mathbf{y}),\quad \forall \mathbf{x} \in \Gamma_0,
\]
whereas in equation \eqref{eq0908_25_}, 
the superscript `$+$' of a function $\eta$ denotes
\[
\eta^+(\mathbf{x}) = \lim_{\substack{\mathbf{y}\in\Omega \\ \mathbf{y}\to\mathbf{x}}} \eta(\mathbf{y}),\quad \forall \mathbf{x} \in \Gamma_1.
\] 
BIEs \eqref{eq0908_23_}-\eqref{eq0908_25_} are established for the unknown densities $\varphi, \psi$ and $\phi$, whose discretized form can be solved by the generalized minimal residual (GMRES) method \cite{saad1986gmres}. After that, the numerical solutions for $\varphi, \psi$ and $\phi$ are utilized to get the approximation of $u_i$ and $u_e$ according to the solution representation formula, i.e., 
\begin{align} \label{eq0908_26}
&W_i^0[\varphi] - W_e^0[\varphi] + W_e^0[g] - V_i^0[\psi]- V_i^0[j]   \\\nonumber
&\qquad+ V_e^0[\psi] - V_e^1[\phi] + W_e^1[h] + Y_i[f_i] + Y_e[f_e]= \begin{cases}u_i & \text { in } \Lambda, \\ u_e & \text { in } \mathbf{N}. \end{cases}
\end{align}

\begin{remark}
In Paper \cite{xie2023fourth}, there is a detailed derivation of the boundary integral equation for a similar interface boundary value problem.
\end{remark}

\subsubsection{The Kernel-Free Boundary Integral Approach} \label{5 The Kernel-Free Boundary Integral Method}
Based on the analysis in Section \ref{5 The Boundary Integral Formulations}, the interface problem \eqref{eq0908_1}-\eqref{eq0908_5} can be addressed by first solving the boundary integral equation system \eqref{eq0908_23_}-\eqref{eq0908_25_} and subsequently deriving its numerical solution from formula \eqref{eq0908_26}. With the KFBI method, integrals in the boundary integral equations \eqref{eq0908_23_}-\eqref{eq0908_25_} and solution representation formula \eqref{eq0908_26} are computed by solving equivalent and simple interface problems in the extended domain $\mathcal{B}$. To evaluate $W_a^k[\varphi]$, $V_a^k[\psi]$ and $\mathcal{Y}_a[f]$ where $a \in \{i, e\}, k \in \{0, 1\}$, one may consider the equivalent simple interface problems. For instance, to evaluate $W_i^0[\varphi]$, $V_i^0[\psi]$, and $\mathcal{Y}_i[f]$, one can consider the following equivalent simple interface problems
\begin{empheq}[left=\empheqlbrace]{align} \Delta v - \kappa_i v = \mathbf{F} \quad & \text {in}\ \mathcal{B}, \label{eq0910_27}\\ [v] = \Phi \quad & \text {on}\ \Gamma_0, \label{eq0910_28}\\ \left[\frac{\partial v}{\partial\mathbf{n}}\right] = \Psi \quad & \text {on}\ \Gamma_0, \label{eq0910_29}\\ v = 0 \quad & \text {on}\ \partial \mathcal{B}, \label{eq0910_30} \end{empheq}
where the source term $\mathbf{F}$ and the interface conditions $\Phi$ and $\Psi$ are specified differently for each potential. The precise configurations are summarized in Table~\ref{tab0910_1} according to the potential theory \cite{hsiao2021boundary, steinbach2008numerical}.
\begin{table}[htbp]
\centering
\begin{tabular}{cccc}
\hline
Potential integral & $\mathbf{F}$ & $\Phi$ & $\Psi$ \\
\hline
$W_{i}^0[\varphi]$ 
  & $\mathbf{F} \equiv 0$ 
  & $\Phi = \varphi$ 
  & $\Psi \equiv 0$ \\
$-V_{i}^0[\psi]$ 
  & $\mathbf{F} \equiv 0$ 
  & $\Phi \equiv 0$ 
  & $\Psi = \psi$ \\
$\mathcal{Y}_i [f]$ 
  & $\mathbf{F}=\left\{
      \begin{array}{ll}
      f & \text{in } \Lambda \\
      0 & \text{in } \mathcal{B} \setminus \Lambda
      \end{array}
      \right.$
  & $\Phi \equiv 0$ 
  & $\Psi \equiv 0$ \\
\hline
\end{tabular}
\caption{Specification of the source term $\mathbf{F}$ and the two interface conditions $\Phi$ and $\Psi$ 
for the simple interface problem \eqref{eq0910_27}--\eqref{eq0910_30}.} 
\label{tab0910_1}
\end{table}

The interface problem \eqref{eq0910_27}--\eqref{eq0910_30} is considered `simple' in the sense that the PDE involves only constant coefficients and zero Dirichlet boundary conditions, and can therefore be solved in a relatively natural way using the finite difference method (FDM). The procedure consists of the following three steps:
\begin{enumerate}
    \item \textbf{Initial discretization:} 
    Apply a standard five-point finite difference scheme to discretize the governing equation \eqref{eq0910_27} on the Cartesian grid, without taking the interface $\Gamma$ into account.
    
    \item \textbf{Interface correction:} 
    Modify the finite difference scheme near the interface to reduce the large local truncation errors. 
    This correction only affects the right-hand side of the linear system obtained in Step (1).

    \item \textbf{Linear system solution:} 
    Solve the resulting linear system using a fast Fourier transform (FFT)-based elliptic solver to obtain the numerical solution on the grid nodes.
\end{enumerate}

\begin{remark} \label{rmk_convergence_KFBI}
The numerical solution obtained from the above procedure for system \eqref{eq0910_27}-\eqref{eq0910_30} achieves second-order accuracy with respect to the exact solution, i.e.
\[
\left\| v_h - v \right\|_{\infty} 
:= \max_{1 \leq i \leq N} \left| v_h(\mathbf{p}_i) - v(\mathbf{p}_i) \right|
= \mathcal{O}\!\left(h^2\right).
\]
Please note that paper \cite[Theorem 6]{ying2007kernel} states an essentially second-order accuracy, but the analysis there is carried out for the finite element discretization. In our case, we use a finite difference discretization, and a similar approach can be applied to show that the convergence order here is second order.
\end{remark}

These three steps constitute the standard procedure for evaluating the integrals in the solution representation formula \eqref{eq0908_26}. The evaluation of $W_i^1[\varphi]$, $V_i^1[\psi]$, $W_e^0[\varphi]$, $V_e^0[\psi]$, $W_e^1[\varphi]$, $V_e^1[\psi]$ are similar and the only thing worth noting is that $$\mathcal{Y}_e[f_e] = \int_{\mathbf{N}} G_e(\mathbf{x}, \mathbf{y}) \, f_e(\mathbf{y}) \, d \mathbf{y} = \int_{\Omega} G_e(\mathbf{x}, \mathbf{y}) \, f_e(\mathbf{y}) \, d \mathbf{y} - \int_{\Lambda} G_e(\mathbf{x}, \mathbf{y}) \, f_e(\mathbf{y}) \, d \mathbf{y}.$$
 
To solve the BIEs \eqref{eq0908_23_}-\eqref{eq0908_25_} with the GMRES method, both Dirichlet and Neumann boundary data on $\Gamma$ of the involved integrals are needed during each iteration. In the perspective of the corresponding equivalent interface problem \eqref{eq0910_27}--\eqref{eq0910_30}, we are required to evaluate the boundary value of the solution $v^+$ as well as its normal derivative $\partial_{\mathbf{n}} v^+ = \mathbf{n} \cdot \nabla v^+$ on the interface $\Gamma$, which are both in the sense of one-sided limit from the subdomain $\Omega_i$. Based on the approximate grid solution derived by the three-step procedure, a quadratic polynomial interpolation is employed therewith.

Implementation details for the steps of discretization, correction and interpolation are presented in \cite{ying2007kernel}.

\subsection{Obstacle Problems} \label{Obstacle Problems}
In this section, we formulate the obstacle problems, for which the
solution algorithms are developed in \cite{karkkainen2003augmented}. We consider the following problem:
\begin{align}
\min_{u \in H_0^1(\Omega)}\; & a(u, u)-2 l(u), \label{eq0914_1} \\
\text{s.t.} \quad & u(x) \leq \psi(x) \text{ a.e. } x \in \Omega, \label{eq0914_2}
\end{align}
where $a$ is a symmetric, $H_0^1$-elliptic bilinear form, $l$ is a linear form, and the function $\psi$ is given in $\Omega$ with $-\infty < \psi < \infty$ a.e. in $\Omega$. Discretization of problem \eqref{eq0914_1}-\eqref{eq0914_2} leads to the finite-dimensional obstacle problem
\begin{equation} \label{eq0914_3}
\min _{u \in C^n} \mathscr{J}(u)=\langle A u, u\rangle-2\langle f, u\rangle,
\end{equation}
where $n$ is the number of the discrete points in $\Omega$,
\begin{equation} \label{eq0914_4}
C^n=\left\{u \in \mathbb{R}^n: u_i \leq \psi_i \text { for all } i=1, \ldots, n\right\},
\end{equation}
$f, \psi \in \mathbb{R}^n$, and $\langle\cdot, \cdot\rangle$ denotes the usual Euclidean inner product in $\mathbb{R}^n$. We omit the derivation and reproduce the Primal–Dual Active Set Algorithm proposed in \cite{karkkainen2003augmented} as Algorithm \ref{alg0914}. 
\begin{algorithm}[htbp]
\caption{Primal--Dual Active Set Algorithm for the Unilateral Obstacle Problem}
\label{alg0914}
\begin{algorithmic}[1]
\State \textbf{Input:} Initial guess $u^0$, $\lambda^0 \ge 0$, parameter $c>0$.
\State \textbf{Output:} Approximate solution $u$ of the obstacle problem.

\State Set $k = 0$. Denote $\mathbf{N} := \{1, 2, ...,n\}$.
\Repeat
    \State Compute 
    \[
    \hat{\lambda}^k = \max \{0, \lambda^k + c (u^k - \psi)\}.
    \]
    \State Determine the active and inactive sets:
    \[
    J^k = \{ j \in \mathbf{N} : \hat{\lambda}^k_j > 0 \} \quad (\text{active}), \qquad
    I^k = \{ i \in \mathbf{N} : \hat{\lambda}^k_i = 0 \} \quad (\text{inactive}).
    \]
    \If{$k \ge 1$ and $J^k = J^{k-1}$}
        \State \textbf{break}; the solution is $u^k$.
    \EndIf
    \State Solve the linear system
    \begin{equation} \label{eq0914_11}
    Au + \lambda = f, \qquad 
    \lambda = 0 \text{ on } I^k, \qquad
    u = \psi \text{ on } J^k.
    \end{equation}
    \State Update
    \[
    u^{k+1} = u, \qquad 
    \lambda^{k+1} = \max\{0, \lambda\}, \qquad
    k \gets k+1.
    \]
\Until{convergence of active set}

\State \textbf{Remark:} The main computational task \eqref{eq0914_11} can be reduced to solving
\[
A_{II} u_I = f_I - A_{IJ}\psi_J,
\qquad 
u_J = \psi_J,
\qquad 
\lambda = f - Au,
\]
for \(I = I^k\) and \(J = J^k\). Here, \(A_{IJ}\) denotes the submatrix of \(A\) formed by selecting the rows indexed by \(I\) and the columns indexed by \(J\).

\end{algorithmic}
\end{algorithm}

\section{Proofs of Lemmas}

\subsection{Proof of Lemma \ref{lem1}} \label{proof_lemma1}
\begin{proof}
To solve Poisson equation 
\begin{equation} \label{eq1208_1}
\begin{aligned}
-\Delta u & = f, \quad && \text{in } \Omega, \\
u & = 0, \quad &&\text{on } \partial \Omega,
\end{aligned}
\end{equation}
by BI method, a particular solution satisfying the first equation of it should be found firstly. For this, KFBI method can be used to obtain a particular solution
\begin{equation} \label{eq1208_2}
u_p (\mathbf{x}) = \int_\Omega G_\mathcal{B}(\mathbf{x}, \mathbf{y}) f(\mathbf{y}) d\mathbf{y}, 
\end{equation}
where $G_\mathcal{B}$ is the Green's function for $- \Delta$ and rectangle $\mathcal{B}$ with $\Omega \subset\mathcal{B}$. Then
\begin{equation} \label{eq1208_3}
\|u_p - u_{p, h}\|_\infty \leq C_1 h^2, \quad \exists C_1 > 0,
\end{equation}
Note that $u_{p, h}$ is the numerical result of equation \eqref{eq1208_2} by KFBI method on the Cartesian grid of $\mathcal{B}$, that is, the numerical solution of \eqref{eq110910_27}--\eqref{eq110910_30} in \ref{Poisson Equations}. And $u_p$ is understood as being restricted to the grid points. This convention is adopted throughout the paper. Since the volume integral \eqref{eq1208_2} is continuous on $\mathcal{B}$, a stable polynomial interpolation scheme can be employed to obtain $u_{p, h}|_{\partial \Omega}$ \cite{ying2007kernel}, such that
\begin{equation} \label{eq1208_4}
\|u_{p}|_{\partial \Omega} - u_{p, h}|_{\partial \Omega}\|_\infty \leq C_2 h^2, \quad \exists C_2 > 0.
\end{equation}
Suppose $u_b$ is the solution of this homogeneous equation
\begin{equation} \label{eq1208_5}
\begin{aligned}
-\Delta u_b & = 0, \quad && \text{in } \Omega, \\
u_b & = - u_{p}|_{\partial \Omega}, \quad &&\text{on } \partial \Omega,
\end{aligned}
\end{equation}
and $\widetilde{u_b}$ is the solution of this homogeneous equation
\begin{equation} \label{eq1208_6}
\begin{aligned}
-\Delta \widetilde{u_b} & = 0, \quad && \text{in } \Omega, \\
\widetilde{u_b} & = - u_{p, h}|_{\partial \Omega}, \quad &&\text{on } \partial \Omega.
\end{aligned}
\end{equation}
Note that in equation \eqref{eq1208_6}, $u_{p,h}|_{\partial \Omega}$ can be regarded as a function on the boundary $\partial \Omega$ obtained by extending the values of the numerical solution at the discrete points. However, in practical numerical computations, only the values at these discrete points are required. Consider the difference
\[
w := u_b - \widetilde{u_b}.
\]
Then \(w\) satisfies the Laplace equation:
\begin{equation} \label{eq1208_8}
\begin{aligned}
-\Delta w & = 0, \quad && \text{in } \Omega,\\
w & = - u_{p}|_{\partial\Omega} + u_{p,h}|_{\partial\Omega}, \quad && \text{on } \partial\Omega.
\end{aligned}
\end{equation}
By the maximum principle for harmonic functions (or equivalently the classical \(L^\infty\)-stability of the Dirichlet problem for Laplace's equation on bounded domains with sufficiently regular boundary), the maximum of \(|w|\) in \(\overline{\Omega}\) is attained on \(\partial\Omega\). Hence
\begin{equation} \label{eq1208_9}
\|u_b - \widetilde{u_b}\|_{\infty}
= \|w\|_{\infty}
\leq \|- u_{p}|_{\partial\Omega} + u_{p,h}|_{\partial\Omega}\|_{\infty}
\le C_2 h^2
\end{equation}
by \eqref{eq1208_4}. Then we have 
\begin{equation} \label{eq1208_10}
\|u_b - \widetilde{u_b}_{,h}\|_{\infty} 
\leq \|u_b - \widetilde{u_b}\|_{\infty} + \|\widetilde{u_b} - \widetilde{u_b}_{,h}\|_{\infty}
\leq C_2 h^2 + C_3 h^2, \quad \exists C_3 > 0,
\end{equation}
where $\widetilde{u_b}_{,h}$ is the numerical solution of equation \eqref{eq1208_6}. Note that the estimate $\|\widetilde{u_b} - \widetilde{u_b}_{,h}\|_{\infty} \leq C_3 h^2$ is given by the second-order accuracy of Nyström discretization of the BI method.

Combining above and using the decomposition
\[
u = u_p + u_b,\qquad u_h := u_{p,h} + \widetilde{u_b}_{,h},
\]
we obtain the uniform error bound
\begin{equation} \label{eq1208_11}
\begin{aligned}
\|u - u_h\|_{\infty}
&\le \|u_p - u_{p,h}\|_{\infty} + \|u_b - \widetilde{u_b}_{,h}\|_{\infty}\\
&\le C_1 h^2 + C_2 h^2 + C_3 h^2 = C_4 h^2,
\end{aligned}
\end{equation}
where \(C_4 := C_1 + C_2 + C_3\) is independent of \(h\).
\end{proof}

\subsection{Proof of Lemma \ref{lem2}} \label{proof_lemma2}
\begin{proof}
Let \(G_\Omega(\mathbf{x},\mathbf{y})\) be the Green's function for \(-\Delta\) on \(\Omega\) with homogeneous Dirichlet boundary condition. Then the exact solution \(u\) of \eqref{eqq1} and the perturbed-solution \(\widetilde{u}\) of \eqref{eq1208_12} admit the volume potential representations
\[
u(\mathbf{x})=\int_\Omega G_\Omega(\mathbf{x},\mathbf{y})\,f(\mathbf{y})\,d\mathbf{y},\qquad
\widetilde{u}(\mathbf{x})=\int_\Omega G_\Omega(\mathbf{x},\mathbf{y})\,\widetilde{f}(\mathbf{y})\,d\mathbf{y}.
\]
Therefore their difference is
\[
u(\mathbf{x})-\widetilde{u}(\mathbf{x})
=\int_\Omega G_\Omega(\mathbf{x},\mathbf{y})\bigl(f(\mathbf{y})-\widetilde{f}(\mathbf{y})\bigr)\,d\mathbf{y}.
\]
For each fixed \(\mathbf{x}\in\Omega\) define
\[
M_\Omega(\mathbf{x}) := \int_\Omega \bigl|G_\Omega(\mathbf{x},\mathbf{y})\bigr|\,d\mathbf{y}.
\]
It is classical that for bounded \(\Omega\) with reasonable boundary regularity one has \(M_\Omega(\mathbf{x})<\infty\) for all \(\mathbf{x}\in\Omega\) and
\[
M_\Omega := \sup_{\mathbf{x}\in\Omega} M_\Omega(\mathbf{x}) < \infty.
\]
Using the representation and the uniform bound above, we obtain the \(L^\infty\)-estimate
\begin{equation} \label{eq1208_14}
\begin{aligned}
\|u - \widetilde{u}\|_{\infty}
&= \sup_{\mathbf{x}\in\Omega} \left|\int_\Omega G_\Omega(\mathbf{x},\mathbf{y})\,(f-\widetilde{f})(\mathbf{y})\,d\mathbf{y}\right| \\
&\le \sup_{\mathbf{x}\in\Omega} \int_\Omega \bigl|G_\Omega(\mathbf{x},\mathbf{y})\bigr|\,|f-\widetilde{f}|(\mathbf{y})\,d\mathbf{y} \\
&\le M_\Omega \,\|f-\widetilde{f}\|_{\infty}.
\end{aligned}
\end{equation}
Combining \eqref{eq1208_13} with \eqref{eq1208_14} yields the desired second-order bound
\begin{equation} \label{eq1208_15}
\|u - \widetilde{u}\|_{\infty} \le C_6 \, h^2,\qquad C_6 := M_\Omega\,C_f.
\end{equation}

Denote the numerical solution (from the BI\&KFBI solver) for the perturbed equation \eqref{eq1208_12} by $\widetilde{u}_h$. Recall that the solver is second-order as we shown before, that is
\begin{equation} \label{eq1208_16}
\|\widetilde{u} - \widetilde{u}_h\|_{\infty} \le C_7 h^2, \quad \exists C_7 > 0.
\end{equation}
Hence, we can obtain the estimate
\begin{equation} \label{eq1208_17}
\|u - \widetilde{u}_h\|_{\infty} \le C_8 h^2, \quad C_8 := C_6 + C_7,
\end{equation}
where $u$ is the exact solution of Poisson problem \eqref{eqq1} and $\widetilde{u}_h$ is the numerical solution given by BI\&KFBI solver of the perturbed equation \eqref{eq1208_12}.
\end{proof}

\section[Detailed Version of Proposition 3.1]{The Detailed Version of Proposition \ref{prop1}} \label{appendix_prop1}

\begin{proposition}[Pointwise velocity error for quadratic reconstruction] \label{prop11}
This proposition provides a fully detailed and pointwise version of Proposition~\ref{prop1},
making all regularity, stencil geometry, and locality assumptions explicit.

Let $\Omega$ be a bounded domain with smooth boundary and let $p$ be the exact pressure. Fix a boundary point $\mathbf{X}\in\partial\Omega$ and suppose
\begin{enumerate}
  \item $p$ is $C^{3}$ in a neighborhood of $\mathbf{X}$.
  
  \item The numerical solver produces values $p_h$ at Cartesian grid nodes in a neighborhood $B(\mathbf{X},Kh)$ such that
  \[
  \max_{\mathbf{y}\in\mathcal{G}\cap B(\mathbf{X},Kh) \cap \Omega} |p_h(\mathbf{y}) - p(\mathbf{y})| \le C_p h^2,
  \]
  where $\mathcal{G}$ denotes the Cartesian grid and $K>0$ is fixed.
  
  \item The reconstruction stencil $\{\mathbf{y}_j\}_{j=1}^{N_s}\subset B(\mathbf{X},Kh)$ satisfies $N_s\ge 6$ and the following geometric nondegeneracy condition:
    \begin{itemize}
        \item There exists a constant $c_0>0$, independent of $h$ such that the convex hull of $\{\mathbf{y}_j\}$ contains a disk of radius $c_0 h$.
    
        \item Equivalently, writing the scaled coordinates
        \[
            \widetilde{\mathbf y}_j = \frac{\mathbf y_j - \mathbf X}{h},
        \]
        the scaled design matrix
        \[
        \widetilde A =
        \begin{pmatrix}
        1 & \widetilde x_1 & \widetilde y_1 &
          \widetilde x_1^{\,2} & \widetilde x_1\widetilde y_1 & \widetilde y_1^{\,2} \\
        \vdots & \vdots & \vdots & \vdots & \vdots & \vdots \\
        1 & \widetilde x_{N_s} & \widetilde y_{N_s} &
          \widetilde x_{N_s}^{\,2} & \widetilde x_{N_s}\widetilde y_{N_s} & \widetilde y_{N_s}^{\,2}
        \end{pmatrix}
        \]
        satisfies the uniform conditioning bound
        \[
            \lambda_{\min}(\widetilde A^\top \widetilde A) \ge \widetilde{c} > 0,
        \]
        where $\widetilde{c}$ is a constant independent of $h$.
    \end{itemize}
    
    This condition is experimentally satisfied on a uniform Cartesian grid, provided that the boundary $\partial\Omega$ has no geometric features too sharply curved to be resolved on the scale $Kh$.
\end{enumerate}

Denote by $\nabla p_{h,\mathrm{rec}}(\mathbf{X})$ the gradient obtained from the local quadratic reconstruction evaluated at the boundary point $\mathbf{X}$. Then the gradient obtained from the quadratic least--squares reconstruction based on the noisy data $p_h$ satisfies
\[
\bigl|\nabla p(\mathbf{X}) - \nabla p_{h,\mathrm{rec}}(\mathbf{X})\bigr| \le C h,
\]
for some constant $C$ independent of $h$ and the selection of reconstruction stencils. Consequently, the computed normal velocity satisfies
\[
|v(\mathbf{X}) - v_h(\mathbf{X})| \le C h.
\]
\end{proposition}

\begin{proof}
We present a matrix-based estimate separating truncation (Taylor) errors and data errors, and using a standard scaling argument to control the conditioning of the least--squares fit.

\medskip

\noindent\textbf{1. Local coordinates and quadratic model.}
Translate coordinates so that $\mathbf{X}$ is the origin and write local coordinates $\mathbf{x}=(x,y)$.  Consider the quadratic polynomial model
\[
q(\mathbf{x}) = a_0 + a_1 x + a_2 y + a_3 x^2 + a_4 xy + a_5 y^2,
\]
with coefficient vector \(a = (a_0,a_1,\dots,a_5)^\top\in\mathbb{R}^6\). The gradient at the origin is given by
\[
\nabla q(0) = (a_1,a_2)^\top.
\]

\medskip

\noindent\textbf{2. Design matrix and data.}
Let the reconstruction use $N_s$ nearby grid points $\mathbf{y}_j=(x_j,y_j)$, $j=1,\dots,N_s$, with $|x_j|,|y_j| = \mathcal{O}(h)$. Define the (unscaled) design matrix
\[
A := \begin{pmatrix}
1 & x_1 & y_1 & x_1^2 & x_1 y_1 & y_1^2 \\
\vdots & \vdots & \vdots & \vdots & \vdots & \vdots \\
1 & x_{N_s} & y_{N_s} & x_{N_s}^2 & x_{N_s} y_{N_s} & y_{N_s}^2
\end{pmatrix} \in\mathbb{R}^{N_s\times 6},
\]
and the observation vector $b\in\mathbb{R}^{N_s}$ with entries $b_j = p_h(\mathbf{y}_j)$.

The (noiseless) Taylor expansion of $p$ around the origin gives
\[
p(\mathbf{y}_j) = \Phi(\mathbf{y}_j) a^\star + r_j,
\]
where $\Phi(\mathbf{y}_j) := (1, x_j, y_j, x_j^2, x_j y_j, y_j^2)$, $a^\star$ are the exact quadratic coefficients for the second-order Taylor polynomial, and the remainder satisfies
\begin{equation} \label{eq1208_888}
|r_j| \le C_R h^3
\end{equation}
because $p\in C^3$ and $|\mathbf{y}_j|=\mathcal{O}(h)$.

The numerical (noisy) data are
\[
b_j = p_h(\mathbf{y}_j) = p(\mathbf{y}_j) + \delta_j
= \Phi(\mathbf{y}_j) a^\star + r_j + \delta_j,
\]
with the data error bound 
\begin{equation} \label{eq1208_999}
|\delta_j| \le C_p h^2
\end{equation}
by assumption. Write vector forms
\[
b = A a^\star + r + \delta,
\]
where $r=(r_j)_{j=1}^{N_s}$ and $\delta=(\delta_j)_{j=1}^{N_s}$.

\medskip

\noindent\textbf{3. Least--squares solution and coefficient error.}
The least--squares (normal equations) solution is
\[
\widehat a = (A^\top A)^{-1} A^\top b.
\]
Thus the coefficient error satisfies
\[
\widehat a - a^\star = (A^\top A)^{-1} A^\top (r + \delta).
\]

\medskip

\noindent\textbf{4. Scaling.}
The columns of $A$ have different $h$--scalings: the $k$-th column scales like $h^{\alpha_k}$ with $\alpha=(0,1,1,2,2,2)$. To isolate $h$ dependence, introduce the diagonal scaling matrix
\[
D := \operatorname{diag}\!\bigl(1,\; h,\; h,\; h^2,\; h^2,\; h^2\bigr),
\]
and write $A = \widetilde A D$, where the scaled design matrix $\widetilde A$ has \(j\)-th row
\[
\widetilde\Phi(\mathbf{y}_j) = \bigl(1,\; x_j/h,\; y_j/h,\; x_j^2/h^2,\; (x_j y_j)/h^2,\; y_j^2/h^2\bigr).
\]
Because $|x_j|,|y_j|=\mathcal{O}(h)$, the entries of $\widetilde A$ are $\mathcal{O}(1)$ and, by the stencil nondegeneracy assumption (item 3), the matrix $\widetilde A^\top \widetilde A$ has eigenvalues bounded above and below by positive constants independent of $h$. Consequently,
\[
\|(\widetilde A^\top \widetilde A)^{-1}\|_\infty \le \frac{1}{\widetilde{c}},
\qquad
\|\widetilde A\|_\infty \le C_{\widetilde A},
\]
with constants independent of $h$. Now
\[
A^\top A = D \widetilde A^\top \widetilde A D,
\]
so
\[
(A^\top A)^{-1} A^\top = D^{-1} (\widetilde A^\top \widetilde A)^{-1} \widetilde A^\top.
\]
Hence
\[
\widehat a - a^\star = (A^\top A)^{-1} A^\top (r + \delta) = D^{-1} (\widetilde A^\top \widetilde A)^{-1} \widetilde A^\top (r + \delta).
\]
Combining the property of $\widetilde A$ with the estimates \eqref{eq1208_888}, \eqref{eq1208_999}, we have
\begin{equation}
[(\widetilde A^\top \widetilde A)^{-1} \widetilde A^\top (r + \delta)]_j = \mathcal{O}(h^2), \quad \forall j = 1, ..., 6,
\end{equation}
which implies that
\begin{equation} \label{eq1208_8888}
\begin{aligned}
\widehat a - a^\star & = D^{-1} (\widetilde A^\top \widetilde A)^{-1}
\widetilde A^\top (r + \delta) \\
& = \text{diag}(1, h^{-1}, h^{-1}, h^{-2}, h^{-2},h^{-2}) [(\widetilde A^\top \widetilde A)^{-1} \widetilde A^\top (r + \delta)] \\
& = (\mathcal{O}(h^2), \mathcal{O}(h), \mathcal{O}(h), \mathcal{O}(1), \mathcal{O}(1), \mathcal{O}(1))^\top.
\end{aligned}
\end{equation}

\medskip

\noindent\textbf{5. Velocity error.}
Because the gradient at the origin is exactly the pair $(a_1,a_2)$,
\[
\bigl|\nabla p(0) - \nabla p_{h,\mathrm{rec}}(0)\bigr|
= |(a_1,a_2) - (\widehat a_1,\widehat a_2)|
= \mathcal{O}(h).
\]

Finally, since $v = -\nabla p\cdot\mathbf{n}$ and $v_h = -\nabla p_{h,\mathrm{rec}}\cdot\mathbf{n}$ (the normal $\mathbf{n}$ is assumed known to the same or higher order), we have
\[
|v(\mathbf{X}) - v_h(\mathbf{X})| = \mathcal{O}(h),
\]
which completes the proof.
\end{proof}

\begin{remark} \label{rmk33}
\begin{itemize}
\item The geometric condition in item~(iii) of Proposition~\ref{prop11} can be 
numerically validated. In a typical experiment, one fixes a reconstruction stencil $\{\mathbf y_j\}_{j=1}^{N_s}$ with $|\mathbf{y}_i - \mathbf{y}_j| = \mathcal{O}(1),\ i, j = 1, ..., N_S, i \neq j$, and samples $\mathbf X$. For each sample point $\mathbf X$, one forms the normalized coordinates $\widetilde{\mathbf y}_j = (\mathbf y_j - \mathbf X)$ and constructs the scaled matrix $\widetilde A$.  The smallest singular value 
$\sigma_{\min}(\widetilde A)$ or, equivalently, the smallest eigenvalue of 
$\widetilde A^\top \widetilde A$, is then recorded. Figure~\ref{fig1209} illustrates the behavior of several commonly used reconstruction stencils, where the blue circle indicates the sampling region for the point $\mathbf X$. One observes that, for each stencil, the quantity $\sigma_{\min}(\widetilde A)$ does not approach zero even when $\mathbf X$ moves. This confirms that the scaled design matrix retains a uniform, $h$–independent lower bound on its smallest singular value, in precisely the sense required by item~(iii) of Proposition~\ref{prop11}. In fact, we have tested a larger collection of stencils, and all results consistently satisfy the geometric nondegeneracy requirement stated in item~(iii) of Proposition \ref{prop11}.
\begin{figure}[htbp]
\centering
\begin{subfigure}{0.23\textwidth}
\centering
\includegraphics[width=\linewidth]{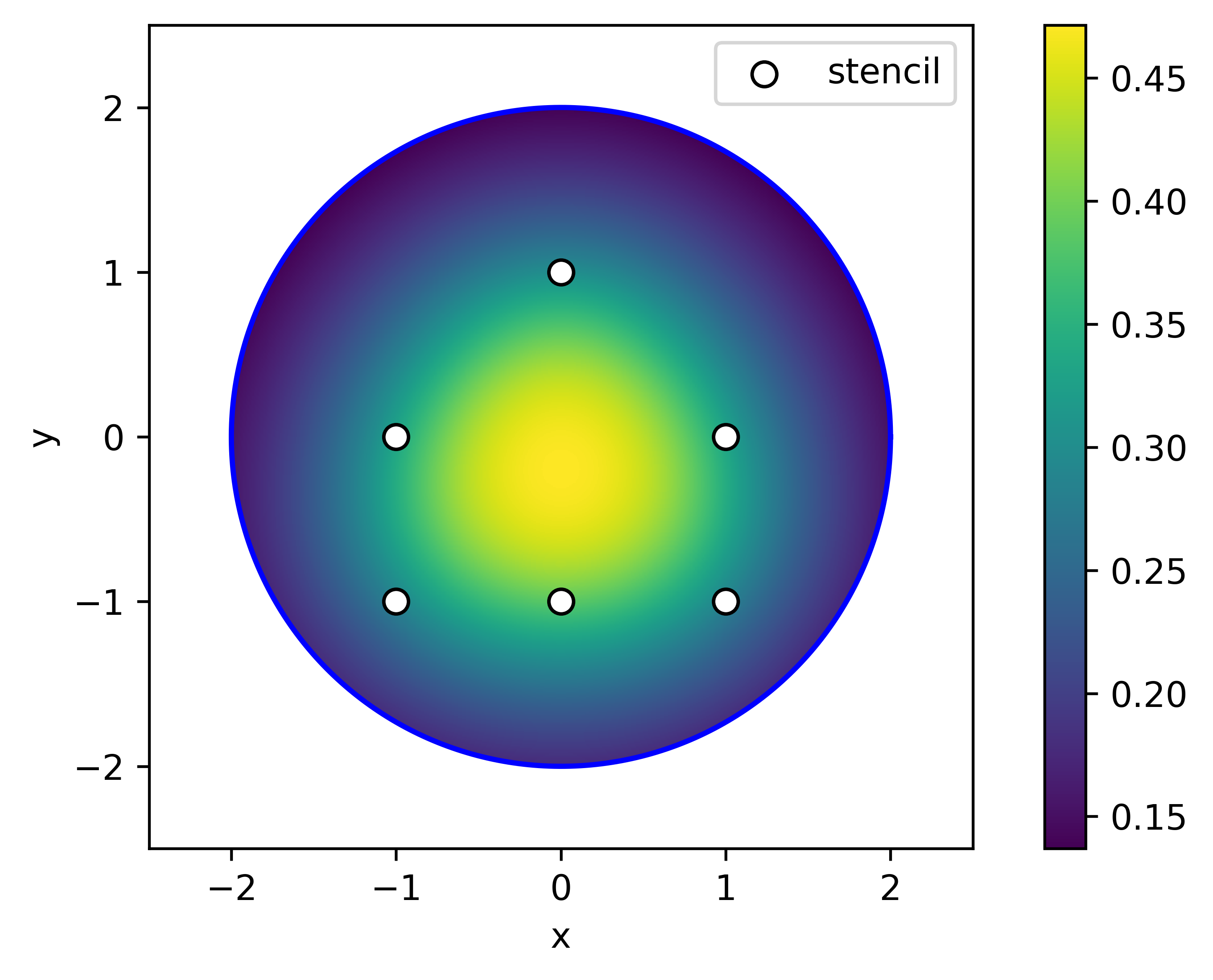}
\end{subfigure}
\begin{subfigure}{0.23\textwidth}
\centering
\includegraphics[width=\linewidth]{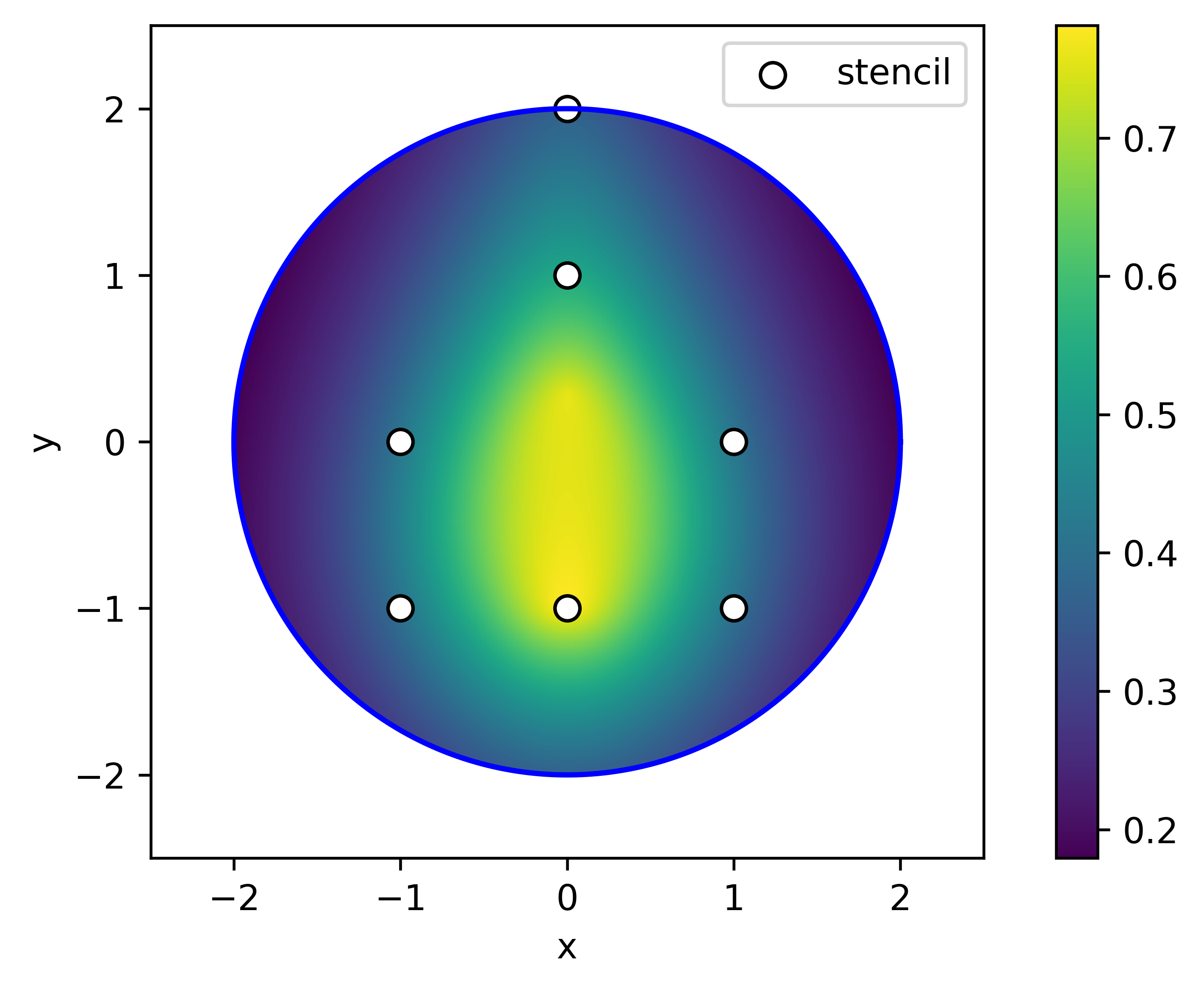}
\end{subfigure}
\begin{subfigure}{0.23\textwidth}
\centering
\includegraphics[width=\linewidth]{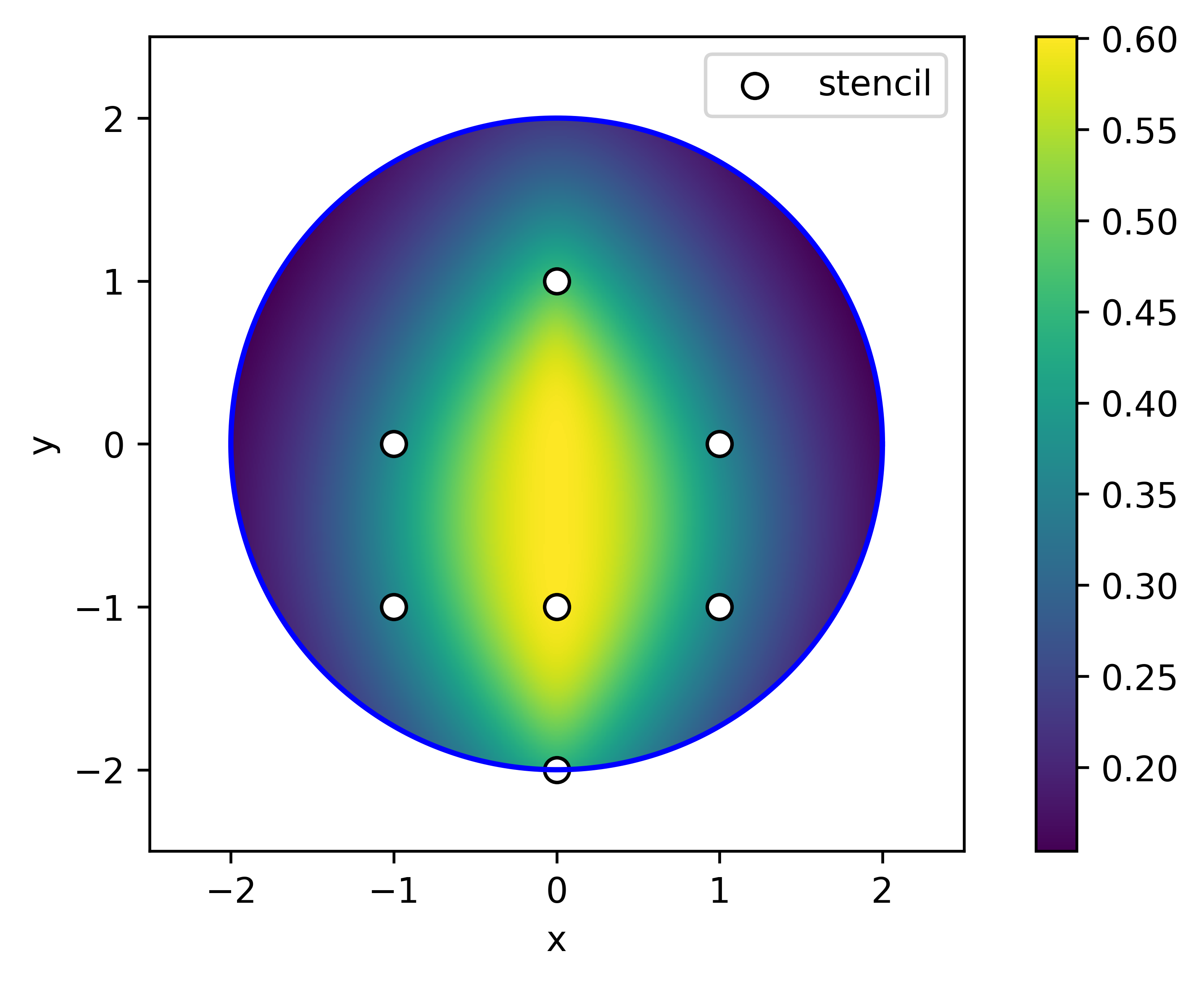}
\end{subfigure}
\begin{subfigure}{0.23\textwidth}
\centering
\includegraphics[width=\linewidth]{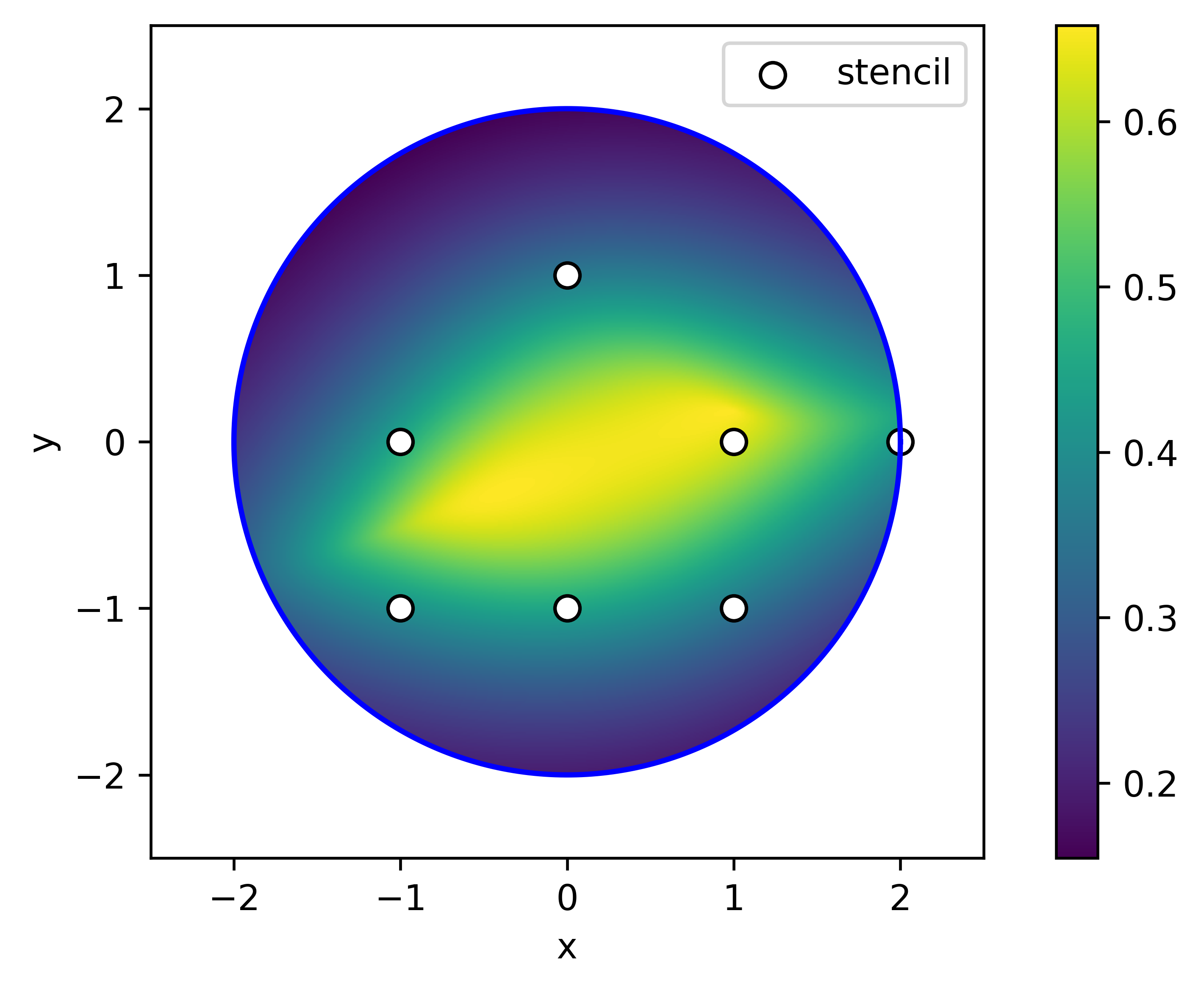}
\end{subfigure}
\begin{subfigure}{0.23\textwidth}
\centering
\includegraphics[width=\linewidth]{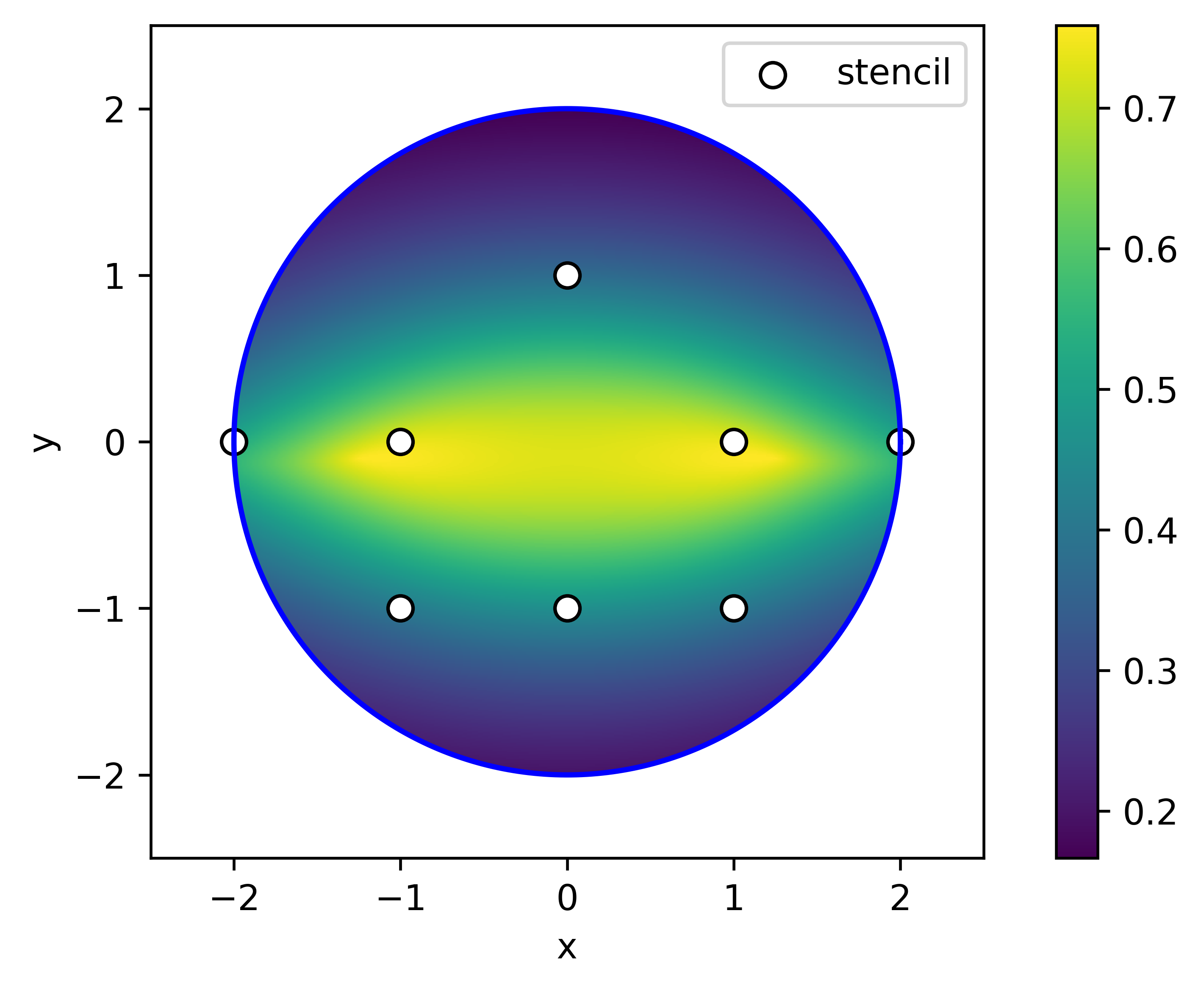}
\end{subfigure}
\begin{subfigure}{0.23\textwidth}
\centering
\includegraphics[width=\linewidth]{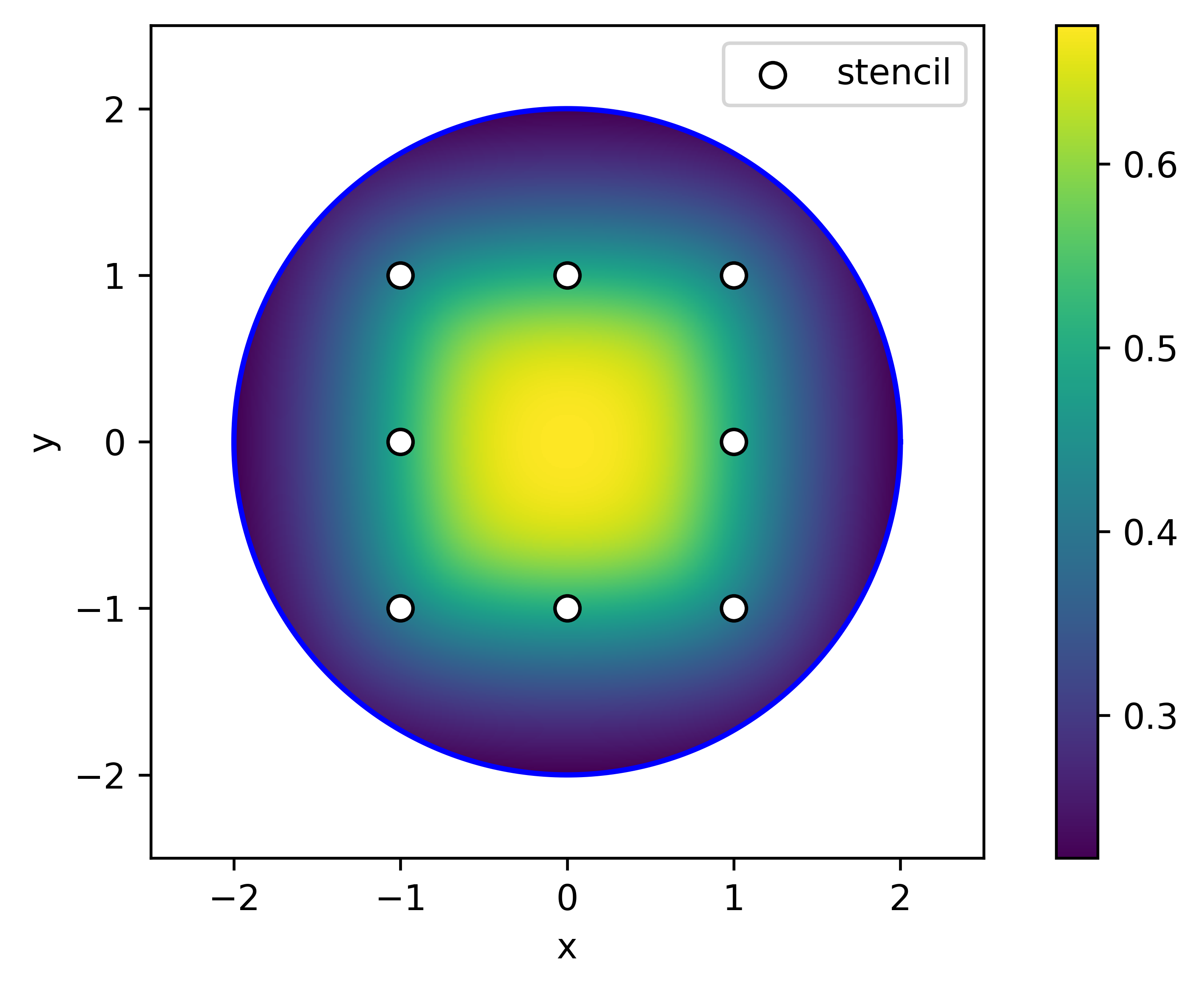}
\end{subfigure}
\begin{subfigure}{0.23\textwidth}
\centering
\includegraphics[width=\linewidth]{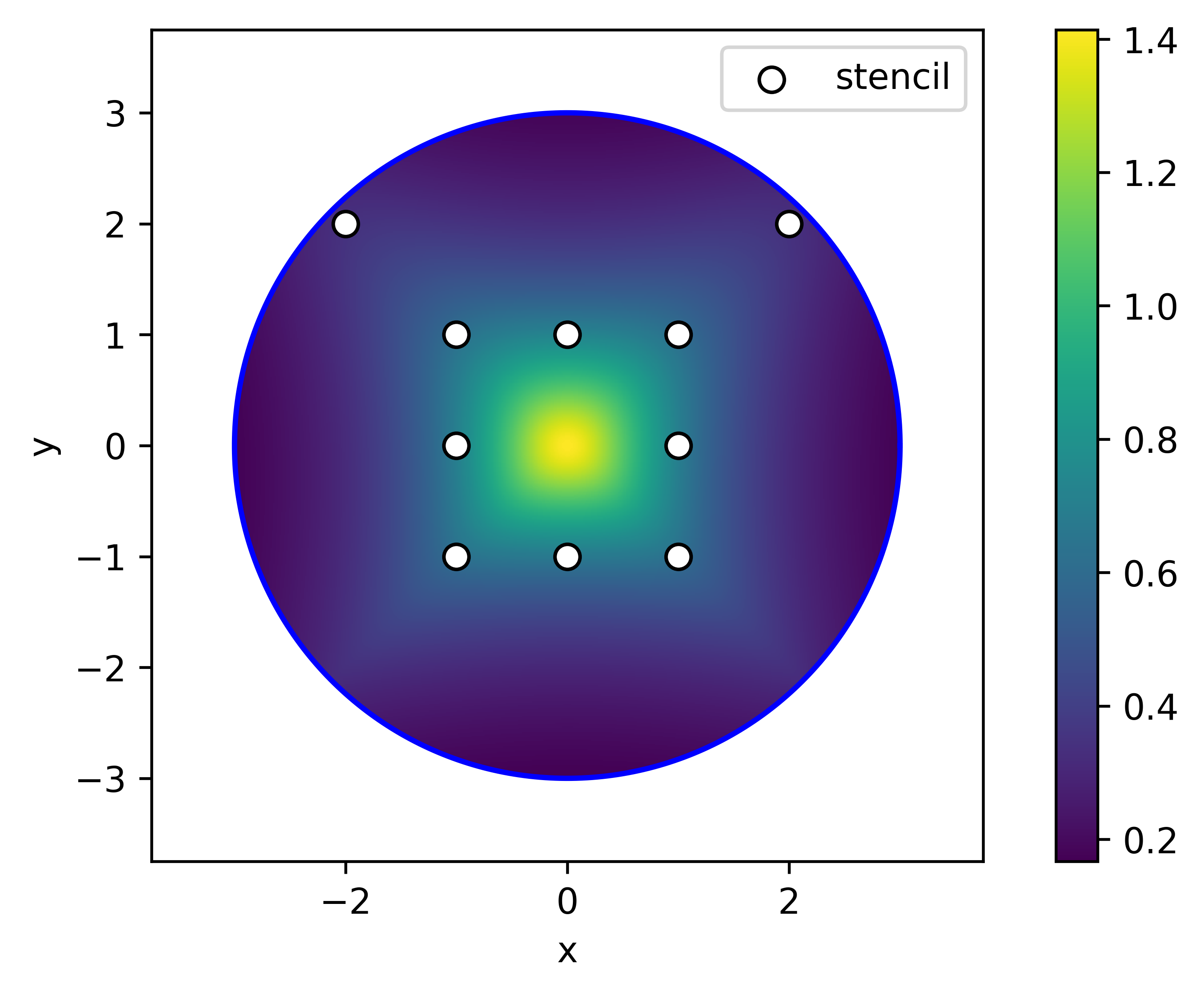}
\end{subfigure}
\begin{subfigure}{0.23\textwidth}
\centering
\includegraphics[width=\linewidth]{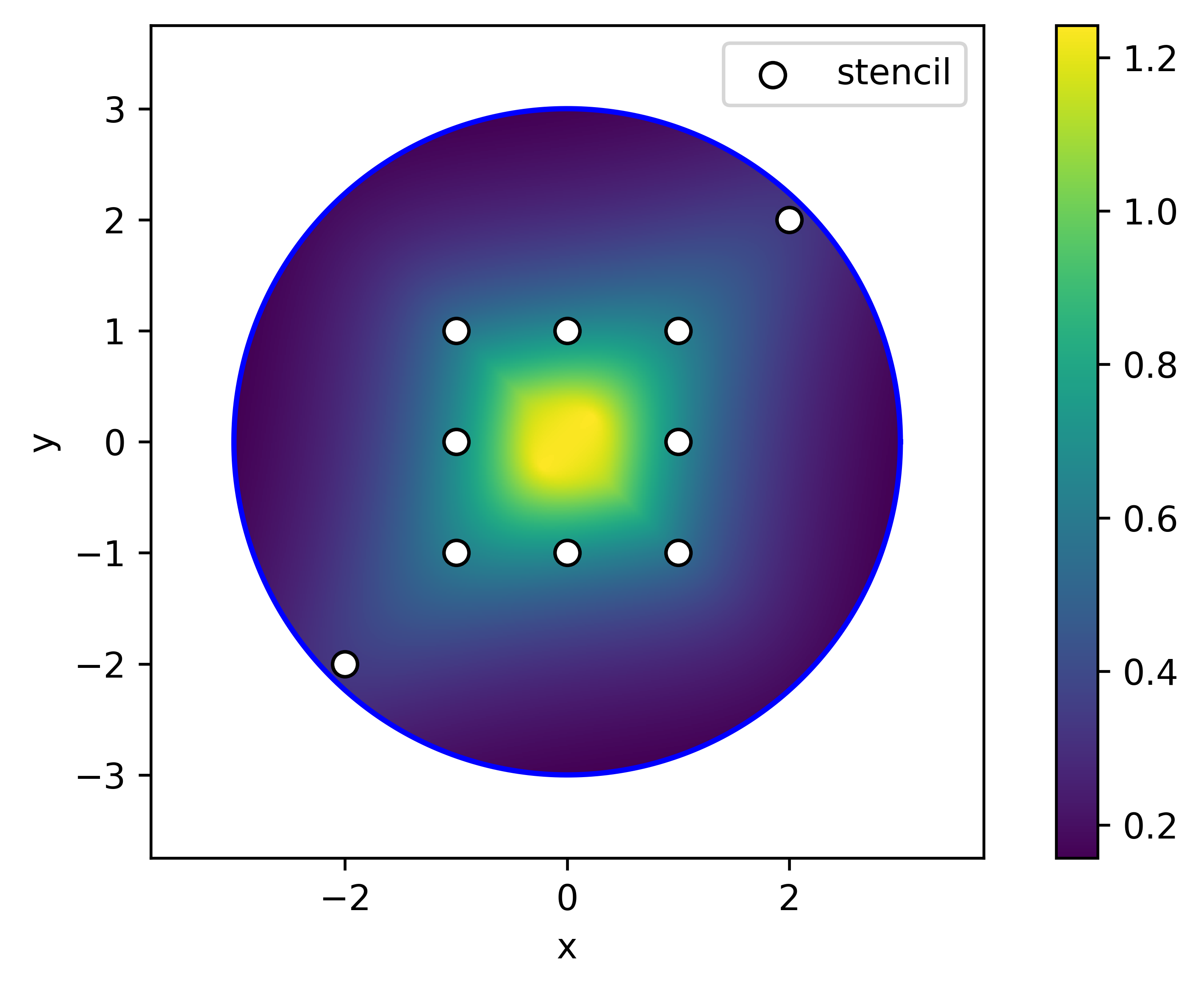}
\end{subfigure}
\caption{Visualization of $\sigma_{\min}(\widetilde A)$ of several reconstruction stencils.}
\label{fig1209}
\end{figure}

\item The constant $C$ in Proposition~\ref{prop11} is a local constant: it may in principle depend on the boundary point $\mathbf{X}$ through the local geometry of $\partial\Omega$. However, if the boundary $\partial\Omega$ is smooth and has no  regions where the curvature varies too rapidly, and if $p$ is
$C^3$ in a fixed tubular neighborhood of $\partial\Omega$, then all such local quantities remain uniformly bounded along the boundary. Under these mild and commonly satisfied assumptions, the constants appearing in the proof admit a global upper bound independent of the choice of $\mathbf{X} \in \partial\Omega$.  Consequently, the error estimate in Proposition \ref{prop11} holds uniformly for all boundary points, that is,
\[
||\nabla p - \nabla p_{h,\mathrm{rec}}||_\infty = \mathcal{O}(h), \quad
||v - v_h||_\infty = \mathcal{O}(h).
\]
\end{itemize}
\end{remark}

\section{Analytical radially symmetric solution} \label{appendix:exact_R0R1}
In this section, we derive the analytical radially symmetric solution for all stages of the model \eqref{eq101010} coupled with \eqref{eqn:nutrient_model} (with $G(c)$ given by \eqref{eqn:G2}), which serves as a benchmark for validating our numerical method. 
Under the radially symmetric assumption, we adopt polar coordinates and assume that all quantities depend only on the radial variable $r$. In this case, the Laplacian operator reduces to
\[
-\Delta u(r)
= -\frac{1}{r}\frac{\partial}{\partial r}
\left( r \frac{\partial u}{\partial r} \right).
\]

\subsection{Before formation of the necrotic core}
Prior to necrotic core formation, the tumor possesses only an outer boundary $\Gamma_1(t)$, and the solution of \eqref{eq101010}--\eqref{eqn:nutrient_model} remains in the form of a patch solution. Specifically, the cell density evolves as a characteristic function
\[
\rho(x,t)=\chi_{D(t)}(x),
\]
where the tumor region $D(t)=\{\rho>0\}$ coincides with the saturated region $S(t)=\{\rho=1\}$.

In the radially symmetric setting, the outer boundary is given by
\[
\Gamma_1(t)=\partial \mathcal{B}_{R_1(t)},
\]
and hence
\[
D(t)=S(t)=\mathcal{B}_{R_1(t)},
\]
where $\mathcal{B}_R$ denotes the disk centered at the origin with radius $R$. The evolution of the tumor radius $R_1(t)$ is governed by an ordinary differential equation; see \eqref{eqn:boundaryspeed1} below.

In this regime, for each fixed time $t$, the solution of \eqref{eq101010}--\eqref{eqn:nutrient_model} coincides with that of the following PDE system:
\begin{subequations}
\begin{alignat}{2}
\label{eqn:pressure_stage1}
- \frac{1}{r}\frac{\partial}{\partial r }
\left( r \frac{\partial p}{\partial r } \right)
&= G_0 (c-\bar{c}), 
&\qquad r < R_1(t),\\
- \frac{1}{r}\frac{\partial}{\partial r }
\left( r \frac{\partial c}{\partial r } \right)
+\lambda c
&= 0,
& r < R_1(t),\\
p &= 0,
& r = R_1(t),\\
c &= c_B,
& r = R_1(t).
\end{alignat}
\end{subequations}

Solving the above system yields
\begin{subequations}
\begin{align}
c(r,t)
&= \frac{c_B}{I_0(\sqrt{\lambda}R_1(t))}
I_0(\sqrt{\lambda}r),\\
p(r,t)
&= -\frac{G_0 c_B}{\lambda}
\frac{I_0(\sqrt{\lambda}r)}{I_0(\sqrt{\lambda}R_1(t))}
+ \frac{G_0\bar{c}}{4} r^2
+ \left(
\frac{G_0 c_B}{\lambda}
- \frac{G_0\bar{c}R_1(t)^2}{4}
\right).
\end{align}
\end{subequations}
Here $I_j$, $j=0,1,2,\dots$, denote the modified Bessel functions of the \emph{first kind}.

The evolution of $R_1(t)$ then follows from Darcy’s law:
\begin{equation}
\label{eqn:boundaryspeed1}
\frac{d R_1(t)}{dt}
= -\frac{\partial p}{\partial r}\bigl(R_1(t)\bigr)
= G_0\left(
\frac{c_B}{\sqrt{\lambda}}
\frac{I_1(\sqrt{\lambda}R_1(t))}
{I_0(\sqrt{\lambda}R_1(t))}
- \frac{\bar{c}}{2} R_1(t)
\right).
\end{equation}

According to the behavior of the solution, this growth period can be further divided into two distinct stages.

\paragraph{Stage 1}
When the tumor is sufficiently small, the nutrient concentration satisfies
$c(r,t)>\bar{c}$ for all $r<R_1(t)$. Consequently, the right-hand side of
\eqref{eqn:pressure_stage1} is strictly positive, and the pressure
$p(r,t)$ remains convex throughout the tumor domain. This stage terminates
at time $T_1^*$, when the tumor radius reaches a critical value
$R^*=R_1(T_1^*)$ determined by
\begin{equation}
c(0,T_1^*)
=\frac{c_B}{I_0(\sqrt{\lambda}R^*)}
=\bar{c}.
\end{equation}

\paragraph{Stage 2}
When the tumor radius exceeds $R^*$ but remains below another threshold
$R^{**}$, implicitly defined by \eqref{eqn:R_threshold2}, the nutrient
concentration becomes smaller than $\bar{c}$ near the tumor center. As a
result, the pressure $p(r,t)$ is no longer strictly convex, and $p(0,t)$
decreases over time until it reaches zero. Let $T_2^*$ denote this time and
$R^{**}=R_1(T_2^*)$ the corresponding radius. Then $R^{**}$ satisfies
\begin{equation}
\label{eqn:R_threshold2}
p(0,T_2^{*})=-\frac{G_0 c_B}{\lambda I_0(\sqrt{\lambda}R^{**})}
+ \left(
\frac{G_0 c_B}{\lambda}
- \frac{G_0\bar{c}(R^{**})^2}{4}
\right)
= 0.
\end{equation}

\subsection{After formation of the necrotic core}
\label{sec: After formation of the necrotic core}

When the tumor radius exceeds $R^{**}$, due to the obstacle constraint
($p\ge 0$), the pressure at the center $p(0,t)$ can no longer decrease.
As a consequence, a necrotic core starts to form. As before, we denote
the necrotic region by $N(t)$ and its boundary by $\Gamma_0(t)$. In the
radially symmetric setting, the necrotic core also takes the form of a
disk, and we denote its radius by $R_0(t)$. Accordingly,
\begin{subequations}
\begin{alignat}{2}
&N(t)=\{\,0\le r<R_0(t)\,\}, \qquad
S(t)=\{\,R_0(t)<r<R_1(t)\,\}, \\
&\Gamma_0(t)=\partial\mathcal{B}_{R_0(t)}, \qquad\qquad \qquad
\Gamma_1(t)=\partial\mathcal{B}_{R_1(t)}.
\end{alignat}
\end{subequations}

We emphasize that at each time $t$, the radius of the necrotic core
$R_0(t)$ is determined implicitly through an equation of the form
$F(R_0,R_1)=0$; see \eqref{eq_exact_10} below.

With the above setting, the solution of
\eqref{eq101010}--\eqref{eqn:nutrient_model} coincides with that of the PDE system
\eqref{eqn:PDE-OP}--\eqref{eq0916_3}. For each fixed time $t$,
the nutrient concentration $c(r,t)$ can be solved explicitly from
\eqref{eq0916_3}. In particular, the two equations in \eqref{eq0916_3}
imply that $c(r,t)$ admits the following form:
\begin{equation}
\label{eq_exact_1}
c(r,t)=
\begin{cases}
a_0(t)\, I_0(\sqrt{\lambda n_c}\, r), 
& 0<r<R_0(t),\\[2mm]
b_0(t)\, I_0(\sqrt{\lambda}\, r)
+ b_1(t)\, K_0(\sqrt{\lambda}\, r),
& R_0(t)<r<R_1(t),\\[2mm]
c_B, & r=R_1(t),
\end{cases}
\end{equation}
where $I_j$ and $K_j$, $j=0,1,2,\dots$, denote the modified Bessel
functions of the first and second kinds, respectively. The coefficients
$a_0(t)$, $b_0(t)$, and $b_1(t)$ are to be determined.

The interface and boundary conditions in \eqref{eq0916_3} yield
\begin{subequations}
\begin{alignat}{2}
c(R_0^-)=c(R_0^+)
&\ \Longrightarrow\
a_0 I_0(\sqrt{\lambda n_c}R_0)
= b_0 I_0(\sqrt{\lambda}R_0)
+ b_1 K_0(\sqrt{\lambda}R_0),
\label{eq_exact_2}\\
c_r(R_0^-)=c_r(R_0^+)
&\ \Longrightarrow\
a_0 \sqrt{n_c}\, I_1(\sqrt{\lambda n_c}R_0)
= b_0 I_1(\sqrt{\lambda}R_0)
- b_1 K_1(\sqrt{\lambda}R_0),
\label{eq_exact_3}\\
c(R_1)=c_B
&\ \Longrightarrow\
b_0 I_0(\sqrt{\lambda}R_1)
+ b_1 K_0(\sqrt{\lambda}R_1)
= c_B.
\label{eq_exact_4}
\end{alignat}
\end{subequations}

Solving the above system, the coefficients $a_0(t)$, $b_0(t)$, and
$b_1(t)$ can be expressed in terms of $R_0(t)$ and $R_1(t)$ as
\begin{subequations}
\label{eqn:a0-b1}
\begin{align}
a_0(t)
&=\frac{b_0(t) I_0(\sqrt{\lambda}R_0(t))
+ b_1(t) K_0(\sqrt{\lambda}R_0(t))}
{I_0(\sqrt{n_c\lambda}R_0(t))},\\
b_0(t)
&=\frac{c_B}
{I_0(\sqrt{\lambda}R_1(t))
+ L(\lambda,n_c,R_0(t))K_0(\sqrt{\lambda}R_1(t))},\\
b_1(t)
&=b_0(t)\,L(\lambda,n_c,R_0(t)),
\end{align}
\end{subequations}
where, 
\begin{equation*}
L(\lambda,n_c,R_0(t))
=\frac{
I_0(\sqrt{n_c\lambda}R_0(t))I_1(\sqrt{\lambda}R_0(t))
-\sqrt{n_c}\,I_1(\sqrt{n_c\lambda}R_0(t))I_0(\sqrt{\lambda}R_0(t))}
{
I_0(\sqrt{n_c\lambda}R_0(t))K_1(\sqrt{\lambda}R_0(t))
+\sqrt{n_c}\,I_1(\sqrt{n_c\lambda}R_0(t))K_0(\sqrt{\lambda}R_0(t))
}.
\end{equation*}

Having determined the nutrient concentration, we next consider the
pressure. Using the expression of $c(r,t)$ in the region $S(t)$,
the first equation in \eqref{eqn:PDE-OP} yields 
\begin{equation}
\label{eq_exact_5}
p(r,t)
= -\frac{G_0}{\lambda}
\bigl(b_0 I_0(\sqrt{\lambda}r)
+ b_1 K_0(\sqrt{\lambda}r)\bigr)
+ \frac{G_0\bar{c}}{4} r^2
+ G_0 A(t)\ln r
+ G_0 B(t),
\end{equation}
with $R_0<r<R_1$ and $A(t)$ and $B(t)$ are unknown coefficients.

The boundary conditions in \eqref{eqn:PDE-OP} imply (we omit the viable $t$ below for notation simplicity)
\begin{subequations}
\label{eqn:ABR0}
\begin{align}
\partial_r p(R_0)=0
&\ \Longrightarrow\
-\frac{1}{\sqrt{\lambda}}
\bigl(b_0 I_1(\sqrt{\lambda}R_0)
- b_1 K_1(\sqrt{\lambda}R_0)\bigr)
+ \frac{R_0}{2}\bar{c}
+ \frac{A}{R_0}=0,
\label{eq_exact_6}\\
p(R_0)=0
&\ \Longrightarrow\
-\frac{1}{\lambda}
\bigl(b_0 I_0(\sqrt{\lambda}R_0)
+ b_1 K_0(\sqrt{\lambda}R_0)\bigr)
+ \frac{R_0^2}{4}\bar{c}
+ A\ln R_0 + B = 0,
\label{eq_exact_7}\\
p(R_1)=0
&\ \Longrightarrow\
-\frac{1}{\lambda}
\bigl(b_0 I_0(\sqrt{\lambda}R_1)
+ b_1 K_0(\sqrt{\lambda}R_1)\bigr)
+ \frac{R_1^2}{4}\bar{c}
+ A\ln R_1 + B = 0.
\label{eq_exact_8}
\end{align}
\end{subequations}

From \eqref{eq_exact_6}--\eqref{eq_exact_7}, one can again solve $A(t)$ and $B(t)$ in terms of $R_0(t)$ and $R_1(t)$
\begin{subequations}
\label{eqn:AB}
\begin{align}
A(t)
&=\frac{R_0(t)}{\sqrt{\lambda}}
\bigl(b_0(t)I_1(\sqrt{\lambda}R_0(t))
- b_1(t)K_1(\sqrt{\lambda}R_0(t))\bigr)
- \frac{R_0^2(t)}{2}\bar{c},\\
B(t)
&=\frac{1}{\lambda}
\bigl(b_0(t)I_0(\sqrt{\lambda}R_0(t))
+ b_1(t)K_0(\sqrt{\lambda}R_0(t))\bigr)
- \frac{R_0^2(t)}{4}\bar{c}
- A(t)\ln R_0(t).
\end{align}
\end{subequations}

Substituting the above expressions into \eqref{eq_exact_8} yields the
transcendental relation between $R_0(t)$ and $R_1(t)$, i.e., provided the position of $R_1(t)$, the core boundary $R_0(t)$ is determined implicitly via the equation
\begin{equation}
\label{eq_exact_10}
F(R_0,R_1)
:= \frac{R_1^2}{4}\bar{c}
+ A\ln R_1 + B
- \frac{1}{\lambda}
\bigl(
b_0 I_0(\sqrt{\lambda}R_1)
+ b_1 K_0(\sqrt{\lambda}R_1)
\bigr)
=0.
\end{equation}

Finally, the evolution of the outer boundary is governed by
\begin{equation}
\label{eq_exact_9}
\frac{dR_1(t)}{dt}
= -\frac{\partial p}{\partial r}(R_1(t))
= G_0\left[
\frac{1}{\sqrt{\lambda}}
\bigl(b_0 I_1(\sqrt{\lambda}R_1)
- b_1 K_1(\sqrt{\lambda}R_1)\bigr)
- \frac{R_1}{2}\bar{c}
- \frac{A}{R_1}
\right].
\end{equation}

By this point, the dynamics of the radial symmetric solution is completely determined.

\bibliographystyle{elsarticle-num} 
\bibliography{cas-refs}







\end{document}